\documentclass[reqno]{amsart}

\usepackage{amsmath,amsfonts,amssymb,amsthm}
\usepackage{calligra,mathrsfs,upgreek,bbm,mathtools}
\usepackage[mathscr]{eucal}

\usepackage{lmodern}
\usepackage[usenames]{color}
\usepackage[colorlinks=true, linkcolor=blue, citecolor=purple, linktoc=page]{hyperref}

\usepackage{graphicx}
\usepackage{subcaption}
\usepackage{tikz}
\usetikzlibrary{patterns}
\usetikzlibrary{arrows}
\tikzset{>=stealth}
\def\gr{1.618034} %

\usepackage[style=alphabetic]{biblatex}
\parindent=\normalparindent 
\usepackage{enumitem}
\setlist{topsep=-2pt, itemsep=4pt}
\usepackage{tabu}

\def\textclap#1{\hbox to 0pt{\hss#1\hss}}
\def\clap{\mathpalette\clapinternal}
\def\clapinternal#1#2{\textclap{$\mathsurround=0pt#1{#2}$}}

\usepackage{colonequals}
\newcommand*\cocolon{\nobreak\mskip6mu plus1mu\mathpunct{}\nonscript\mkern-\thinmuskip{:}\mskip2mu\relax}

\input{headers/xdasharrows.tex}
\theoremstyle{plain}
\newtheorem{theorem}{Theorem}[section]
\newtheorem{lemma}[theorem]{Lemma}
\newtheorem{prop}[theorem]{Proposition}
\newtheorem{corollary}[theorem]{Corollary}

\theoremstyle{definition}
\newtheorem{setup}[theorem]{Setup}
\newtheorem{define}[theorem]{Definition}
\newtheorem{remark}[theorem]{Remark}

\theoremstyle{plain}
\newtheorem{introthm}{Theorem}

\makeatletter
\def\thm@space@setup{%
  \thm@preskip=\parskip
  \thm@postskip=0pt
}
\makeatother

\makeatletter
\renewenvironment{proof}[1][\proofname]{
  \par
  \vspace{-\parskip}
  \pushQED{\qed}
  \normalfont
  \topsep5pt \partopsep0pt
  \trivlist
  \item[\hskip\labelsep
        \itshape
    #1\@addpunct{.}]\ignorespaces
}{%
  \popQED\endtrivlist\@endpefalse
  \addvspace{6pt plus 6pt} 
}
\makeatother

\newcommand{\N}{\mathbb{N}}
\newcommand{\Z}{\mathbb{Z}}
\newcommand{\Q}{\mathbb{Q}}
\newcommand{\R}{\mathbb{R}}
\newcommand{\C}{\mathbb{C}}

\newcommand{\into}{\hookrightarrow}
\newcommand{\onto}{\twoheadrightarrow}

\newcommand{\<}{\langle}
\renewcommand{\>}{\rangle}

\newcommand{\half}{{\frac 12}}
\newcommand{\mhalf}{{-\kern-.5pt\frac{1}{2}}}

\newcommand{\del}{\partial}

\DeclareMathOperator{\Spec}{Spec}

\newcommand{\pt}{\mathrm{pt}}

\DeclareMathOperator{\Bl}{Bl}

\newcommand{\con}{\mathrm{con}}

\newcommand{\gitquot}{/\kern-3pt/}

\DeclareMathOperator{\Sym}{Sym}

\newcommand{\A}{\mathbb{A}}
\newcommand{\G}{\mathbb{G}}
\renewcommand{\P}{\mathbb{P}}

\newcommand{\curve}{\mathrm{C}}

\renewcommand{\O}{\mathcal{O}}

\DeclareMathOperator{\rk}{\mathrm{rk}}

\DeclareMathOperator{\Coh}{coh}

\DeclareMathOperator{\refl}{ref}

\DeclareMathOperator{\Per}{Per}
\newcommand{\pPer}{{}^p{\kern-2pt}\Per}
\newcommand{\zPer}{{}^0{\kern-2pt}\Per}
\newcommand{\mPer}{{}^{-1}{\kern-2.5pt}\Per}

\newcommand{\St}{\mathrm{St}}

\DeclareMathOperator{\cone}{cone} %
\DeclareMathOperator{\silt}{silt}  %
\DeclareMathOperator{\tilt}{tilt}  %

\newcommand{\Var}{\mathrm{Var}}

\DeclareMathOperator{\Mot}{Mot}
\renewcommand{\L}{\mathbb{L}}

\newcommand{\muhat}{{\hat\upmu}}
\newcommand{\nilp}{\mathrm{nilp}}

\newcommand{\virt}[1]{\left[#1\right]_{\mathrm{virt}}}
\newcommand{\cvirt}[1]{\left[#1\right]\clap{\kern9pt{}_{\mathrm{virt}}}}

\DeclareMathOperator{\MMHS}{MMHS}
\newcommand{\mmhs}{\mathrm{mmhs}}

\newcommand{\hsp}{\mathrm{hsp}}
\newcommand{\wt}{\mathrm{wt}}

\newcommand{\BPS}{\mathrm{BPS}}
\newcommand{\cBPS}{\mathcal{BPS}}
\newcommand{\GV}{\mathrm{GV}}

\newcommand{\fCrit}{\mathscr{C}\mathrm{rit}\kern1pt}

\newcommand{\fExt}{\fE\kern-1pt\mathrm{xt}}

\DeclareMathOperator{\IC}{IC}
\newcommand{\constQ}{\underline\Q}

\DeclareMathOperator{\pr}{pr}

\DeclareMathOperator{\K}{K}

\let\Hu\H
\let\H\relax
\DeclareMathOperator{\H}{H}

\DeclareMathOperator{\Mat}{Mat}
\DeclareMathOperator{\GL}{GL}

\newcommand{\tr}{\mathrm{tr}}
\newcommand{\loc}{\mathrm{loc}}

\newcommand{\id}{\mathrm{Id}}

\DeclareMathOperator{\D}{D}
\let\d\relax
\DeclareMathOperator{\d}{d}

\newcommand{\ab}{\mathrm{ab}}
\newcommand{\floor}[1]{\left\lfloor #1 \right\rfloor}

\DeclareMathOperator{\Jac}{Jac}
\DeclareMathOperator{\Rep}{Rep}

\newcommand{\cyc}{\mathrm{cyc}}

\renewcommand{\min}{\mathrm{min}}

\newcommand{\ddim}{\underline{\mathrm{dim}}}

\newcommand{\m}{\mathfrak{m}}
\newcommand{\n}{\mathfrak{n}}
\renewcommand{\o}{\mathfrak{o}}

\renewcommand{\k}{\kappa}

\let\mod\relax
\DeclareMathOperator{\mod}{mod}

\DeclareMathOperator{\fdmod}{fdmod}
\DeclareMathOperator{\tw}{tw}

\DeclareMathOperator{\proj}{proj}

\newcommand{\perf}{\mathrm{perf}}

\newcommand{\DGMod}{\cM\mathrm{od}\kern1pt}
\newcommand{\DGPerf}{\cP\mathrm{erf}\kern1pt}

\newcommand{\op}{\mathrm{op}}
\newcommand{\Ob}{\mathrm{Ob}\,}

\DeclareMathOperator{\irr}{Irr}

\DeclareMathOperator{\im}{im}

\DeclareMathOperator{\Aut}{Aut}
\DeclareMathOperator{\Hom}{Hom}
\DeclareMathOperator{\End}{End}
\DeclareMathOperator{\Ext}{Ext}

\newcommand{\RHom}{\mathbf{R}\mathrm{Hom}}
\newcommand{\REnd}{\mathbf{R}\mathrm{End}}
\newcommand{\Lotimes}{\overset{\textbf{L}}\otimes}

\newcommand{\cA}{\mathcal{A}}

\newcommand{\cC}{\mathcal{C}}
\newcommand{\cD}{\mathcal{D}}

\newcommand{\cF}{\mathcal{F}}

\newcommand{\cH}{\mathcal{H}}

\newcommand{\cK}{\mathcal{K}}
\newcommand{\cL}{\mathcal{L}}
\newcommand{\cM}{\mathcal{M}}
\newcommand{\cN}{\mathcal{N}}

\newcommand{\cP}{\mathcal{P}}
\newcommand{\cQ}{\mathcal{Q}}

\newcommand{\cS}{\mathcal{S}}

\newcommand{\cW}{\mathcal{W}}

\newcommand{\bC}{\mathbf{C}}

\newcommand{\bR}{\mathbf{R}}

\newcommand{\fC}{\mathscr{C}}

\newcommand{\fE}{\mathscr{E}}

\newcommand{\fJ}{\mathscr{J}}

\newcommand{\fM}{\mathscr{M}}
\newcommand{\fN}{\mathscr{N}}

\newcommand{\fP}{\mathscr{P}}

\newcommand{\fU}{\mathscr{U}}

\newcommand{\fX}{\mathscr{X}}
\newcommand{\fY}{\mathscr{Y}}
\newcommand{\fZ}{\mathscr{Z}}

\DeclareMathOperator{\HH}{HH}

\newcommand{\dvector}[2]{(#1,#2)}

\newcommand{\Nccr}{\Lambda}

\newcommand{\Mothat}{\mathrm{Mot}^{\muhat}}

\title{Donaldson--Thomas Invariants of Length 2 Flops}
\author{Okke van Garderen}

\begin{document}

\begin{abstract}
  We develop theoretical aspects of refined Donaldson--Thomas theory for threefold flops, and use these to determine all DT invariants for a doubly infinite family of length 2 flopping contractions.
  Our results show that a refined version of the strong-rationality conjecture of Pandharipande--Thomas holds in this setting, and also that refined DT invariants do not classify flops.
  Our main innovation is the application of tilting theory to better understand the stability conditions and cyclic $A_\infty$-deformation theory of these spaces.
  Where possible we work in the motivic setting, but we also compute intermediary refinements, such as mixed Hodge structures.
\end{abstract}

\maketitle

\tableofcontents

\section{Introduction}

Threefold flops are a fundamental class of birational surgeries, given that they connect minimal models in the minimal model program \cite{KM98}.
In this paper we focus on simple flops, where a single rational curve $\curve$ in a smooth threefold $Y$ is contracted to a point:
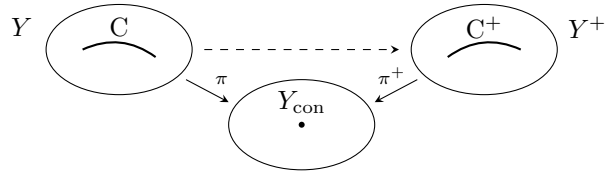
\[\begin{tikzpicture}
    \draw[thick] (-1.8*\gr,1) to[bend left] (-1.2*\gr,.9);
    \draw[thick] (1.8*\gr,1) to[bend right] (1.2*\gr,.9);
    \node[fill,circle,inner sep=.8pt] at (0,0) {};
    \draw (0,0) ellipse (0.6*\gr cm and 0.6cm);
    \draw (-1.5*\gr,1) ellipse (0.6*\gr cm and 0.6cm);
    \draw (1.5*\gr,1) ellipse (0.6*\gr cm and 0.6cm);
    \node at (-2.3*\gr,1.3) {$Y$};
    \node at (-1.5*\gr,1.3) {$\curve$};
    \node at (0,.3) {$Y_\con$};
    \node at (2.35*\gr,1.3) {$Y^+$};
    \node at (1.5*\gr,1.3) {$\curve^+$};
    \draw[->] (-.95*\gr,.6) to[edge label=$\hphantom{{}^+}\scriptstyle\pi$,near start,inner sep=.15pt] (-.6*\gr,.3);
    \draw[->] (.95*\gr,.6) to[edge label'=$\scriptstyle\pi^+$,near start,inner sep=.15pt] (.6*\gr,.3);
    \draw[->,dashed] (-.8*\gr,1) to (.8*\gr,1);
\end{tikzpicture}\]
This innocent diagram is the basis for a rich geometry which is still, remarkably, not completely understood.
Several invariants have been studied: ranging from the \emph{length} invariant $1\leq \ell \leq 6$ of the curve \cite{KM92}, to Gopakumar--Vafa invariants \cite{Katz08,BKL01}, to \emph{Donaldson--Thomas invariants}.
DT invariants are of a motivic nature \cite{KS08}, and considerable work has been expended towards their \emph{refinement}.
Such refined invariants have been computed for only a few examples, which include affine threespace \cite{BBS13}, and other toric varieties \cite{MN15}, while only the most elementary class of flops with length $\ell=1$ has been studied \cite{DM17}. 
The goal of this paper is to develop the DT theory of higher length flops.
There is a jump in complexity, which can already be seen when moving from length $\ell=1$ to $\ell=2$, and hence we primarily focus on the Donaldson--Thomas theory for flops of length two.

We work in a noncommutative setting, representing flopping contractions $Y\to Y_\con$ as the Jacobi algebra of a quiver with potential $(Q,W)$ via a derived equivalence
\[
  \D^b(\Coh Y) \simeq \D^b(\mod \Jac(Q,W)).
\]
In this noncommutative setting, the Donald\-son--Thomas theory is captured by a \emph{partition function} $\Upphi(t)$, a powerseries indexed by dimension vectors $\updelta \in \N Q_0$ with coefficients given by motivic classes in a certain Grothen\-dieck ring of varieties equipped with monodromy. 
For a fixed dimension vector $\updelta$ the motivic class acts as a "virtual count" of the nilpotent $\Jac(Q,W)$-modules of dimension $\updelta$. 
We present this partition function as a plethystic exponential
\[
  \Upphi(t) = \Sym\left(\sum_{\updelta\in \N Q_0} \frac{\BPS_\updelta}{\L^\half-\L^\mhalf} \cdot t^\updelta\right),
\]
parametrised by \emph{BPS invariants}, and our aim is to describe these explicitly.

\subsection{Main result}

We are able to explicitly calculate the invariants for a new infinite family $\{Y_{a,b}\to \Spec R_{a,b}\}$ of length 2 flopping contractions parametrised by pairs $(a,b)$ where $a\in\N$, and $b \in \N\cup \{\infty\}$. This family was recently and independently constructed by Kawamata \cite{Kawamata20}.
Each member of the family can be represented by a quiver with potential of the form given in figure \ref{fig:quivpot}, for which we determine the BPS invariants $\BPS_\updelta$.

\begin{figure}
  \begin{tikzpicture}[>=stealth,inner sep=1pt]
    \clip (-5,1) rectangle (5,-1.4);
    \node at (-2,0) {$Q\colon$};
    \node[outer sep=1pt] (A) at (-.8,0) {$\scriptstyle 0$};
    \node[outer sep=1pt] (B) at (.8,0) {$\scriptstyle 1$};
    \draw[->] (A) to[bend left] node[midway,fill=white]{$c$} (B);
    \draw[->] (B) to[bend left] node[midway,fill=white]{$d$} (A);
    \draw[->] (B) to[loop,looseness=15,in=280,out=-10] node[midway,fill=white]{$x$} (B);
    \draw[->] (B) to[loop,looseness=15,in=80,out=10] node[midway,fill=white]{$y$} (B);  
    \draw[->] (A) to[loop,looseness=10,in=140,out=220] node[midway,fill=white]{$s$} (A);
    \node at (0,-1.2) {$W_{a,b} = x^2y - f_{a,b}(y) + y^2cd - sdc + 2s^a$};
  \end{tikzpicture}
  \caption{The family of quivers with potential.}
  \label{fig:quivpot}
\end{figure}

Across the derived equivalence, a dimension vector $\updelta$ for the quiver corresponds to a K-theory class $\K_0(\curve)$ for the curve and corresponds to a unique pair of rank and Euler characteristic $\rk(\updelta), \chi(\updelta) \in \Z$. 
Our main result is the following theorem which describes the dependence of the BPS invariants on the rank and Euler characteristic. 
Where possible we calculate $\BPS_\updelta$ motivically, and otherwise calculate \emph{realisations} in the Grothendieck ring $\K_0(\MMHS)$ of monodromic mixed Hodge structures.
\begin{introthm}[Theorem \ref{thm:theinvariants}]\label{introthm:A}
  The BPS invariants $\BPS_\updelta$ associated to the length 2 flopping contraction $Y_{a,b} \to \Spec R_{a,b}$ have the following dependence on $\rk$ and $\chi$:
  \begin{itemize}
  \item if $\rk(\updelta) = 0$ then
    \[
      \BPS_\updelta = \L^{-\frac32}[\P^1],
    \]
    where $\L^{\half}$ is a formal square root for the Lefschetz motive $\L = [\A^1]$,

  \item if $\rk(\updelta) = \pm 1$ then
    \[
      \BPS_\updelta = \begin{cases}
        \L^{-1}(1 - [D_{4a}]) + 2   & a \leq b, \\
        \L^{-1}(1 - [D_{2b+1}]) + 3 & a > b,
      \end{cases}
    \]
    where $D_{4a}$ and $D_{2b+1}$ are curves of genus $a$ and $b$, with a monodromy action of $\upmu_{4a}$ and $\upmu_{2b+1}$ respectively.

  \item if $\rk(\updelta) = \pm 2$ and $\chi(\updelta)$ is odd then
    \[
      \BPS_\updelta = \L^\mhalf(1-[\upmu_a]),
    \]
    and if $\chi(\updelta)$ is even the BPS invariant has the realisation
    \[
      \upchi_\mmhs(\BPS_\updelta) = \upchi_\mmhs(\L^\mhalf(1-[\upmu_a])),
    \]
    where $\upchi_\mmhs$ denotes a realisation map into monodromic mixed Hodge structures,

  \item if $|\!\rk(\updelta)| \geq 3$ and $\chi(\updelta)$ is not divisible by $\rk(\updelta)$ then
    \[
      \BPS_\updelta = 0,
    \]
    while for $|\!\rk(\updelta)| \geq 3$ and $\chi(\updelta)$ divisible by $\rk(\updelta)$ the realisation vanishes:
    \[
      \upchi_\mmhs(\BPS_\updelta) = 0.
    \]
  \end{itemize}
\end{introthm}
The theorem shows that (at least after applying the realisation) the DT theory of this family is essentially controlled by three invariants: a rank $0$ count of the points on $C\simeq \P^1$, and invariants for each rank smaller than the length $\ell$. 
The latter invariants are a refinement of the genus 0 Gopakumar--Vafa invariants of curve-classes in $\H_2(\curve,\Z)$, which one expects to only depend on the rank: this is equivalent to the
strong rationality conjecture of Pandharipande and Thomas \cite{PT09}.
Theorem \ref{introthm:A} shows that the refined version of this conjecture, as described in \cite{Davison19}, holds in our setting.
\begin{corollary}
  For the family $\{Y_{a,b}\}$, the refined strong rationality conjecture holds in the $\K_0(\MMHS)$-realisation.
\end{corollary}
For every $a>1$ the flopping contractions associated to the pairs $(a,a)$, $(a,a+1)$, $\ldots$, $(a,2a-1)$, and $(a,\infty)$ are analytically distinct, see e.g. Kawamata \cite{Kawamata20}, but our theorem shows that the refined BPS invariants for these pairs are nonetheless the same.
\begin{corollary}\label{cor:donotdet}
  $\K_0(\MMHS)$-realisations of BPS invariants do not classify flops.
\end{corollary}
This corollary strengthens a result of Brown--Wemyss \cite{BW18}: they showed that (numerical) GV invariants do not determine flops.
It also puts their result in a wider context, as the two examples they use form a subset of our family.
As in \cite{BW18} we also compare with the noncommutative contraction algebra invariant of \cite{DW16}, which \emph{does} separate the flops.
Corollary \ref{cor:donotdet} suggests that, even at this level of refinement, some essential aspect of the noncommutative deformation theory is lost 
in the calculation of DT invariants.

To prove our main result we introduce two new techniques, which leverage the powerful theory of tilting equivalences of noncommutative crepant resolutions \cite{HW19,DW19}. 
The first method yields a classification of stable objects for a certain stability condition, showing that these objects are related by tilting functors. 
The second method shows that the tilts preserve the associated Calabi--Yau potentials of these objects, which determine their contribution to the DT theory.
As a result we are able to compute the BPS invariants by considering the deformation theory of just three types of objects.

\subsection{Stability \& tilting theory}

\begin{figure}
  \begin{subfigure}{.5\textwidth}
    \centering
    \vfill
    \begin{tikzpicture}[scale=.5]
      \node at (0,5) {${\scriptstyle\K_0(\proj \Nccr)_\R}$};
      \begin{scope}[thick]
        \clip (-4,-4) rectangle (4,4);
        \draw[color=blue!60] (0,6) to (0,-6);
        \foreach \i in {0,1,2,3,4,5}
        {
          \draw[color=blue!60]  (16*\i,   -8*\i+4) to (-16*\i , 8*\i-4);  %
          \draw[color=red!60]   (12*\i+6, -6*\i) to (-12*\i-6, 6*\i);  %
        }
        \foreach \i in {6,...,90}
        {
          \draw[color=blue!60]   (4*\i,   -2*\i+1) to (-4*\i , 2*\i-1);  %
          \draw[color=red!60]   (2*\i+1, -\i) to (-2*\i-1, \i);  %
        }
        \foreach \i in {90,...,6}
        {
          \draw[color=blue!60]   (4*\i,   -2*\i-1) to (-4*\i , 2*\i+1);  %
          \draw[color=red!60]   (2*\i-1, -\i) to (-2*\i+1, \i);  %
        }
        \foreach \i in {5,...,0}
        {
          \draw[color=blue!60]   (16*\i,   -8*\i-4) to (-16*\i , 8*\i+4);  %
          \draw[color=red!60]   (12*\i-6, -6*\i) to (-12*\i+6 ,6*\i);  %
        }
        \draw[color=green] (12, -6) to (-12,6);  %
      \end{scope}
      \node[circle,fill=black,draw=black,inner sep=.8pt] at (0,0) {};
      \node at (-.75,-1) {$\scriptstyle -[P_0]$};
      \node at (1,.4) {$\scriptstyle [P_1]$};
    \end{tikzpicture}
    \caption{g-vectors of tilting complexes}    
    \label{subfig:A}
  \end{subfigure}
  \begin{subfigure}{.48\textwidth}
    \centering
    \begin{tikzpicture}
      \node at (2,4.5) {${\scriptstyle \N Q_0}$};
      \begin{scope}[thick]
        \clip (0,0) rectangle (4,4);
        \foreach \i in {60,...,4}
        {
          \draw[color=red!60]  (0,0) to (\i+1, 2*\i+1);
          \draw[color=blue!60]  (0,0) to (2*\i+1,4*\i);
        }
        \foreach \i in {3,...,0}
        {
          \draw[color=red!60]  (0,0) to (4*\i+4, 8*\i+4);
          \draw[color=blue!60]  (0,0) to (8*\i+4, 16*\i);
        }
        \foreach \i in {0,...,3}
        {
          \draw[color=red!60]  (0,0) to (4*\i, 8*\i + 4);
          \draw[color=blue!60]  (0,0) to (8*\i+4, 16*\i+16);
        } 
        \foreach \i in {4,...,60}
        {
          \draw[color=red!60]  (0,0) to (\i, 2*\i + 1);
          \draw[color=blue!60]  (0,0) to (2*\i+1, 4*\i+4);
        } 
        \draw[color=green] (0,0) to (4,8);  %
      \end{scope}
      \begin{scope}[fill=white,draw=black,inner sep=.8pt]
        \clip (-.6,-.6) rectangle (4.2,4.2);
        \node at (-.35,.5) {$\scriptstyle [S_1]$};
        \node at (.5,-.25) {$\scriptstyle [S_0]$};
        \foreach \n in {1,2,3,4,5,6,7,8}
        {
          \foreach \i in {0,1,2}
          {
            \node[circle,fill,draw] at (.5*\n*\i, \n*\i + .5*\n) {};
            \node[circle,fill,draw] at (\n*\i+.5*\n, 2*\n*\i+2*\n) {};
            \node[circle,fill,draw] at (.5*\n*\i+.5*\n, \n*\i+.5*\n) {};
            \node[circle,fill,draw] at (\n*\i+.5*\n, 2*\n*\i) {};
          }
          \node[circle,fill,draw] at (.5*\n, 1*\n) {};
        }
        \foreach \i in {3,4}
        {
          \node[circle,fill,draw] at (.5*\i, \i + .5) {};
          \node[circle,fill,draw] at (\i+.5, 2*\i+2) {};
          \node[circle,fill,draw] at (.5*\i+.5, \i+.5) {};
          \node[circle,fill,draw] at (\i+.5, 2*\i) {};
        }
        \node[circle,fill=black,draw] at (0,0) {};
      \end{scope}
    \end{tikzpicture}
    \caption{dimension vectors of semi-stables}
    \label{subfig:B}
  \end{subfigure}
  \caption{For a generic stability condition, the dimension vectors of (semi-)stable objects
    are on the rays \ref{subfig:B} perpendicular to the tilting hyperplane arrangement \ref{subfig:A}.
    Each ray is spanned by an indivisible class/dimension vector $\updelta_{\curve,n}$ (red), $\updelta_{2\curve,n}$ (blue) or $\updelta_\pt$ (green).
  }
\label{fig:scattering}
\end{figure}
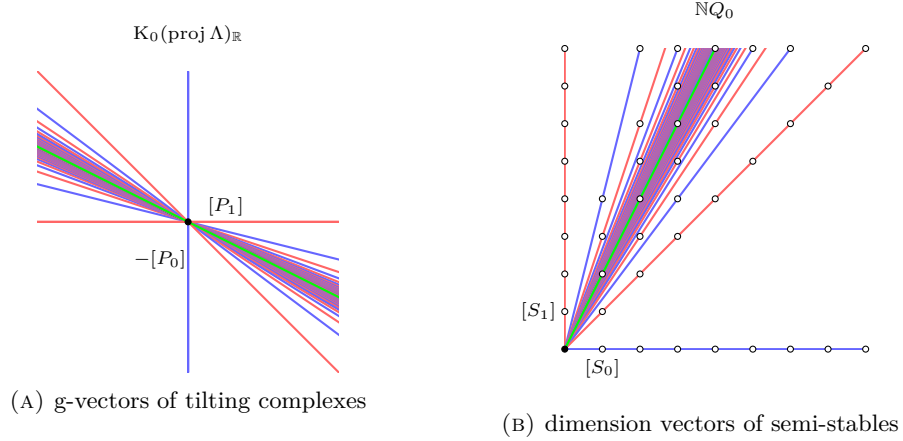

The partition function $\Upphi(t)$ counts the moduli of \emph{nilpotent} modules for the Jacobi algebra of $(Q,W)$, or equivalently finite dimensional modules of the completion $\Nccr \colonequals \widehat\Jac(Q,W)$.
It is well-known that $\Upphi(t)$ can be decomposed by introducing a stability condition: given a central charge,
\[
  Z \colon \K_0(\fdmod \Nccr) \simeq \Z^2 \to \C,
\]
on the category $\fdmod\Nccr\subset\mod\Nccr$ of finite dimensional $\Nccr$-modules, a formula due to Kontsevich--Soibelman shows \cite{KS08} yields a decomposition of $\Upphi(t)$ into a product of partition functions $\Upphi^\uptheta(t)$, which count semistable objects with a prescribed phase $\uptheta$ in the complex plane.
To make use of this construction, it is imperative to classify the phases $\uptheta$ for which there are semistable objects and the associated classes $\updelta \in \K_0(\fdmod\Nccr)$.

The completed Jacobi algebra is a noncommutative crepant resolution (NCCR) of the singularity in $\Spec R$, and its tilting theory has recently been completely determined by Hirano--Wemyss \cite{HW19}.
In particular, it follows from \cite{HW19} that the 2-term tilting complexes generate a wall-and-chamber structure in the real vectorspace $\K_0(\proj \Nccr)_\R$, which is pictured in figure \ref{subfig:A}. Every chamber corresponds to a unique tilting complex $T = T_i \oplus T_{i+1} \in \cK^b(\proj \Nccr)$ whose \emph{g-vectors}
\[
  [T_i], [T_{i+1}] \in \K_0(\proj \Nccr)_\R,
\]
span the walls bounding a chamber, and adjacent tilting complexes are related by a mutation.
We show that, for a suitably generic choice of Bridgeland stability condition, the Euler pairing
\[
  \<-,-\> \colon \K_0(\proj \Nccr)_\R \otimes_\Z \K_0(\fdmod \Nccr) \to \R,
\]
yields a duality between the walls in $\K_0(\proj \Nccr)_\R$ and the lattice of dimension vectors of semistable modules in $\K_0(\fdmod \Nccr)$ in figure \ref{fig:scattering}, which is generated by a sequence of classes
\[
  \updelta_{\curve,n},\quad \updelta_{2\curve,n},\quad \updelta_\pt.
\]
Moreover, we are able to explicitly describe the stable objects: for every every wall spanned by a summand $T_i$, there is an adjacent tilting complex $T$ such that the tilting functor
\[
  -\Lotimes T \colon \D^b(\fdmod \End_\Nccr(T)) \xrightarrow{\ \sim\ } \D^b(\fdmod \Nccr),
\]
maps a simple $S \in \fdmod \End_\Nccr(T)$ of the tilted algebra to a stable $\Nccr$-module.
Each tilted algebra is isomorphic to $\Nccr$, and one therefore obtains two families of stable modules whose classes are $\updelta_{\curve,n}$ and $\updelta_{2\curve,n}$ respectively.
We show that these are the unique stable objects for each class, and the remaining stable objects are therefore of class $\updelta_\pt$.

To give an explicit description of the stable modules we employ the derived equivalence, which yields a description in terms of (shifted) sheaves supported on $\curve$.
\begin{introthm}[Theorem \ref{thm:stableobjs}]\label{introthm:B}
  For the above choice of central charge $Z$, there exists a $Z$-(semi)stable module with class $\updelta\in \N Q_0 \subset \K_0(\fdmod \Nccr)$ if and only if $\updelta$ is (a multiple of) one of
  \[
    \updelta_{\curve,n},\quad \updelta_{2\curve,n},\quad \updelta_\pt\quad\quad n\in\Z.
  \]
  For each $n\in\Z$ there is a unique $Z$-stable module of class $\updelta_{\curve,n}$ corresponding to a twist/shift
  \[
    \O_{\curve}(n-1) \quad (n\geq 0),\quad \O_\curve(n-1)[1]\quad (n<0),
  \]
  of the structure sheaf of $\curve$ across the derived equivalence.
  For each class $\updelta_{2\curve,n}$ with $n\in\Z$ there is a unique $Z$-stable module corresponding to a twist/shift
  \[
    \O_{2\curve}(n-1) \quad (n>0),\quad\quad \O_{2\curve(n-1)}[1]\quad (n\leq 0),
  \]
  of the structure sheaf of a certain thickening $2\curve \supset \curve$.
  The remaining stable objects are of class $\updelta_\pt$, corresponding to the point sheaves 
  \[
    \O_p\quad p\in\curve,
  \]
  and there are no other stable objects.
\end{introthm}
The proof of the theorem relies on a reduction to the setting of finite dimensional algebras: we show that the g-vectors of $\Nccr$-tilting complexes and the dimension vectors stable $\Nccr$-modules coincide with those of a finite-dimensional quotient.
For such finite dimensional algebras, the duality between g-vectors and stable modules was established in \cite{BST19,Asai21}, and we show that these results lift to the geometric setting.

As a consequence of Theorem \ref{introthm:B} we find a decomposition of the partition function along the phases $\uptheta_{\curve,n}$, $\uptheta_{2\curve,n}$, $\uptheta_\pt$ associated to $\updelta_{\curve,n}$, $\updelta_{2\curve,n}$, $\updelta_\pt$ respectively:
\[
  \Upphi(t) = \prod_n \Upphi^{\uptheta_{\curve,n}}(t) \cdot \prod_n \Upphi^{\uptheta_{2\curve,n}}(t) \cdot \Upphi^{\uptheta_\pt}(t).
\]
From this we deduce that a BPS invariant $\BPS_\updelta$ vanishes whenever $\updelta$ is \emph{not} a multiple of one of the given dimension vectors, while the remaining invariants can be extracted from one of the partition functions $\Upphi^\uptheta(t)$ for the given phases.
For $\uptheta=\uptheta_{\curve,n}, \uptheta_{2\curve,n}$ these partition function can be expressed via the deformation theory of the stable objects.

\subsection{Tilting preserves potentials}

In the setup of Kontsevich--Soibelman \cite{KS08} the DT theory of a quiver with potential $(Q,W)$ is determined by an enhancement of the derived category, which encodes the Calabi--Yau deformation theory of its objects.
In particular, the contribution of an object in $\D^b(\mod \Jac(Q,W))$ and its self-extensions to the partition function is determined by a quiver with potential obtained from this enhancement.
Moreover, in the setting of cluster algebras it is known by the work of Keller--Yang \cite{KY11} that these potentials are preserved under a process of mutation.

For our length $2$ flops, the partition function $\Upphi^\uptheta(t)$ for $\uptheta = \uptheta_{\curve,n},\uptheta_{2\curve,n}$ is precisely the contribution of the stable module $M$ of class $\updelta = \updelta_{\curve,n}, \updelta_{2\curve,n}$ and its self-extensions to the DT theory of $A=\Jac(Q,W)$.
Working with the enhancement, one therefore finds
\[
  \Upphi^\uptheta(t) = \Upphi_{\cQ_M,\cW_M}(t^\updelta),
\]
where the right-hand-side is the partition function of the quiver with potentials associated to $M$.
Our quiver $Q$ is not of cluster type, as each arrow is a loop or part of a 2-cycle, so the mutation theory of Keller--Yang does not apply.
Instead, we will deduce an analogous result for the tilting functors.

Working over the commutative base ring $R$ we consider \emph{$R$-linear standard equivalences} that satisfy a homological condition: if $F\colon \D^b(\mod A) \to \D^b(\mod A)$ is an equivalence that lifts to an $R$-linear enhancement, there is an induced $R$-linear action
\[
  \HH_3(F) \colon  \HH_3(A) \to \HH_3(A),
\]
on Hochschild homology, and we show that $F$ preserves the potentials if this action is a scalar.
This yields the following theorem, which applies to a much more general setting.
\begin{introthm}[Theorem \ref{thm:CYstructPots}]\label{introthm:C}
  Let $A = \Jac(Q,W)$ be a Jacobi algebra which is finite as an algebra over a central Noetherian subring $R \subset A$.
  Suppose $F\colon \D^b(\mod A) \to \D^b(\mod A)$ is an $R$-linear standard equivalence such that
  \[
    \HH_3(F) = \uplambda \in \C^\times.
  \]
  Then for every pair of nilpotent modules $M,N$ such that $\End_A(M) \simeq \C$ and $F(M)\simeq N$, the potentials $\cW_M$ and $\uplambda\cdot \cW_N$ are equivalent via a formal change of variables.
\end{introthm}
In our setting, we use the geometry of the flopping contraction $Y_{a,b}\to \Spec R_{a,b}$ to deduce that \emph{any} $R=R_{a,b}$-linear standard equivalences satisfies the homological condition, and in particular this applies to the tilting functors which relate the stable modules to the simple modules.
This implies that the contributions $\Upphi^\uptheta(t)$ for $\uptheta=\uptheta_{\curve,n}, \uptheta_{2\curve,n}$ do not depened on $n$ (up to a change in the variable $t$), and neither do the associated BPS invariants:
\[
  \BPS_{k \updelta_{\curve,n}} = \BPS_{k\updelta_{\curve,0}},\quad
  \BPS_{k \updelta_{2\curve,n}} = \BPS_{k\updelta_{2\curve,0}}\quad\quad \forall n\in\Z.
\]
As a result, it suffices to calculate these invariants for multiples of the classes $\updelta_{\curve,0} = [S_1]$ and $\updelta_{2\curve,0} = [S_0]$ of the vertex simples.

To prove theorem \ref{introthm:C} we use two enhancements: a DG-enhancement $\cA$ of the derived category $\D^b(\mod A)$ and a cyclic $A_\infty$-enhancement $(\cH,\upsigma)$ of $\D^b(\nilp A) \simeq \D^b(\fdmod \Nccr)$, as in the work of Kontsevich--Soibelman \cite{KS08,KS09}.
The cyclic structure $\upsigma$ is determined \emph{up to homotopy} by its Hochschild cohomology class $[\upsigma] \in \HH^3(\cH,\cH^*)$ and any auto-equivalence of $\cH$ which preserves this class yields an equivalence between induced potentials.
The enhancements $\cA$ and $\cH$ are related by local duality over the singularity in $\Spec R$, and at the level of Hochschild (co-)homology this yields a map
\[
  \Upsilon\colon \HH_3(A) \simeq \HH_3(\cA) \to \HH^3(\cH,\cH^*),
\]
as in the construction of Brav--Dyckerhoff \cite{BD19}.
An $R$-linear standard equivalence $F\colon \D^b(\mod A) \to \D^b(\mod A)$ lifts to $\cA$ and induces an equivalence $\cH\to\cH$, whose action on $\HH^3(\cH,\cH^*)$ is completely determined by the $R$-linear action of $F$ on $\HH_3(\cA)$: it is the unique $R$-linear map making the diagram
\[
  \begin{tikzpicture}
    \node (A) at (0,0) {$\HH_3(A)$};
    \node (B) at (0,-\gr) {$\HH_3(A)$};
    \node (C) at (3,0) {$\HH^3(\cH,\cH^*)$};
    \node (D) at (3,-\gr) {$\HH^3(\cH,\cH^*)$};
    \draw[->] (A) to[edge label'=$\HH_3(F)$] (B);
    \draw[->] (A) to[edge label=$\Upsilon$] (C);
    \draw[->] (B) to[edge label=$\Upsilon$] (D);
    \draw[dashed,->] (D) to (C);
  \end{tikzpicture}
\]
commute.
In this way the homological condition in theorem \ref{introthm:C} translates to a linear scaling of the cyclic structure, and thereby a linear scaling of the potentials.

The homological condition in the theorem can be motivated via the Calabi--Yau setting: if $A$ is derived equivalent to a $3$CY variety $X$ then the enhancement $\cA$ carries a (weak) \emph{left} CY structure induced by the volume form in $\H^0(X,\omega_X) = \HH_3(X) \simeq \HH_3(\cA)$, which the map $\Upsilon$ maps to a (weak) \emph{right} CY structure that determines a cyclic structure.
If $\HH_3(F) = \uplambda$ then $F$ scales the Calabi--Yau volume on $X$ linearly, and induces the inverse scaling on the cyclic structure, and thereby the potential.
However, Theorem \ref{introthm:C} deliberately does not use the existence of a volume form, and therefore avoids the somewhat delicate question of wether our choice of cyclic structure $\upsigma$, which comes from the presentation $(Q,W)$ of $A$, is the image of a volume form.

\subsection{Outline of the paper}

In \S\ref{sec:flops} we briefly recall the (non-commutative) geometry of simple flopping contractions and a construction of the family of length $\ell=2$ flops.
Section \S\ref{sec:DTtools} sets up the framework of Donaldson--Thomas theory, while the following sections contain our main theorems.
In \S\ref{sec:stables} we establish the relation between tilting and stability and give the classification of stable modules of theorem \ref{introthm:B}.
The BPS invariants are computed in \S\ref{sec:DTcalc}, resulting in Theorem \ref{introthm:A}.
This calculation relies heavily on a corollary to Theorem \ref{introthm:C}, which is proved in \S\ref{sec:autoequiv}.

\subsection{Acknowledgements}

The author would first and foremost like to thank his supervisors Michael Wemyss and Ben Davison
for their guidance, patience and ever-present optimism.
He would also like to thank Jenny August, Theo Raedschelders, Greg Stevenson, Hipolito Treffinger, and the members of his viva commission for helpful discussions. 
This work is part of the author's PhD dissertation at the University of Glasgow, and the author thanks the university for their generous support.

\section{Flopping Geometry}\label{sec:flops}
Let $Y$ be a smooth quasiprojective threefold. Recall that a map $\uppi\colon Y \to Y_\con$ onto a normal variety $Y_\con$ is a \emph{flopping contraction} if it is projective birational, with exceptional locus consisting of curves that are mapped to isolated Gorenstein singularities in $Y_\con$, and satisfies the condition
\[
  \bR \uppi_* \O_Y = \O_{Y_\con}.
\]
A flopping contraction is \emph{simple} over a point $p\in Y_\con$ if the exceptional curve $\uppi^{-1}(p) \simeq \P^1$. In general, the scheme theoretic fibre $Y \times_{Y_\con} \{p\}$ is non-reduced, and one defines the \emph{length} $\ell$ as the multiplicity of its structure sheaf at the generic point of $\P^1$, see e.g. \cite{Katz08}.

In what follows we consider a simple flopping contraction $\uppi\colon Y\to Y_\con = \Spec R$ over an affine base, which has length $\ell=2$ over some maximal ideal $\o\in \Spec R$ and write
\[
  \curve \colonequals \uppi^{-1}(\o),\quad 2\curve \colonequals Y \times_{Y_\con} \{\o\},
\]
for the reduced and scheme theoretic fibres respectively.
Where convenient we also consider the complete local case, where $R$ is completed at $\o$.

\subsection{Noncommutative description of simple flops}

To apply noncommutative methods, we will also make the assumption that $Y$ admits a \emph{tilting bundle} of the form $\cP = \O_Y \oplus \cN$ for some indecomposable vectorbundle $\cN$, which induces a derived equivalence
\[
  \begin{tikzpicture}
    \node (A) at (-2.5,0) {$\D^b(\Coh Y)$};
    \node (B) at (2.5, 0) {$\D^b(\mod \End_Y(\cP))$};
    \draw[->] ([yshift=1mm]A.east) to[edge label=${\scriptstyle\Psi=\RHom_Y(\cP,-)}$] ([yshift=1mm]B.west);
    \draw[->] ([yshift=-1mm]B.west) to[edge label=${\scriptstyle\Psi^{-1}=-\Lotimes\cP}$] ([yshift=-1mm]A.east);
  \end{tikzpicture}
\]
onto the derived category of the endomorphism algebra of $\cP$. As we work over an affine base, $\Uppsi$ is an $R$-linear equivalence with respect to the obvious $R$-linear structures on the derived categories. In particular, $\Uppsi$ restricts to an equivalence
\[
  \begin{tikzpicture}
    \node (A) at (-2.5,0) {$\D^b(\Coh_\curve Y)$};
    \node (B) at (2.5, 0) {$\D^b(\fdmod_\o \End_Y(\cP))$};
    \draw[->] ([yshift=1mm]A.east) to[edge label=${\scriptstyle\Psi}$] ([yshift=1mm]B.west);
    \draw[->] ([yshift=-1mm]B.west) to[edge label=${\scriptstyle\Psi^{-1}}$] ([yshift=-1mm]A.east);
  \end{tikzpicture}
\]
between complexes of sheaves with compact support contained in the fibre $\curve = \uppi^{-1}(\o)$, and complexes of finite dimensional modules which are $\o$-power torsion as $R$-modules.

If the base is complete local, a tilting bundle of the above form is known to exist by a construction of Van den Bergh \cite{VandenBergh04}. In this construction $\cN = \cM^*$ is the dual of the unique extension
\[
  0 \to \O_Y \to \cM \to \O_Y(1) \to 0,
\]
associated to a generator of $\H^1(Y,\O_Y(-1))$. The endomorphism algebra $\End_Y(\cP)$ is moreover a maximal Cohen-Macaulay module over $R$, making it a \emph{noncommutative crepant resolution} (NCCR).
In this setting the algebra $\End_Y(\cP)$ has two unique simple modules $S_0$, $S_1$ which correspond to the shifted sheaves
\[
  \O_{2\curve}(-1)[1] = \Psi^{-1}(S_0),\quad \O_\curve(-1) =\Psi^{-1}(S_1),
\]
In specific examples, a tilting bundle can be constructed even if the base is not complete local, see e.g. the work of Aspinwall-Morrison \cite{AM12}.

To calculate Donaldson--Thomas invariants, we will present $\End_Y(\cP)$ as the Jacobi algebra of a quiver with potential $(Q,W)$, which is the quotient
\[
  \Jac(Q,W) \colonequals \frac{\C Q}{(\del_a W \mid a \in Q_1)}
\]
of the path algebra of $\C Q$ by cyclic derivatives of the potential $W \in \C Q/[\C Q,\C Q]$. The existence of such a presentation is known by explicit construction in some cases, and is again guaranteed to exist in the complete local case by work of Van den Bergh \cite{VandenBergh15}, in which case one has to consider the completion $\widehat \Jac(Q,W)$ of the Jacobi algebra at the ideal generated by the arrows $a\in Q_1$.

\subsection{A family of length $2$ flops}\label{sec:familyexamples}
For the DT theory calculations in section \ref{sec:DTcalc} we will use an explicit family of flopping contractions. During the initial write up of this paper we discovered that this same family was simultaneously and independently studied by Kawamata \cite{Kawamata20}, who moreover classified all isomorphism classes in it. In view of this fact, we only give a brief alternative construction using moduli spaces of quivers with potential; a more complete account is give in the thesis version of this paper \cite{vGThesis}.

Given parameters $a\in \N_{\geq 2}$ and $b\in \N_{\geq 1} \cup \{\infty\}$, consider the quiver with potential
\[
  \begin{tikzpicture}[>=stealth,circle,inner sep=1pt]
    \clip (-5,1) rectangle (7,-1);
    \node at (-4,0) {$Q\colon$};
    \node[outer sep=.8pt] (A) at (-2.8,0) {0};
    \node[outer sep=.8pt] (B) at (-1.2,0) {1};
    \draw[->] (A) to[bend left] node[midway,fill=white, inner sep=.2pt]{$c$} (B);
    \draw[->] (B) to[bend left] node[midway,fill=white, inner sep=.2pt]{$d$} (A);
    \draw[->] (B) to[loop,looseness=10,in=280,out=-10] node[midway,fill=white, inner sep=.2pt]{$x$} (B);
    \draw[->] (B) to[looseness=10,in=80,out=10] node[midway,fill=white, inner sep=.2pt]{$y$} (B);  
    \draw[->] (A) to[loop,looseness=10,in=140,out=220] node[midway,fill=white, inner sep=.2pt]{$s$} (A);
    \node at (3.5,0) {$W_{a,b}= x^2y - f_{a,b}(y) + y^2cd - sdc + 2 s^a,$};
  \end{tikzpicture}
\]
where $f_{a,b}(y) = y^{2a}$ for $b=\infty$ and is $f_{a,b}(y) = y^{2a} + y^{2b+1}$ otherwise.
The representations of the quiver $Q$ of dimension vector $\updelta = (1,2) \in \N Q_0$ are parametrised by the affine space
\[
  \Rep_{\updelta}(Q) \simeq \Mat_{1\times1} \times \Mat_{2\times 1} \times \Mat_{1\times 2} \times \Mat_{2\times 2}^2,
\]
acted on by $\GL_{\updelta} \colonequals \GL_1 \times \GL_2$ via conjugation, and the cyclic derivatives of the potential cut out a $\GL_{\updelta}$-invariant subvariety $\Rep_{\updelta}(Q,W_{a,b}) \subset \Rep_{\updelta}(Q)$.
The GIT quotient by the action yields the singular affine moduli scheme of semisimple representations
\[
  \cM_{\updelta}(Q,W_{a,b}) \colonequals \Rep_{\updelta}(Q,W_{a,b})\gitquot \GL_{\updelta} \colonequals \C[\Rep_\updelta(Q,W)]^{\GL_{\updelta}}.
\]
A direct computation shows that this is a hypersurface singularity with coordinate ring $R_{a,b} \simeq \C[u,v,r,s]/(U_{a,b})$ defined by an equation
\[
  U_{a,b} = \begin{cases}
    u^2 + r^3 + sv^2 + 4a^2r s^{2a-1} & b = \infty\\
    u^2 + r(r + (2b+1)s^b)^2 + sv^2 + 4a^2 r s^{2a-1} & b\neq \infty
  \end{cases},
\]
which is related to the base of a flopping contraction in the family constructed by Kawamata \cite[\S5]{Kawamata20} by a change of variables.
To obtain the flopping contraction, one can resolve $\Spec R_{a,b}$ by the moduli scheme of (semi)stable representations
\[
  \cM_{\updelta}^\uptheta(Q,W_{a,b}) \colonequals \Rep_{\updelta}^\uptheta(Q,W_{a,b})\gitquot \GL_\updelta,
\]
where $\Rep_\updelta^\uptheta(Q,W_{a,b})$ is the subspace of semistable representations for a King stability condition $\uptheta = (-2,1) \colon \N Q_0 \to \R$.
Defining $Y_{a,b} \colonequals \cM_\updelta^\uptheta(Q,W_{a,b})$, the GIT construction yields a surjective and projective map $\uppi\colon Y_{a,b} \to \Spec R_{a,b}$.
The moduli scheme is a gluing $Y_{a,b} = U_x \cup U_y$ of two smooth affine charts
\[
  \begin{aligned}
    U_x &= \Spec \frac{\C[d_1,x_{01},y_{00},y_{10}]}{\left(x_{01} - 2a y_{00}(y_{00}^2-x_{01}y_{10}^2)^{a-1} + d_1y_{10} (-(y_{00}^2-x_{01}y_{10})^b))\right)}\\
    U_y &= \Spec \C[x_{00},x_{01},y_{01}] \simeq \A^3,
  \end{aligned}
\]
Hence $\uppi\colon Y_{a,b} \to \Spec R_{a,b}$ is a resolution of singularities, and coincides with the length 2 flopping contraction constructed by Kawamata (or its flop).
Because one of the charts is $\A^3$ and the restricted map $\uppi|_{U_x}\colon U_X \to Y_\con$ is dominant, the base has units
\[
  R_{a,b}^\times \simeq \H^0(U_x,\O_Y|_{U_x})^\times \simeq \H^0(\A^3,\O_{\A^3})^\times \simeq \C^\times.
\]
This property will be crucial for our DT calculations in \S\ref{sec:DTcalc}.
The space $Y_{a,b}$ comes equipped with a tilting bundle $\cP$ such that $\End_{Y_{a,b}}(\cP) \simeq \Jac(Q,W_{a,b})$, as shown in \cite[\S 2]{vGThesis}.

\section{The DT Toolbox}\label{sec:DTtools}
In this section we recall the machinery of motivic Donaldson-Thomas theory for symmetric quivers with potential, based on the foundational work of  Kontsevich--Soibelman \cite{KS08} Szendr\Hu oi \cite{Szendroei08} Joyce--Song \cite{JS12} and others. 
Our setup and notation mostly follows the work of Davison--Meinhardt \cite{DM15,DM17}.

Let $Q$ be a finite symmetric quiver and denote by $\Updelta = \N Q_0$ its monoid of dimension vectors.
Given a dimension vector $\updelta\in\Updelta$ the $\C Q$-modules with dimension vector $\updelta$ form the affine space 
\[
  \Rep_\updelta(Q) \simeq \prod_{(a\colon v\to w) \in Q_1} \Hom_\C(\C^{\updelta_v}, \C^{\updelta_w}),
\]
which carries an action of the algebraic group $\GL_\updelta \colonequals \prod_{v\in Q_0} \GL_{\updelta_v}$ in the obvious way. One constructs the associated \emph{moduli stack} as the quotient
\[
  \fM_\updelta \colonequals \Rep_\updelta(Q) / \GL_\updelta,
\]
of which the $\C$-points correspond to isomorphism classes of $\updelta$-dimensional $\C Q$-modules. The \emph{coarse moduli scheme} of $Q$ is the scheme-theoretic quotient
\[
  \cM_\updelta \colonequals \Rep_\updelta(Q) \gitquot \GL_\updelta,
\]
of which the $\C$-points correspond to isomorphism classes of \emph{semisimple} modules. There is a well-defined map $\fM_\updelta \to \cM_\updelta$, which (at the level of $\C$-points) sends a $\updelta$-dimensional module $M$ to the sum of the simples appearing in its composition series.
For each $\updelta$, we define the moduli stack of \emph{nilpotent} representations $\fN_\updelta \subset \fM_\updelta$ as the fibre above the semisimple module 
\[
  \bigoplus_{v\in Q_0} S_v^{\updelta_v} \in \cM_\updelta(Q),
\]
where $S_v \in \fdmod \C Q$ denotes the one-dimensional simple supported on the vertex $v$.
Where appropriate we drop the subscript $\updelta$ from the notation to denote the disjoint union over all dimension vectors, writing e.g. $\fM \colonequals \coprod_{\updelta\in \Delta} \fM_\updelta$ and $\fN \colonequals \coprod_{\updelta\in\Delta} \fN_\updelta$.

Given a potential $W\in \C Q_\cyc \colonequals \C Q/[\C Q,\C Q]$, its trace is a well-defined $\GL_\updelta$-equivariant function $\tr(W)$ on each affine space $\Rep_\updelta(Q)$, and hence descends to a regular function $\tr(W)$ on $\fM$.
This function has a well-defined stacky critical locus $\fM_{Q,W} \subset \fM$, whose intersection with $\fN$ we denote by 
\[
  \fC = \fC_{Q,W} \colonequals \fM_{Q,W} \cap \fN.
\]
The closed points $\fC(\C)$ correspond to the nilpotent $\Jac(Q,W)$-modules, although it is not necessarily a good moduli space.
The goal of motivic Donaldson--Thomas theory is to assign a motivic invariant to the critical locus $\fC$ which acts as a ``virtual count'' of the points $\fC(\C)$.
To do this, one constructs a \emph{motivic vanishing cycle} $\upphi_{\tr(W)}$ in some ring of motivically valued measures. 
Integrating this vanishing cycle over the components of $\fC$ defines a generating function
\[
  \Upphi(t) = \Upphi_{Q,W}(t) \colonequals \sum_{\updelta \in \Updelta}\int_{\fC_\updelta}\upphi_{\tr(W)} \cdot  t^\updelta,
\]
with motivic coefficients. This generating function is the \emph{DT partition function} and its coefficients the \emph{DT invariants}, which are a motivic refinement of the enumerative DT invariants of Joyce--Song \cite{JS12}.
The partition function can moreover be simplified using various wall-crossing relations and multiple-cover formulas. We recall the relevant constructions in the following subsections.

\subsection{Rings of motives}
Recall that the Grothendieck ring of varieties $\K_0(\Var/\C)$, is the abelian group generated by the isomorphism classes $[X]$ of reduced separated schemes of finite type over $\C$ subject to the cut-and-paste relations
\[
  [X] = [Z] + [X\setminus Z] \quad \text{ for } Z\subset X \text{ a closed subvariety},
\]
which has the structure of a ring via the multiplication $[X]\cdot[Y] = [X\times Y]$ and unit $1 = [\Spec \C]$. The goal of motivic counting theories is to refine an integer valued invariant by lifting it to a motive in the ring $\K_0(\Var/\C)$. However, in motivic Donaldson--Thomas theory one has to make some additional technical modifications.

Firstly, motivic DT invariants naturally come equipped with \emph{monodromy}, so that one requires an equivariant equivariant version of $\K_0(\Var/\C)$.
Secondly, the invariants are defined by integrating a motivic vanishing cycle, which requires the use of relative classes.
Lastly, the stacky nature of the moduli spaces $\fM_\updelta$ requires one to also consider stack rather than just varieties.
The necessary modifications are available in the literature, see e.g. \cite{DM15} for a good treatment. We recall here the important parts of the construction.

Let $\St$ denote the category of Artin stacks, locally of finite type over $\C$, having affine stabilisers, and fix $\fM\in \St$.
Given a map $f\colon \fX \to \fM$ in $\St$, a \emph{monodromy action} on $\fX$ is an action of the group scheme 
\[
  \muhat \colonequals \lim (z^a \colon \upmu_{an} \to \upmu_n)_{a,n\in\N},
\]
which factors through a sufficiently nice action of the group scheme $\upmu_n$ of $n$th roots of unity for some $n$, for which the map $f$ is $\muhat$-invariant.
If $\fM$ is of finite type, the Grothendieck group of stacky monodromic motives is the abelian group $\K^\muhat(\St/\fM)$ generated by equivalence classes $[\fX\to\fM]$ of stacks with monodromy over $\fM$, subject to the relations
\[
  \begin{aligned}
    [\fX \xrightarrow{f} \fM] &= [\fZ\xrightarrow{f|_\fZ}\fM] + [\fX\setminus\fZ\xrightarrow{f|_{\fX\setminus\fZ}}\fM],\\
    0 &= [\fY\xrightarrow{f\circ g} \fM] - [\A^r \times \fX \xrightarrow{f \circ \pr_\fX} \fM]
  \end{aligned}
\]
for closed substacks $\fZ\subset\fX$, and $\muhat$-equivariant vector bundles $g\colon\fY\to\fX$ of rank $r$.
For general $\fM\in\St$, the above defines a group $\K^\muhat_{pre}(\St/\fM)$ and $\K^\muhat(\St/\fM)$ can be defined via a suitable completion.
The classes $[\fX\to\fM]$ with trivial monodromy action form a subgroup which is denoted by $\K(\St/\fM)$.
Any finite type map $j\colon \fM \to \fN$ in $\St$ induces pull-back and push-forward functors via
\[
  j_*[f\colon \fX\to \fM] = [j\circ f\colon \fX\to \fN],\quad
  j^*[f\colon \fX\to\fN] = [j^*f \colon \fX\times_\fN \fM \to \fM]. 
\]
For a substack $\fZ\subset \fM$ we use the special notation $|_\fZ$ for the pullback along the inclusion.

Any variety $X$ is in particular a finite type stack, and the classes $[X\to \fM]$ 
generate a subgroup $\K^\muhat(\Var/\fM) \subset \K^\muhat(\St/\fM)$.
In particular, for $\fM=\Spec \C$ one obtains the (commutative) ring of absolute monodromic motives $\K^\muhat(\Var/\C)$, equipped with a certain exotic product (see \cite{Looijenga02}).
The ordinary Grothendieck ring of varieties can be recovered as the subring of classes with trivial monodromy:
\[
  \K_0(\Var/\C) = \K(\St/\C) \cap \K^\muhat(\Var/\C).
\]
We write absolute motives simply as $[X]$, ignoring
the structure morphism to $\Spec \C$. The class of the affine line, known as the Lefschetz motive, will be denoted
\[
  \L \colonequals [\A^1] \in \K(\Var/\C) \subset \K^\muhat(\Var/\C).
\]
The Lefschetz motive has a distinguished square root in $\K^\muhat(\Var/\C)$ given by
\[
  \L^\half \colonequals 1 - [\upmu_2] \in \K^\muhat(\Var/\C),
\]
where $[\upmu_2]$ is the class of the group scheme $\upmu_2$ equipped with its obvious monodromy action.
The ring of motives we consider is the localisation
\[
  \Mothat \colonequals \K^\muhat(\Var/\C)\left[[\GL_n]^{-1} \mid n \in \N \right],
\]
at the classes $[\GL_n] = (\L^n-\L^{n-1})\cdots(\L^n-1)$. Note that this localisation in particular contains the classes $\L^{-n/2}$ for any $n\in \N$.

The ring $\K^\muhat(\Var/\C)$ acts on $\K^\muhat(\Var/\fM)$ and $\K^\muhat(\St/\fM)$ for any $\fM\in\St$, again via an exotic product; for a variety $X$ with with trivial monodromy, its class acts simply as
\[
  [X]\cdot[\fY\to \fM] = [X\times \fY\to \fM].
\]
In particular, it makes sense to localise these modules at the classes $[\GL_n]$. For the case of varieties this yields a genuinely new module
\[
  \Mothat(\fM) \colonequals \K^\muhat(\Var/\fM)\left[[\GL_n]^{-1} \mid n\in \N\right],
\]
while for the stacky case $\K^\muhat(\St/\fM)$ is already a module over $\Mothat$ with action
\[
  [\GL_n]^{-1}\cdot [\fY\to \fM] = [B\!\GL_n\times \fY\to \fM],
\]
where $B\!\GL_n$ denotes the classifying stack, and is therefore equal to its localisation.
Moreover, by \cite[Proposition 2.8]{DM15} the localisation of the inclusion $\K^\muhat(\Var/\fM) \into \K^\muhat(\St/\fM)$ yields an isomorphism
\[
  \Mothat(\fM) \xrightarrow{\ \sim\ } \K^\muhat(\St/\fM),
\]
as any element in $\St$ has a stratification into quotient stacks $X/\GL_n \simeq B\!\GL_n\times X$.
Via this isomorphism, one can interpret the elements of the module $\Mothat(\fM) \simeq \K^\muhat(\St/\fM)$ as measures on $\fM$ valued in $\Mothat$: 
for a finite type stack $a\colon \fX\to \Spec \C$ with a map $i\colon \fX \to \fM$ an element $m\in \Mothat(\fM) \simeq \K^\muhat(\St/\fM)$ has a well-defined integral
\[
  \int_\fX m \colonequals a_*i^*m \in \Mothat.
\]
One can moreover show that this integral only depends on the class of $[i\colon \fX\to\fM]$ inside $\K(\St/\fM)$, yielding a pairing $\K(\St/\fM) \times \Mothat(\fM) \to \Mothat$.

Motivic invariants can be collected in generating series, expressed as elements of a ring of multi-variate motivic power series: if $S$ is a free monoid of finite rank we let
\[
  \Mothat[[S]] = \Mothat[[t^s \mid s \in S]],
\]
denote the multivariate powerseries, where the product of two indeterminates is defined as $t^s\cdot t^{s'} = t^{s+s'}$.
Such rings have an additional pre-$\uplambda$-ring structure \cite[\S3]{DM15}, defined by a map
\[
  \Sym \colon \Mothat(C)[[S]] \to 1 + \Mothat(\C)[[S]],
\]
called the \emph{plethystic exponential}, which satisfies the exponential identities
\[
  \begin{gathered}
    \Sym(0) = 1,\quad \Sym(a +b) = \Sym(a)\Sym(b),\\
    \Sym(a\cdot t^s) = 1 + a\cdot t^s + \ldots\text{higher order terms}\ldots
  \end{gathered}
\]
The plethystic exponential allows one to systematically derive multiple-cover formulas for motivic invariants by using an ansatz $\Sym(\sum_{s\in S}a_s t^s)$, and computing the values $a_s \in \Mothat$ term by term.

\subsection{Motivic vanishing cycles}\label{sec:vancyc}

The motivic vanishing cycle is a rule which assigns to a regular function $f\colon \fM \to \A^1$ on a smooth stack $\fM\in\St$ a measure $\upphi_f \in \Mothat(\fM)$. Its construction proceeds in successive levels of generality.
\begin{enumerate}
\item For a regular function $f\colon M\to \A^1$ one a smooth scheme $M$, Denef--Loeser \cite{DL98} construct a vanishing cycle $\upphi_f \in \Mothat(M)$ via a certain rational function, defined in terms of the lifts of $f$ to the arc-space of $M$.
\item   For a regular function $f\colon \fM\to \A^1$ on a quotient stack $\fM = M/\GL_n$ of a smooth variety $M$ one defines
  \[
    \upphi_f = \L^{\dim \GL_n\!/2} [\GL_n]^{-1} q_*\upphi_{f\circ q} \in \K^\muhat(\St/\fM) \simeq \Mothat(\fM),
  \]
  where $q\colon M\to \fM$ is the quotient map, and $\upphi_{f\circ q}$ is the Denef--Loeser vanishing cycle of $f\circ q \colon M \to \A^1$.
\item For a general $\fM\in\St$, the vanishing cycle is recovered from a constructible decomposition of $\fM$ into suitable quotient stacks via the Luna slice theorem.
\end{enumerate}
Several technical tools have been developed to explicit computate the vanishing cycle. For quasi-homogenous functions one has the following theorem of Nicaise--Payne.
\begin{theorem}[{\cite{NP19}}]
  Let $M$ be a smooth variety and let $\G_m$ act on the product $\overline M = M\times \A^n$ via an action on $\A^n$ by nonnegative weights. Suppose 
  \[
    f\colon \overline M \to \A^1,
  \]
  is a $\G_m$-equivariant function, where $\G_m$ acts on $\A^1$ by a weight $d>0$. Then
  \[
    \upphi_f = \L^{-\dim \overline M/2}\left([f^{-1}(0) \to M] - [f^{-1}(1)\to M]\right),
  \]
  where $f^{-1}(1)$ carries the residual $\upmu_d$-action as its monodromy.
\end{theorem}

For nonhomogenous functions, one can instead apply the following construction of Denef--Loeser.

Let $f: M \to \A^1$ be a non-constant regular function on a smooth scheme of pure dimension $d$, and write $M_0 \colonequals f^{-1}(0)$ for the associated divisor.
Let $p: \widetilde M \to M$ be an embedded resolution of $M_0$, i.e. $p$ is an isomorphism away from $M_0$
and the pull-back $E \colonequals p^*M_0 = m_1 E_1 + \ldots + m_n E_n$
has normal crossings\footnote{Here we mean normal crossing in the weak sense, allowing multiplicities in the divisors, and not the stronger notion of \emph{simple} normal crossing.} in a neighbourhood of $p^{-1}(M_0)$.
For any non-empty subset $I\subset \irr(E)$ of the irreducible components $\irr(E) = \{E_1,\ldots,E_n\}$ let
\[
E_I \colonequals \bigcap_{E_i \in I} E_i,\quad E_I^\circ \colonequals E_I \setminus \bigcup_{E_i \in \irr(E)\setminus I} E_i.
\]
The spaces $E_I^\circ$ form a constructible decomposition of $p^{-1}(X_0)$, and for each stratum there exists a cover $D_I \to E_I$, \'etale over $E_I^\circ$ with a canonical action of the Galois group $\upmu_{m_I}$ for $m_I \colonequals \gcd\{m_i\}_{E_i\in I}$.
The vanishing cycle is then computed by the following formula of Denef--Loeser \cite{DL99} (see also Looijenga \cite{Looijenga02}):
\begin{equation}\label{eq:motintformula}
   \upphi_f = \L^{-\frac{\dim M}{2}}\big([M_0 \into M] -
  \sum_{\varnothing \neq I \subset \irr(E)}
  (1-\L)^{|I|-1}\left[D_I^\circ \to M_0 \into M\right]\big),
\end{equation}
where $D_I^\circ$ is understood to carry monodromy via the $\upmu_{m_I}$-action.

Finally, there is the following motivic Thom-Sebastiani identity, which allows one to compute motivic vanishing cycles of decomposable functions on direct products.
\begin{theorem}[{\cite{GLM06}}]\label{thm:thomsebas}
  Let $f\colon \fM\to \A^1$ and $g\colon \fN\to \A^1$ be functions on smooth stacks of finite type, and $\fX \subset \fM$, $\fY \subset \fN$ closed substacks, then
  \[
    \int_{\fX\times \fY} \upphi_{f + g} = \int_{\fX}\upphi_f\cdot \int_{\fY}\upphi_g.
  \]
\end{theorem}

\subsection{The Motivic Hall algebra}\label{ssec:mothall}

We return to the quiver setting, working with the moduli stack $\fM = \fM_Q$.
Given a potential $W\in \C Q_\cyc$, the trace function $\tr(W) \colon \fM \to \A^1$ defines a motivic vanishing cycle, which is supported on the stacky critical locus $\fM_{Q,W}$.
The stack $\fM_{Q,W}$ parametrises $\Jac(Q,W)$-modules, and the points of $\fM_{Q,W}$ are therefore related by short-exact sequences.
Using this additional structure, one can make $\K(\St/\fM_{Q,W})$ into an algebra, the \emph{motivic Hall algebra} \cite{Joyce07}.
We recall the construction here, refering to \cite{Bridgeland12} for a more complete review.

Given dimension vectors $\updelta_1,\updelta_2 \in\Updelta$, there exists a moduli stack $\fExt_{\updelta_1,\updelta_2}$ parametrising equivalence classes of short-exact sequences 
\[
  0 \to M_1 \to N \to M_2 \to 0
\]
where $M_i$ is a module of dimension $\updelta_i$ and $N$ a module of dimension $\updelta_1+\updelta_2$. There are three projections
\[
  p_i\colon \fExt_{\updelta_1,\updelta_2} \to \fM_{Q,W,\updelta_i}, \quad
  q \colon \fExt_{\updelta_1,\updelta_2} \to \fM_{Q,W,\updelta_1+\updelta_2},
\]
mapping a short-exact sequence to the modules $M_i$, and $N$ respectively.
Given maps $f_i\colon \fX_i \to \fM_{Q,W,\updelta_i}$ of finite type, there is a pullback diagram
\begin{equation}\label{eq:HAprod}
  \begin{tikzpicture}[baseline=(current bounding box.center)]
    \node (A) at (0,0) {$\fY$};
    \node (B) at (4,0) {$\fExt_{\updelta_1,\updelta_2}$};
    \node (C) at (8,0) {$\fM_{Q,W,\updelta_1+\updelta_2}$};
    \draw[->] (A) edge[edge label=$\scriptstyle g$] (B);
    \draw[->] (B) edge[edge label=$\scriptstyle q$] (C);
    \node (D) at (0,-\gr) {$\fX_1\times\fX_2$};
    \node (E) at (4,-\gr) {$\fM_{Q,W,\updelta_1}\times \fM_{Q,W,\updelta_2}$};
    \draw[->] (A) edge[edge label=$\scriptstyle g$] (D);
    \draw[->] (B) edge[edge label=$\scriptstyle p_1\times p_2$] (E);
    \draw[->] (D) edge[edge label=$\scriptstyle f_1\times f_2$] (E);
  \end{tikzpicture}
\end{equation}
and one defines the \emph{convolution product} of the classes $[f_i\colon \fX\to \fM_{Q,W,\updelta_i}]\in \K(\St/\fM_{Q,W,\updelta_i})$ as the top row in the diagram:
\[
  [\fX_1\xrightarrow{f_1}\fM_{Q,W,\updelta_1}]
  \star [\fX_2\xrightarrow{f_2}\fM_{Q,W,\updelta_2}] =  [\fY\xrightarrow{q \circ g} \fM_{Q,W,\updelta_1+\updelta_2}] \in \K(\St/\fM_{Q,W,\updelta_1+\updelta_2}).
\]
Interpreting the $\K(\St/\fM_{Q,W,\updelta})$ as submodules of $\K(\St/\fM_{Q,W})$ via the pushforward along the inclusion, and noting that any class $[\fX\to \fM_{Q,W}]$ splits as a sum
\[
  [\fX \to \fM_{Q,W}] = \textstyle \sum_{\updelta\in \Updelta}[\fX_\updelta \to \fM_{Q,W,\updelta}],
\]
one sees that $\star$ endows $\K(\St/\fM_{Q,W})$ with an algebra structure over $\K(\Var/\C)$.
For our purposes we restrict to the nilpotent locus $\fC = \fN \cap \fM_{Q,W}$, and define the \emph{motivic Hall algebra} as the subalgebra
\[
  \cH(Q,W) \colonequals (\K(\St/\fC),\star).
\]

There is an integration map, which maps an element $[\fX\to \fC] \in \cH(Q,W)$ to the motivic powerseries 
\[
  \int_{[\fX\to\fC]} \upphi_{\tr(W)}|_{\fC} \colonequals \sum_{\updelta\in\Updelta} \int_{\fX_\updelta} \upphi_{\tr(W)} \cdot t^\updelta \in \Mothat[[\Delta]],
\]
obtained by integrating the restriction of the vanishing cycle on each stratum.
It follows from the proof of Kontsevich--Soibelman's integral identity \cite{KS08} by Thuong \cite{Le15} that the integration map defines a \emph{$\K(\Var/\C)$-algebra homomorphism}
\[
  \int_\bullet\upphi_{\tr(W)}|_{\fC}\colon \cH(Q,W) \to \Mothat[[\Updelta]],
\]
see \cite[Prop. 6.19]{DM15a} for an explanation of this fact. 
The \emph{Donaldson--Thomas partition function} of $(Q,W)$ is the image of the canonical element $[\id\colon \fC\to \fC]$:
\[
  \Upphi(t) = \int_{[\id\colon \fC\to\fC]} \upphi_{\tr(W)}|_\fC = \sum_{\updelta\in\Delta} \int_{\fC_\updelta} \upphi_{\tr(W)} \cdot t^\updelta.
\]
Using the pre-$\uplambda$-ring structure on $\Mothat[[\Delta]]$, one can recast the partition function as a plethystic exponential with the following ansatz
\begin{equation}\label{eq:BPSansatz}
  \Sym\left(\sum_{\updelta\in\Delta} \frac{\BPS_\updelta}{\L^\half-\L^\mhalf} t^\updelta \right) \colonequals \Upphi(t).
\end{equation}
This defines a sequence of invariants, which are supposed to be a refinement of BPS numbers in physics, so we refer to them as \emph{motivic BPS invariants}.

Using the integration map, one can systematically express various identities between DT invariants via algebraic identities in the Hall algebra. Most prominently, the wall-crossing identities induced by stability conditions.

\begin{remark}
  For the statement that the integration map is a homomorphism, our assumption that $Q$ is a symmtric quiver is crucial. For general quivers, one has to modify the multiplication in the ring $\Mothat[[\Delta]]$ by a sign twist.
\end{remark}

\subsection{Decomposition through stability}

\begin{define}
  Let $\cA$ be an abelian category such that $\K_0(\cA) \simeq \Z^{\oplus n}$ has finite rank.
  Then a stability condition on $\cA$ is a group homomorphism $Z\colon \K_0(\cA) \to \C$
  such that any non-zero object of $M\in \cA$ is mapped to a non-zero vector $Z([M])$
  with phase 
  \[
    \Uptheta([M]) \colonequals \mathrm{Arg}(Z([M])) \in (0,\pi].
  \]
  A non-zero object $M\in\cA$ is \emph{$Z$-semistable} if for every subobject $N\into M$ there is an inequality
  \[
    \Uptheta([N]) \leq \Uptheta([M]) \leq \Uptheta([M/N]).
  \]
  The object $M$ is \emph{$Z$-stable} if this inequality is strict for $N\neq 0,M$.
  The semistable objects of a phase $\uptheta \in (0,\pi]$ together with the zero-object, form an abelian subcategory $\cA^\uptheta \subset \cA$.
\end{define}
For a quiver with potential, the abelian category $\cA = \nilp \Jac(Q,W)$ of nilpotent modules has Grothendieck group $\K_0(\nilp \Jac(Q,W)) \simeq \Z Q_0$, and a stability condition is determined by the images $Z([S_v])$ of the vertex simples.
Moreover, $\cA$ is a \emph{finite length category}, meaning that any object has a finite composition series.
This finite length property implies the existence of \emph{Harder-Narasimhan} filtrations: if $\Uptheta$ is a phase function for a stability condition, then for any $M\in\cA$ there exists a \emph{unique} filtration
\[
  0 = M_0 \subset M_1 \subset \ldots M_n = M
\]
where the subquotients $M_i/M_{i-1}$ are semistable and the phases satisfy an inequality
\[
  \Uptheta(M_1/M_0) > \Uptheta(M_2/M_1) > \ldots \Uptheta(M_n/M_{n-1}).
\]
As shown by Reineke \cite{Reineke03}, the HN filtration induces a stratification of the moduli stack $\fC$ with strata indexed by the tuples $(\uptheta_1,\ldots,\uptheta_n)$ of phases of the semistable subquotients. This yields the following identity in the motivic Hall algebra:
\begin{equation}\label{eq:stabdecompMoHa}
  [\fC\to\fC] = [\fC_0\into \fC] + \sum_{n\in\N} \sum_{\uptheta_1 > \ldots > \uptheta_n} 
  [(\fC^{\uptheta_1}\setminus\fC_0) \into \fC]\star\cdots\star
  [(\fC^{\uptheta_n}\setminus\fC_0) \into \fC],
\end{equation}
where $\fC^\uptheta\subset \fC$ denotes the (open) substack of $\fC$ whose points correspond to objects in $\cA^\uptheta$ for a phase $\uptheta \in (0,\uppi]$.
For each phase $\uptheta$, the integration map sends the element $[\fC^\uptheta \into \fC]$ to 
a power series
\[
  \Upphi^\uptheta(t) \colonequals \int_{[\fC^\uptheta\into\fC]} \upphi_{\tr(W)}|_\fC = 
  \sum_{\updelta\in\Updelta} \int_{\fC_\updelta^\uptheta}\upphi_{\tr(W)} \cdot t^\updelta,
\]
and Reineke's identity \eqref{eq:stabdecompMoHa} translates to the factorisation identity which was first conjectured by Kontsevich--Soibelman \cite{KS08}.

\begin{lemma}\label{lem:decomp}
  For a symmetric quiver, the following equality holds in $\Mothat[[\Delta]]$:
  \begin{equation}\label{eq:stabilitydecomp}
    \Upphi(t) = \prod_{\uptheta \in (0,\pi]} \Upphi^\uptheta(t).
  \end{equation}
\end{lemma}
The decomposition \eqref{eq:stabdecompMoHa} depends only on the Harder-Narasimhan filtrations induced by the stability condition and not on the specific homomorphism $Z\colon \K_0(\nilp \Jac(Q,W)) \to \C$ chosen.
We therefore fix the following notion of equivalence.
\begin{define}\label{def:stabequivclass}
  Two stability conditions $Z,Z'\colon\K_0(\nilp \Jac(Q,W))\to\C$ are \emph{equivalent} if
  they induce the same Harder-Narasimhan filtration on every non-zero representation.
\end{define}
Not every choice of stability condition will give a good decomposition of the partition function.
For instance, the stability condition $Z\colon \K_0(\nilp \Jac(Q,W))\to \C$ that maps all modules onto a single ray with phase $\uptheta$ gives the trivial relation $\Upphi(t) = \Upphi^\uptheta(t)$.
The following genericity assumption guarantees that the decomposition \eqref{eq:stabilitydecomp} is optimal.
\begin{define}\label{def:generic}
  Let $Z\colon \K_0(\nilp \Jac(Q,W)) \to \C$ be a stability condition with $\Uptheta$ its phase function,
  then $Z$ is \emph{generic} if for every pair of $Z$-semistable representations $N,M$
  \[
    \Uptheta(N) = \Uptheta(M) \quad\Longleftrightarrow\quad \ddim N = q\cdot \ddim M \quad \text{for some } q\in \Q.
  \]
\end{define}
Let $Z$ be a generic stability condition, and $\uptheta$ a phase for which $\cA^\uptheta$ is nonzero.
Then the genericity implies that the dimension vectors of objects in $\cA^\uptheta$ are multiples of a common, indivisible dimension vector $\updelta \in \Updelta$.
The coefficients of $\Upphi^\uptheta(t)$ are zero for any dimension vector which is not a multiple of $\updelta$, so after comparing with the BPS ansatz in \eqref{eq:BPSansatz} one finds
\[
  \Upphi^\uptheta(t) = \Sym\left(\sum_{n\in \N} \frac{\BPS_{n \updelta}}{\L^\half - \L^\mhalf} \cdot t^{n\updelta}\right).
\]
In particular, this puts a restriction on the nonzero BPS invariants: $\BPS_{\updelta'} = 0$ if there is a stability condition for which $\updelta'$ is not a rational multiple of the dimension vector of a (semi)stable module.
Moreover, the functions $\Upphi^\uptheta(t)$ can often be computed via deformation theory.

\subsection{Formal non-commutative functions on a point}

In section \S\ref{sec:stables} we identify a stability condition and a set of phases for the quiver with potential of length 2 flops.
With one exception, there exists a unique stable module $M$ for each of these phases $\uptheta$.
In this setting the semistable locus $\fC^\uptheta$ parametrises the extensions of $M$, and the partition function is determined by the deformation theory of $M$: one has
\[
  \Upphi^\uptheta(t) = \Upphi_{\cQ_M,\cW_M}(t^{\ddim M}).
\]
for some potential $\cW_M$ on a ``non-commutative neighbourhood'' of $M$ described by an $N$-loop quiver $\cQ_M$.
The potential $\cW_M$ is defined, up to a \emph{formal} coordinate change, by a cyclic minimal $A_\infty$-structure on $\Ext^\bullet(M,M)$. 
We will prove a few results that allow us to work with formal coordinate changes, 
deferring the $A_\infty$-deformation theory for section \S\ref{sec:autoequiv}.

\begin{lemma}\label{lem:motivgerms}
  Let $f,g\colon Y \to \A^1$ be nonconstant regular functions on a smooth scheme, and $Z\subset Y$
  a closed subscheme with $X\supset Z$ a formal neighbourhood in $Y$.
  Suppose there exists an automorphism $t\colon X\to X$ that identifies the
  germs $f|_X \circ t = g|_X$, then
  \[
    \int_Z \upphi_f = \int_Z \upphi_g.
  \]
\end{lemma}
\begin{proof}
  By the construction of Denef--Loeser \cite{DL99}, the integral $\int_Z\upphi_f$ is the (well-defined) value at $T=\infty$ (see also \cite[\S 5]{DM15}) of a generating series
  \[
    \sum_{n\geq1} \L^{-(n+1)\dim Y/2}\left([(f_n|_Z)^{-1}(0)] - [(f_n|_Z)^{-1}(1)]\right) \cdot T^n,
  \]
  where the $f_n$ are lifts of $f$ to the arc spaces $\cL_n(Y)|_Z$ of $Y$ with support on $Z$, defined by the composition  
  \[
    f_n|_Z \colon \cL_n(Y)|_Z \xrightarrow{\cL_n(f)|_Z} \cL_n(\A^1) \simeq \A^n \xrightarrow{(z_1,\ldots,z_n) \mapsto z_n} \A^1.
  \]
  Every length $n$ arc with support on $Z$ can be identified with an arc in an $n$-fold thickening of $Z$ in $Y$. 
  The automorphism $t\colon X\to X$ restricts to an automorphism on such a finite thickening and hence induces an automorphism $t_n\colon \cL_n(Y)|_Z \to \cL_n(Y)|_Z$ on arc spaces satisfying $f_n|_Z \circ t_n = g_n|_Z$. 
  In particular, for any $n\in\N$ and $\uplambda=0,1$ one has
  \[
    [(f_n|_Z)^{-1}(\uplambda)] =  [t_n^{-1}((g_n|_Z)^{-1}(\uplambda))] = [(g_n|_Z)^{-1}(\uplambda)] \in \Mothat(\C).
  \]
  It follows that the generating series involving $f_n$ and $g_n$ are equal, and therefore their values $\int_Z\upphi_f$ and $\int_Z\upphi_g$ at $T=\infty$ agree.
\end{proof}

Let $(\cQ,\cW)$ be a quiver with potential and $I=(a\mid a\in \cQ_1)$ the two-sided ideal generated by its arrows.
Then the path algebra has an $I$-adic completion $\widehat{\C\cQ} = \lim_n\C\cQ/I^n$, and the potential has a well-defined noncommutative germ $\widehat{\cW} \in \widehat{\C\cQ}_\cyc \colonequals \lim_n (\C\cQ/I^n)_\cyc$. This data defines the \emph{completed Jacobi algebra}
\[
  \widehat\Jac(\cQ,\cW) \colonequals \frac{\widehat\C\cQ}{(\!(\del_a \widehat \cW \mid a\in \cQ_1)\!)},
\]
where the double braces denote the I-adic completion of the ideal.
If the noncommutative germs $\widehat{\cW}$ and $\widehat{\cW'}$ of potentials $\cW$, $\cW'$ are related by an $I$-adic automorphism of $\widehat{\C\cQ}$, then the completed Jacobi algebras are isomorphic and as the following lemma shows, this yields equivalent DT theories.

\begin{lemma}\label{lem:motintformal}  
  Let $\cQ$ be a quiver with potentials $\cW,\cW'\in (\C\cQ)_\cyc$, 
  Suppose there exists an $I$-adic automorphism $\uppsi\colon \widehat{\C\cQ}\to\widehat{\C\cQ}$
  such that $\uppsi(\widehat\cW) = \widehat{\cW'}$ then
  \[
    \Upphi_{\cQ,\cW}(t) = \Upphi_{\cQ,\cW'}(t).
  \]  
\end{lemma}
\begin{proof}
  Fix a dimension vector $\updelta$, and let $\{X^{(n)} \to X^{(m)}\}_{m\geq n}$ denote the directed system of subschemes $X^{(n)} \subset \Rep_\updelta(\cQ)$ defined by all powers $I^m$ of $I$.  
  Any cyclic path $a \in (\C\cQ/I^n)_\cyc$ has a well-defined trace $\tr(a)\colon X^{(n)} \to \A^1$, which satisfies
  \[
    \tr(\cW_n) = \tr(\cW)|_{X^{(n)}},
  \]
  for $\cW_n\in (\C\cQ/I^n)_\cyc$ the value of $\cW$ in the quotient.
  An endomorphism $\uppsi_n$ of $\C\cQ/I^n$ induces a map $t_n\colon X^{(n)}\to X^{(n)}$ such that $\tr(a)\circ t_n = \tr(\uppsi_n(a))$.
  In particular
  \[
    \tr(\cW)|_{X^{(n)}} \circ t_n = \tr(\cW_n)\circ t_n = \tr(\uppsi_n(\cW_n)).
  \]
  The $I$-adic isomorphism $\uppsi\in\End(\widehat{\C\cQ})$ consists of a compatible sequence $(\uppsi_n)_{n\geq1}$
  of isomorphisms of $\C\cQ/I^n$ for each $n$ such that $\uppsi_n(\cW_n) = \cW'_n$.
  Let $X$ be the colimit of the $X^{(n)}$, and let $t\colon X\to X$ be the isomorphism associated to the
  sequence $t_n\colon X^{(n)} \to X^{(n)}$ of isomorphisms induced by the $\uppsi_n$.
  Then for each $n$
  \[
    \tr(\cW)|_{X^{(n)}} \circ t_n = \tr(\uppsi_n(\cW_n)) = \tr(\cW_n') = \tr(\cW')|_{X^{(n)}},
  \]
  which shows that $\tr(\cW)|_X \circ t = \tr(\cW')|_X$. 
  Let $C_\updelta \subset \Rep_\updelta(Q)$ be the nilpotent part of the critical locus, i.e. $\fC_\updelta = C_\updelta/\GL_\updelta$.
  Then $X$ is a formal neighbourhood of $C_\updelta$, and it follows from Lemma \ref{lem:motivgerms} that
  \[
    \int_{\fC_\updelta} \upphi_{\tr(\cW)} 
    = \frac{\L^{\dim \GL_\updelta\kern-2pt/2}\int_{C_\updelta} \upphi_{\tr(\cW)}}{[\GL_\updelta]} 
    = \frac{\L^{\dim \GL_\updelta\kern-2pt/2}\int_{C_\updelta} \upphi_{\tr(\cW')}}{[\GL_\updelta]}
    = \int_{\fC_\updelta} \upphi_{\tr(\cW')} 
  \]
  The equality $\Upphi_{\cQ,\cW}(t) = \Upphi_{\cQ,\cW'}(t)$ follows by comparing coefficients for each $\updelta$.
\end{proof}

Using formal coordinate changes, the potential on an $N$-loop quiver $\cQ$ can be brought into a simplified standard form $\cW_\min + q$, where $\cW_\min$ is a sum of degree $\geq 3$ terms in arrows $x_i$ 
and $q$ is a non-degenerate quadratic form in a complimentary set of arrows $y_i$.
Such a quadratic form does not contribute\footnote{
  In general such a quadratic form encodes \emph{orientation data} on the moduli space. Here, our moduli space in question is a point, parametrising a single stable module, and the orientation data is immaterial.
} to the DT theory.

\begin{lemma}\label{lem:quadterms}
  Let $\cQ$ be an $N$-loop quiver with loops $\{x_1,\ldots,x_n,y_1,\ldots,y_{N-n}\}$ and suppose $\cW = \cW_\min + q \in \C Q_\cyc$ is a standard form potential as above. Then
  \[
    \Upphi_{\cQ,\cW}(t) = \Upphi_{\cQ_\min,\cW_\min}(t),
  \]
  where $\cW_\min$ is interpreted as a potential on the $n$-loop quiver $\cQ_\min$ with loops $x_1,\ldots,x_n$.
\end{lemma}
\begin{proof}
  For each $k\in\N$ the variety $\Rep_k(\cQ)$ decomposes as a product $\Rep_k(\cQ_\min) \times \A^m$ and $\tr(\cW)$ is the two terms in $\tr(\cW_\min) + \tr(q)$ restrict to the respective factors.
  The function $\tr(q)$ is a nondegenerate quadratic form in the usual sense, and without loss of generality we may assume that it is of the form $\tr(q) = z_1^2 + \ldots + z_m^2$.
  The Nicaise-Payne theorem implies that a function $z^2\colon \A^1\to \A^1$ has a normalised integral
  \[
    \int_{\A^1} \upphi_{z^2} = \L^\mhalf(1-[\upmu_2]) = \L^\mhalf\L^\half = 1,
  \]
  so it follows by the repeated application of the Thom-Sebastiani identity that
  \[
      \int_{\fC_{\cQ,k}} \upphi_{\tr(\cW_\min+q)} 
      = \int_{\fC_{\cQ_\min,k}} \upphi_{\tr(\cW_\min)} \cdot 
         \left(\int_{\A^1} \upphi_{z^1}\right)^m
      = \int_{\fC_{\cQ_\min,k}}\upphi_{\tr(\cW_\min)}.
  \]
  The equality $\Upphi_{\cQ,\cW}(t) = \Upphi_{\cQ_\min,\cW_\min}(t)$ then follows by comparing coefficients.
\end{proof}

\subsection{Intermediary refinements}

The motivic theory described so far is a refinement of the enumerative Donaldson--Thomas theory of Joyce--Song \cite{JS12}, in terms of partition function with rational coefficients. 
This partition function can be expressed in terms of \emph{integer} BPS invariants via a multiple-cover formula.
The motivic BPS invariants similarly lie in an ``integral'' subring $\K^\muhat(\Var/\C)[\L^\mhalf] \subset \Mothat$ (see \cite[Conjecture 6.5, Corollary 6.25]{DM15a}), and are related to the BPS numbers via the Euler characteristic
\[
  \upchi \colon \K^\muhat(\Var/\C)[\L^\mhalf] \to \Z.
\]
There are various alternative refinements between $\K^\muhat(\Var/\C)[\L^\mhalf]$ and $\Z$, which are more closely related to vanishing cycle cohomology.
Following \cite{Davison19}, we will consider the following hierarchy of invariant rings
\[
  \K^\muhat(\Var/\C)[\L^\mhalf]
  \xrightarrow{\upchi_\mmhs} \K_0(\MMHS) 
  \xrightarrow{\upchi^\mmhs_\hsp} \Z[u^{\pm\kern-1pt\frac 1n},v^{\pm\kern-1pt\frac 1n} \mid n\in\N]
  \xrightarrow{\upchi^\hsp_\wt} \Z[q^{\pm\kern-1pt\half}]
  \xrightarrow{\upchi^\wt} \Z.
\]
Here $\K_0(\MMHS)$ is the Grothendieck ring of the category of \emph{monodromic mixed Hodge structures},
and the map $\upchi_\mmhs$ assigns to a class $[X]$ the class
\[
  \upchi_\mmhs([X]) = -[\H_c(X,\Q)],
\]
of the mixed Hodge structure on the compactly supported cohomology, with a monodromy induced by the monodromy on $X$. 
The map $\upchi^\mmhs_\hsp$ assigns to each MMHS its equivariant Hodge polynomial: if $H$ is a pure Hodge structure of dimension $d$ with an action of $\upmu_n$, then its Hodge spectrum is
\[
  \upchi_\hsp^\mmhs(H) = 
  \sum_{p+q=d} (-1)^d \dim_\C H_\C^{p,q,0} u^pv^q + \sum_{a \neq 0} \sum_{p+q=d} (-1)^d \dim_\C H^{p,q,a}_\C u^{p+\frac an}v^{q+\frac{n-a}n} 
\]
where $\bigoplus_{p+q=d} H^{p,q}_\C \simeq H_\C$ is the Hodge decomposition and  $H^{p,q,a}_\C \subset H^{p,q}_\C$ is the subspace on which $\upmu_n$ acts with weight $a$.
The map $\upchi^\hsp_\wt$ assigns the weight-polynomial 
\[
  \upchi^\hsp_\wt(h(u,v)) = h(q^\half,q^\half),
\]
and the map $\upchi^\wt$ evaluates the weight-polynomial at $q^\half=1$.
Composing this chain of maps recovers the Euler characteristic.
We will find all these intermediate invariants for length 2 flops in section \S\ref{sec:DTcalc}.

\section{Classification of Stable Modules}\label{sec:stables}
Let $\uppi\colon Y\to Y_\con = \Spec R$ be a simple flopping contraction of length $\ell=2$ over a complete local ring $(R,\o)$ as in \S\ref{sec:flops}. Let $\cP = \O_Y \oplus \cN$ denote Van den Bergh's tilting bundle and
\[
  \Uppsi\colon \D^b(\Coh Y) \to \D^b(\mod \Nccr),
\]
the associated tilting equivalence onto the derived category of $\Nccr \colonequals \End_Y(\cP)$. Note that $\Nccr$ has precisely two indecomposable projectives $P_0$, $P_1$ which correspond to the summands $\O_Y$, $\cN$ respectively. 
The algebra has precisely two simple modules $S_0$, $S_1$, of which the iterated extensions generate the subcategory $\fdmod\Nccr \subset \mod \Nccr$.
The goal of this section is to classify the stable objects in $\fdmod\Nccr$ for a suitable stability condition.

Our approach is as follows. In \S\ref{ssec:king} we firstly parametrise stability conditions by a linear parameter in the space $\K_0(\proj \Nccr)_\R = \K_0(\proj \Nccr) \otimes_\Z \R$, and in \S\ref{ssec:tiltingflop} we construct a hyperplane arrangement inside this space that expresses the tilting theory of $\Nccr$.
In \S\ref{ssec:tiltingsilting} we show that both the hyperplane arrangment and the set of stable objects are preserved when passing to a finite dimensional quotient $\Nccr\to \Nccr/I$, for which it is known that the two are related by a duality.
Using this reduction method, we are then able to give a complete classification of the stable modules in \S\ref{ssec:stables}.

\subsection{King stability}\label{ssec:king}

Given a finite type $\C$-algebra $A$, there is a well-define Euler pairing 
\[
  \<-,-\>\colon \K_0(\proj A)_\R \otimes_\Z \K_0(\fdmod A) \to \R,
\]
which is determined by its value on effective classes: for $P\in \proj A$ and $M\in\fdmod A$
\[
  \<[P],[M]\> = \dim_\C \Hom_A(P,M).
\]
The Euler pairing is known to be nondegenerate if $A$ is finite-dimensional or finite as an algebra over a complete local ring.
Via the pairing, any element $v\in \K_0(\proj A)_\R$ induces a stability condition $Z_v$ on $\fdmod A$ via
\[
  Z_v([M]) = \dim_\C M \cdot i - \<v,[M]\>.
\]
These stability conditions are closely related to \emph{King}-stability conditions \cite{King94}, as a module $M$ is King-(semi)stable for a parameter $v\in \K_0(\proj A)_\R$ if and only if it is $Z_v$-(semi)stable and $\<v,[M]\> = 0$.
We therefore refer to $\K_0(\proj A)_\R$ as the space of King-stability parameters and write
\[
  \cS_v(A) \colonequals \left\{M \in \fdmod A\mid M=0 \text{ or } M \text{ is } Z_v\text{-semistable with } \<v,[M]\> = 0\right\},
\]
for the subcategory of King-semistable $A$-modules (including $0$) for the parameter $v$.

The algebra $\Nccr$ is finite over a complete local ring, and $\<-,-\>$ defines a nondegenerate pairing between the space $\K_0(\proj \Nccr)_\R \simeq \R^2$, with basis $[P_0],[P_1]$, and the space $\K_0(\fdmod \Nccr) \simeq \Z^2$, with dual basis $[S_0],[S_1]$. The possible generic King stability conditions are as follows.

\begin{lemma}\label{lem:vstab}
  For parameters $v = v_0[P_0] + v_1[P_1]$, the equivalence class of $Z_v$ (in the sense of \ref{def:stabequivclass}) is uniquely determined by the sign of $v_0 - v_1$.
  Moreover, $Z_v$ is generic (in the sense of \ref{def:generic}) if and only if $v_0-v_1 \neq 0$.
\end{lemma}
\begin{proof}
  Let $v$ be a parameter, and let $\Uptheta$ denote the slope function of the associated stability condition $Z_v$. 
  Given two nonzero classes $a = a_0[S_0] + a_1[S_1]$ and $b = b_0[S_0] + b_1[S_1]$, one has $\Uptheta(a) < \Uptheta(b)$ if and only if the vectors $Z_v(a), Z_v(b)$ span a parallelogram in $\C\simeq \R^2$ for which the signed area is strictly positive.
  By inspection, the area is given by
  \[
    (a_0+a_1)(-b_0v_0-b_1v_1) - (-a_0v_0-a_1v_1)(b_0+b_1) = (a_0b_1-a_1b_0)(v_0-v_1),
  \]
  and for fixed $a$ and $b$ its sign depends only on the sign of $v_0-v_1$.
  In particular, if $v'$ is a second parameter with $v_0'-v_1'$ of the same sign, then $Z_v$ and $Z_{v'}$ are equivalent, and the converse follows by considering the case $a=[S_0]$, $b=[S_1]$.

  For the second statement, note that $v_0 \neq v_1$ implies that the map $Z_v \colon \R^2 \to \C$ is an isomorphism.
  Hence, if $v_0\neq v_1$ and $Z_v(a) = r\cdot Z_v(b)$ for some classes $a,b$ and $r\in \R_{>0}$, it follows that $a = rb$, so that $Z_v$ is generic.
  Conversely, for $v_0 = v_1$ one sees that $Z_v([S_0]) = Z_v([S_1])$, so $Z_v$ is not generic.
\end{proof}

Consider an ideal $I\subset \o$ such that $R/I$ is artinian. Then the NCCR $\Nccr$ has a \emph{fibre} over the thick point $\Spec R/I$, given by the finite dimensional algebra
\[
  \Nccr/I\Nccr \colonequals \Nccr \otimes_R R/I.
\]
Extension and restriction of scalars defines a pair of adjoint functors
\[
  -\otimes_\Nccr \Nccr/I\Nccr\colon \mod \Nccr \rightleftarrows \mod \Nccr/I\Nccr\cocolon (-)_\Nccr.
\]
Because $I$ is contained in the radical, $-\otimes_\Nccr \Nccr/I\Nccr$ preserves and reflects projectives, while $(-)_\Nccr$ preserves/reflects simples. 
In particular, there are isomorphisms
\[
  \begin{alignedat}{4}
    \upzeta\colon &\K_0(\proj \Nccr)_\R \to \K_0(\proj \Nccr/I\Nccr)_\R,\quad
    [(-)_\Nccr] \colon& \K_0(\fdmod \Nccr/I\Nccr) \to \K_0(\fdmod \Nccr),
  \end{alignedat}
\]
which are adjoint with respect to the Euler pairing $\<-,-\>$. 
The first isomorphism identifies King-stability parameters for $\Nccr$ and $\Nccr/I\Nccr$, and the following lemma shows that the second identifies the dimension vectors of stable modules.\footnote{This same result was observed in {\cite{DM17}} and used to compute stable modules for length 1 flops.}
\begin{prop}\label{prop:quotcentral}
  Let $v\in \K_0(\proj \Nccr)_\R$, then $(-)_\Nccr$ identifies $\upzeta(v)$-stable $\Nccr/I\Nccr$-modules with $v$-stable $\Nccr$-modules. 
  In particular
  \[
    \cS_v(\Nccr) = \<\cS_{\upzeta(v)}(\Nccr/I\Nccr)_\Nccr\>,
  \]
  where $\<-\>$ denotes the extension closure.
\end{prop}
\begin{proof}
  Let $\Uptheta_v$ and $\Uptheta_{\upzeta(v)}$ denote the phase functions of $Z_v$ and $Z_{\upzeta(v)}$.
  The exact functor $(-)_\Nccr$ embeds $\fdmod \Nccr/I\Nccr$ into $\fdmod \Nccr$ as a Serre subcategory in $\mod \Nccr$. Hence, for any module $N\in \fdmod\Nccr/I\Nccr$ the submodules of its image $N_\Nccr$ are precisely the images of its of submodules. It moreover follows from the adjunction that $Z_v((-)_\Nccr) = Z_{\upzeta(v)}(-)$:
  \[
    Z_v([N_\Nccr]) = \dim_\C N_\Nccr \cdot i + \<v,[N_\Nccr]\> = \dim_\C N\cdot i + \<\upzeta(v),[N]\> =  Z_{\upzeta(v)}([N]).
  \]
  for all $N\in\fdmod \Nccr/I\Nccr$. Hence $N_\Nccr$ is King (semi)stable for $v$ if and only if $N$ is King (semi)stable for $\upzeta(v)$, and the functor $(-)_\Nccr$ restricts to an exact embedding
  \[
    (-)_\Nccr \colon \cS_{\upzeta(v)}(\Nccr/I\Nccr) \to \cS_{\upzeta(v)}(\Nccr/I\Nccr)_\Nccr \subset \cS_v(\Nccr),
  \]
  By the finite length property, $\cS_v(\Nccr)$ is generated via extension by its stable modules, so
  it suffices to show that any stable module in $\cS_v(\Nccr)$ is in the image of $\cS_{\upzeta(v)}(\Nccr/I\Nccr)$. Suppose $M \in \fdmod \Nccr$ is $Z_v$-stable and let $c\in I$. Because $c$ is central in $\Nccr$ it induces an endomorphism $f\colon M\to M$. The submodule $\im f \subset M$ satisfies $\Uptheta_v([\im f]) \leq \Uptheta_v([M])$ by semistability, and because $\im f$ is also a quotient 
  \[
    0 \to \ker f \to M \to \im f \to 0,
  \]
  it follows that $\Uptheta_v([\im f]) = \Uptheta_v([M])$. Hence $\im f = M$ or $\im f = 0$. Because $I$ is contained in the radical $\o\subset R$ and $c \in I$, it follows from Nakayama's lemma that $\im f = cM \neq M$, which implies that $f$ acts trivially on $M$. It follows that $M\simeq (M/I M)_\Nccr$ lies in the image of $(-)_\Nccr$, which finishes the proof.
\end{proof}

\subsection{Tilting theory of the length 2 flop}
\label{ssec:tiltingflop}

We recall some terminology regarding tilting complexes of algebras.

\begin{define}
  Let $A$ be an algebra for which the homotopy category $\cK^b(\proj A)$ of bounded complexes of 
  projectives is Krull-Schmidt. 
  Then a complex $T \in \cK^b(\proj A)$ is 
  \begin{itemize}
  \item \emph{basic} if its Krull-Schmidt decomposition has no repeated summands,
  \item a \emph{2-term complex} if $T$ is concentrated in degrees $-1$ and $0$,
  \item \emph{partial tilting} if $\Ext^i_A(T,T) = 0$ for all $i\neq 0$,
  \item \emph{tilting} if it is partial tilting and $T$ generates $\cK^b(\proj A)$ as a triangulated category.
  \end{itemize}
  The set of basic 2-term tilting complexes is denoted $\tilt A$. A module $M\in \mod A$ is a \emph{classical tilting module} if it has a 2-term projective resolution which is in $\tilt A$.
\end{define}
The tilting theory of NCCRs for Gorenstein threefold singularities is now well understood \cite{IR08,IW14,IW14a,HomMMP}.
Let $\refl R$ denote the set of reflexive $R$-modules, then \cite[Theorem 1.4]{IW14a} shows that any NCCR of $R$ is isomorphic to $\End_R(M)$ for some $M \in \refl R$.
Moreover, the different NCCRs are connected by tilting modules $\Hom_R(M,N)$, and the functor $\Hom_R(M,-)$ defines a bijection
\begin{equation}\label{eq:maxmodtilt}
    \left\{
      \,
      \begin{gathered}
        N \in \refl R \text{ such that}\\
        \End_R(N) \text{ is an NCCR}
      \end{gathered}
      \,
    \right\}
    \xrightarrow{\quad\sim\quad} 
    \left\{\,
      \begin{gathered}
        \text{classical tilting modules}\\
        \text{in } \refl\,\End_R(M)
      \end{gathered}\,
    \right\},
\end{equation}
where $\refl\, \End_R(M)$ denotes the set of $\End_R(M)$-modules that are reflexive over $R$.

The correspondence \eqref{eq:maxmodtilt} applies in particular to the ring $\Nccr$, which is the NCCR defined by the image $M_0 = \uppi^*\cP = \uppi^*\O_Y \oplus \uppi^*\cN$ of the Van den Bergh tilting bundle in $\mod R$.
All other NCCRs can be obtained via mutation of this reflexive module (see \cite[\S6]{IW14}), and a complete classification was obtained by Donovan--Wemyss \cite{DW19}.
By \cite[Theorem 5.9]{DW19} the NCCRs form a sequence $\Nccr_i = \End_R(M_i)$ corresponding to the reflexive modules $M_i = V_i \oplus V_{i+1}$ defined by the twists
\[
  V_{2k} \colonequals \uppi_*\O_Y(k),\quad
  V_{2k+1} \colonequals \uppi_*\cN(k).
\]
The bijection in \eqref{eq:maxmodtilt} relates the NCCRs $\Nccr_i$ to our distinguished NCCR $\Nccr_0 \simeq \Nccr$, via the tilting modules $\Hom_R(M_0,M_i)$.
Hence, the minimal projective resolutions
\[
  T_k \onto \Hom_R(M_0,V_k),
\]
of their summands are partial tilting complexes in $\cK^b(\proj \Nccr)$.
Dually, the modules $\Hom_R(M_i,M_0)$ are tilting in $\cK^b(\proj \Nccr^\op)$ with endomorphism
algebra $\Nccr_i^\op$.
Let $F_i\in \cK^b(\proj \Nccr^\op)$ denote the minimal projective resolutions 
\[
  F_{2k} \onto \Hom_R(V_k,M_0),  
\]
then the shifted duals $E_i = (F_i)^*[1] \in \cK^b(\proj \Nccr)$ are again partial tilting complexes.

\begin{lemma}
  The complexes $T_{i-1} \oplus T_i$ and $E_{i-1} \oplus E_i$ are in $\tilt \Nccr$ for all $i\in\Z$.
\end{lemma}
\begin{proof}
  Because the tilting module  $\Hom_R(M_i,M_0)$ is reflexive, it follows from the generalised Auslander-Buchsbaum formula \cite[Lem. 2.16]{IW14}
  that $\Hom_R(M_i,M_0)$ has projective dimension $\leq 1$.
  Hence its minimal resolution $T_{i-1} \oplus T_i$ is a 2-term tilting complex, which is basic because
  \[
    \End_{\D^b(\Nccr)}(T_{i-1}\oplus T_i) \simeq \Nccr_i = \End_R(M_i),
  \]
  is a basic algebra.
  By \cite{IR08} the dual $M_i^*$ of $M_i$ defines an NCCR 
  \[
    \End_R(M_i^*) \simeq \End_R(M_i)^\op = \Nccr_i^\op,
  \]
  for each $i$ and $\Hom_R(M_0^*,M_i^*) \simeq \Hom_R(M_i,M_0)$ is a tilting $\Nccr^\op$-module.
  By a similar argument, $F_{i-1} \oplus F_i$ is a basic 2-term tilting complex in $\cK^b(\proj \Nccr^\op)$.
  By \cite[Cor. 3.4]{IR08}, the $R$-linear dual $(-)^*$ defines an exact duality
  \[
    (-)^* \colon \cK^b(\proj \Nccr^\op) \leftrightarrows \cK^b(\proj \Nccr)\cocolon (-)^*,
  \]
  which implies $E_{i-1} \oplus E_i = (F_{i-1}\oplus F_i)^*[1]$ is a basic 2-term tilting complex.
\end{proof}

\begin{figure}
  \begin{tikzpicture}
    \node[fill=black,circle,inner sep=1pt] at (2,0) {};
    \node at (2,.3) {$[P_0]$};
    \node[fill=black,circle,inner sep=1pt] at (4,0) {};
    \node at (4,.3) {$2[P_0]$};
    \node[fill=black,circle,inner sep=1pt] at (0,2) {};
    \node at (.35,2) {$[P_1]$};
    \begin{scope}
      \clip (-4,-2) rectangle (4,2);
      \draw (0,0) to (8,0);
      \foreach \i in {0,1,2,3,4,5}
      {
        \draw   (16*\i,   -8*\i-4) to (-16*\i , 8*\i+4);  %
        \draw   (16*\i,   -8*\i+4) to (-16*\i , 8*\i-4);  %
        \draw   (12*\i-6, -6*\i) to (-12*\i+6 ,6*\i);  %
        \draw   (12*\i+6, -6*\i) to (-12*\i-6, 6*\i);  %
      }
      \foreach \i in {6,7,8,9,10,11,12,13,14,15,16,17,18,19,20,21,22,23,24,25,26,27,28,29,30,35,40,45,50,55,60,70,80,90}
      {
        \draw   (4*\i,   -2*\i-1) to (-4*\i , 2*\i+1);  %
        \draw   (4*\i,   -2*\i+1) to (-4*\i , 2*\i-1);  %
        \draw   (2*\i-1, -\i) to (-2*\i+1, \i);  %
        \draw   (2*\i+1, -\i) to (-2*\i-1, \i);  %
      }
      \draw[color=white] (12, -6) to (-12,6);  %
    \end{scope}
  \end{tikzpicture}
  \caption{Wall-and-chamber structure of the $\ell=2$ flop.}
  \label{fig:siltingfan}
\end{figure}
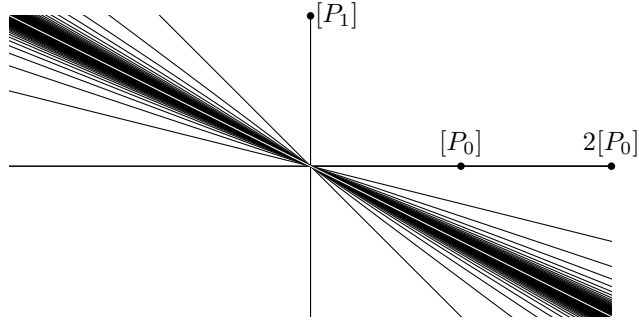

For a basic complex $U\in \cK^b(\proj \Nccr)$ with decomposition $U = U_1\oplus \ldots\oplus U_n$, the indecomposable summands define \emph{g-vectors} $[U_j] \in \K_0(\proj \Nccr)_\R$, which span a cone
\[
  \cone(U) \colonequals \{\textstyle\sum_i \uplambda_i\cdot [U_i] \mid  \uplambda_i \geq 0\} \subset \K_0(\proj \Nccr)_\R.
\]
If $U \in \tilt \Nccr$, then by \cite[Theorem 2.8]{AI12} the g-vectors of $U$ form a basis of $\K_0(\proj \Nccr)_\R \simeq \R^2$, and the interior $\cone^\circ(U)$ is therefore a non-empty open subspace of $\K_0(\proj \Nccr)_\R$.
The cones therefore determine a wall-and-chamber structure in $\K_0(\proj \Nccr)_\R$ with walls correspond to the partial tilting complexes $T_i$ and chambers corresponding to the full tilting complexes $T_i \oplus T_{i+1}$.
As the following lemma shows, this wall-and-chamber structure is the hyperplane arrangement of figure \ref{fig:siltingfan}.

\begin{lemma}\label{lem:gvects}
  The g-vectors of the complexes $T_i$ are
  \begin{align*}
    [T_i]   &=
              \begin{cases}
                [P_0] + n \cdot(2[P_0] - [P_1]) & \text{if } i = 2n\\
                [P_1] + 2n \cdot (2[P_0] - [P_1]) & \text{if } i = 2n-1.
              \end{cases}
  \end{align*}
  and the complexes $E_i$ have g-vectors  $[E_i] = -[T_i]$.
\end{lemma}
\begin{proof}
  As shown in \cite[\S3.3]{DW19}, there is an isomorphism $\Nccr_{2n} \xrightarrow{\sim} \Nccr$ for all $n\in \Z$, and the composition with the tilt 
  \[
    \K_0(\proj \Nccr) \xrightarrow{[\RHom_{\Nccr_{2n}}(\Hom_R(M_0,M_{2n}),-)]} \K_0(\proj \Nccr_{2n}) \xrightarrow{\ \sim\ } \K_0(\proj\Nccr),
  \]
  maps the class $[T_{2n}]$ to $[P_0]$ and the class $[T_{2n-1}]$ to $[P_1]$.
  By \cite[Theorem 7.4, Lemma 7.6]{HW19} this isomorphism can be presented in the basis $[P_0], [P_1]$ as the matrix
  \begin{equation}\label{eq:Kmatrix}
    \begin{pmatrix}
      -1 & -4 \\
      1 & 3
    \end{pmatrix}^n
    = 
    \begin{pmatrix}
      1 - 2n & -4n \\
      n & 1 + 2n
    \end{pmatrix}.
  \end{equation}
  The g-vectors of $T_{2n}$ and $T_{2n-1}$ can then be computed from the inverse:
  \[
    [T_{2n}] = (1+2n)[P_0] - n[P_1],\quad
    [T_{2n-1}] = 4n[P_0] + (1-2n) [P_1].
  \]
  Likewise, each tilting module $\Hom_R(M_{2n},M_0)$ defines an isomorphism
  \[
    \K_0(\proj \Nccr) \xrightarrow{\upepsilon^{-1}} \K_0(\proj \Nccr_{2n}) \xrightarrow{[\RHom_{\Nccr_{2n}}(\Hom_R(M_{2n},M_0),-)]} \K_0(\proj \Nccr),
  \]
  which maps $[P_0]$ to $[F_{2n}^*]$ and $[P_1]$ to $[F_{2n-1}^*]$. 
  This isomorphism can also be presented as the inverse of the matrix \eqref{eq:Kmatrix}
  by \cite[Rem. 7.5]{HW19}, hence
  \[
    [E_i] = -[F_i^*] = -[T_i].\qedhere
  \]
\end{proof}

\subsection{From tilting to silting on the fibre}
\label{ssec:tiltingsilting}

In \cite{BST19} and \cite{Asai21} it is shown how to recover the subcategories $\cS_v(A)$ of semistable modules over a finite dimensional algebra $A$ using \emph{silting theory}.
\begin{define}
  Let $A$ be an algebra for which the homotopy category $\cK^b(\proj A)$ is Krull-Schmidt. Then a complex $U \in \cK^b(\proj A)$ is called
  \begin{itemize}
    \item \emph{pre-silting} if $\Hom_{\cK^b(\proj A)}(U,U[i]) = 0$ for $i>0$,
    \item \emph{silting} if it is pre-silting and generates $\cK^b(\proj A)$ as a triangulated category.
  \end{itemize}
  The set of isomorphism classes of basic 2-term silting complex is denoted $\silt A$.
\end{define}
Clearly, the set $\tilt \Nccr$ of tilting complexes is contained in $\silt \Nccr$, so that silting is suitable generalisation.
The set $\silt\Nccr$ is moreover partially ordered: one considers $U \leq V$ if and only if $\Hom_{\cK^b(\proj A)}(U,V[i]) = 0$ for all $i>0$.

To apply the results of \cite{BST19} and \cite{Asai21} to our geometric setting, we will relate the silting theory of $\Nccr$ with that of a finite dimensional fibre $\Nccr/I\Nccr$.
\begin{prop}\label{prop:injsilt}
  There exists an ideal $I\subset \o$ for which $\Nccr/I\Nccr$ is finite dimensional, such that the functor $-\otimes_R R/I\colon \cK^b(\proj \Nccr) \to \cK^b(\proj \Nccr/I\Nccr)$ induces a map of posets
  \[
    \silt \Nccr \to \silt \Nccr/I\Nccr
  \]
\end{prop}
\begin{proof}
  Because $R$ is a Gorenstein local of dimension $3$, the maximal ideal $\o$ contains an ideal $I \subset \o$ generated by a regular sequence $g_1,g_2,g_3 \in I$. 
  The quotient $R/I$ is artinian and therefore $\Nccr/I\Nccr \simeq \Nccr\otimes_R R/I$ is finite dimensional.
  Because $\Nccr$ is an NCCR, it is a maximal Cohen-Macaulay $R$-module, which implies that $g_1,g_2,g_3$ is also a regular sequence for any projective $\Nccr$-module.
  If $U = U^1 \to U^0$ is a basic 2-term chain complex of projectives which is silting in the homotopy category $\cK^b(\proj \Nccr)$, then there are induce short exact sequences in chain complexes:
  \begin{equation}\label{eq:siltsliceses}    
    \begin{tikzpicture}
      \node (A1) at (1,0) {$0$};
      \node (A2) at (3,0) {$U/I_{k-1}U$};
      \node (A3) at (6,0) {$U/I_{k-1}U$};
      \node (A4) at (9,0) {$U/I_kU$};
      \node (A5) at (11,0) {$0$};
      \draw[->] (A1) to (A2);
      \draw[->] (A2) to[edge label=$\scriptstyle g_k$] (A3);
      \draw[->] (A2) to (A3);
      \draw[->] (A3) to (A4);
      \draw[->] (A4) to (A5);
    \end{tikzpicture}
  \end{equation}
  where the successive quotients by $I_k = (g_1,\ldots,g_k)$ slice down to yield a
  2-term complex of projectives over the finite dimensional algebra $\Nccr/I\Nccr$.
  Applying $\Hom_{\D(\Nccr)}(U,-)$ yields the following long exact sequence in cohomology:
  \[
    \begin{tikzpicture}
      \node (A1) at (-7,0) {$\ldots$};
      \node (A2) at (-3,0) {$\Hom_{\D(\Nccr)}(U,U/I_{k-1}U[i])$};
      \node (A3) at (2.5,0) {$\Hom_{\D(\Nccr)}(U,U/I_kU[i])$};
      \node (B1) at (-4.75,-.8) {$\Hom_{\D(\Nccr)}(U,U/I_{k-1}U[i+1])$};
      \node (B2) at (1.25,-.8) {$\Hom_{\D(\Nccr)}(U,U/I_{k-1}U[i+1])$};
      \node (B3) at (5,-.8) {$\ldots$};
      \draw[->] (A1) to (A2);
      \draw[->] (A2) to (A3);
      \draw[->] (A3.east) to ([xshift=.5]A3.east) arc (90:-90:.2) to (-7.55,-.4) arc (90:270:.2) to (B1.west);
      \draw[->] (B1) to[edge label=$\scriptstyle g_k$] (B2);
      \draw[->] (B2) to (B3);
    \end{tikzpicture}
  \]
  Because $U$ is silting, $\Hom_{\D(\Nccr)}(U,U[i]) = 0$ for $i>0$ and it also follows by induction on $k$ that also $\Hom_{\D(\Nccr)}(U,U/I_kU[i]) = 0$ for all $i>0$.
  It follows by adjunction that
  \[
    \Hom_{\D(\Nccr/I \Nccr)}(U/I  U,U/I  U[i]) \simeq \Hom_{\D(\Nccr)}(U,U/I  U[i]) = 0 \quad \forall i>0,
  \]
  making $U/I  U$ a 2-term pre-silting complex in $\cK^b(\proj \Nccr/I \Nccr)$, and that the map $-\otimes_R R/I \colon \End_\Nccr(U) \to \End_{\Nccr/I\Nccr}(U \otimes_R R/I)$ induces an algebra isomorphism
  \begin{equation}\label{eq:endoalg}
    \End_{\D(\Nccr)}(U)/I \End_{\D(\Nccr)}(U) \xrightarrow{\sim} \End_{\D(\Nccr/I \Nccr)}(U/I  U).
  \end{equation}
  Because $\End_{\D(\Nccr)}(U)$ is a complete algebra and $I$ is contained in the radical, 
  it follows that idempotents lift over the quotient $R\to R/I$.
  Hence, any indecomposable summand of $U$ remains indecomposable in the quotient $U/I U$.
  Because $U$ is a basic 2-term silting complex, it has exactly $\rk\K_0(\Nccr) = 2$
  indecomposable summands, and therefore $U/I U$ is a basic presilting complex with $2$ indecomposable summands.
  By \cite[Proposition 3.3]{AIR14} a basic presilting complex for a finite dimensional algebra
  is silting if and only if it has the maximal number of indecomposable summands.
  Hence $U/I U$ is in fact \emph{silting}, because $\Nccr/I\Nccr$ is finite dimensional.

  The above shows that $-\otimes_R R/I$ restricts to a map $\silt \Nccr \to \silt\Nccr/I\Nccr$, which we claim to be a morphism of posets. To see this, consider $U,V\in \silt\Nccr$ with $V\geq U$, then by applying $\Hom_{\D(\Nccr)}(V,-)$ to the short exact sequence \eqref{eq:siltsliceses} one sees that 
  \[
    \Hom_{\D(\Nccr/I\Nccr)}(V/I,U/IU[i]) \simeq \Hom_{D(\Nccr)}(V,U/I U[i]) = 0 \quad \forall i>0
  \]
  which shows that $V/IV \geq U/IU$ in $\silt \Nccr/I\Nccr$ as claimed.
\end{proof}
\begin{remark}
  In independent work by Kimura \cite{Kimura20}, which appeared while writing this paper, 
  it is shown that the above map is a bijection in a much more general setting.
\end{remark}
Using the map $\silt \Nccr \to \Nccr/I\Nccr$, the results of \cite{BST19} and \cite{Asai21} now yield the following.
\begin{prop}\label{prop:emptychambers}
  Let $U = U_1 \oplus U_2\in\silt\Nccr$, then for any stability parameter
  \begin{itemize}
  \item $v\in \cone^\circ(U)$ the subcategory $\cS_v(\Nccr)$ is trivial, and for
  \item $v\in \cone^\circ(U_i)$ the subcategory $\cS_v(\Nccr)$ contains a unique stable module.
  \end{itemize}
\end{prop}
\begin{proof}
  It follows from Proposition \ref{prop:injsilt} that $U/I U \in \silt \Nccr/I \Nccr$ with g-vectors
  \[
    [U_i/I U_i] = \upzeta([U_i]) \in \K_0(\proj\Nccr/I\Nccr).
  \]
  If $v$ lies in $\cone^\circ(U_i)$ then $\upzeta(v)$ lies in $\cone^\circ(U_i/I U_i)$, so
  it follows from \cite[Theorem 1.1]{BST19} that $\cS_{\upzeta(v)}(\Nccr/I \Nccr)$ contains
  a unique stable module $N$. By Proposition \ref{prop:quotcentral} 
  \[
    \cS_v(\Nccr) = \<N_\Nccr\>.
  \]
  so that $N_\Nccr$ is the unique stable module in $\cS_v$. Likewise, if $v \in \cone^\circ(U)$, then \cite[Theorem 1.1]{BST19} implies $\cS_{\upzeta(v)}(\Nccr) = 0$ and hence $\cS_v(\Nccr) = 0$ is trivial.
\end{proof}
Suppose $U,V \in \silt\Nccr$ share a summand $U_1 = V_1$ and $U > V$, then as in \cite{AIR14} the larger silting complex $U$ is called the Bongartz completion of $U_1$.
\begin{prop}\label{prop:uniquewalls}
  Suppose $U \in\silt\Nccr$ is the Bongartz completion of a summand $U_1$, 
  then $\Hom_{\D^b(\Nccr)}(U,-)$ restricts to an abelian equivalence
  \[
    \cS_{[U_1]}(\Nccr) \xrightarrow{\sim} \fdmod \End_{\D^b(\Nccr)}(U)/(e),
  \]
  where $(e)$ denotes the two-sided ideal of the idempotent $e\colon U \to U_1 \to U$.
\end{prop}
\begin{proof}
  Let $M \in \cS_{[U_1]}(\Nccr)$ be the unique stable module, then $M = N_\Nccr$ 
  for some stable module $N\in\cS_{[U_1/\o U_1]}(\Nccr/\o\Nccr)$ by proposition \ref{prop:quotcentral}.
  By proposition \ref{prop:injsilt} the complex $U/\o U$ is in $\silt\Nccr/\o\Nccr$ and is
  the Bongartz completion of $U_1/\o U_1$.
  Because $\Nccr/\o\Nccr$ is finite dimensional, the silting version \cite[Prop. 4.1]{Asai21} of \cite[Thm. 1.1]{BST19} then implies that
  \[
    \Hom_{\D^b(\Nccr/\o\Nccr)}(U/\o U, N[i]) = 
    \begin{cases}
      S & \text{if } i = 0,\\
      0 & \text{otherwise}.
    \end{cases}
  \]
  where $S$ is the simple $\Gamma' \colonequals \End_{\D^b(\Nccr)}(U/\o U)$-module that is killed by the idempotent $e'\colon U/\o U \to U_1/\o U_1 \to U/\o U$. 
  By \eqref{eq:endoalg} the algebra $\Gamma'$ is a quotient of $\Gamma \colonequals \End_{\cK^b(\proj \Nccr)}(U)$ by a radical ideal, hence
  $S$ restricts to a simple $S_\Gamma$ and $e'$ lifts to the idempotent $e \colon U \to U_1 \to U$. By adjunction, 
  \[
  \Hom_{\D^b(\Nccr)}(U,M[i]) =
  \Hom_{\D^b(\Nccr/\o\Nccr)}(U/\o U, N[i])_\Gamma =
    \begin{cases}
      S_\Gamma & \text{if } i = 0,\\
      0 & \text{otherwise}.
    \end{cases}
  \]
  Because $\cS_{[U_1]}(\Nccr)$ is generated by its stable modules and $\fdmod \Gamma/(e) \subset \fdmod\Gamma$
  is generated by $S_\Gamma$, it follows that $U$ defines an additive functor 
  \[
    \Hom_{\D^b(\Nccr)}(U,-)\colon \cS_{[U_1]}(\Nccr) \to \fdmod \Gamma/(e),
  \]
  which is exact by the vanishing of $\Hom_{\D^b(\Nccr)}(U,M[i])$ for $i\neq 0$.
\end{proof}

\subsection{Identifying the stable modules}
\label{ssec:stables}

The results of the previous section imply that the hyperplane arrangement of figure \ref{fig:siltingfan} controls the stability of $\Nccr$: 
if $v\in \K_0(\proj \Nccr)_\R$ is stability parameter such that $\cS_v(\Nccr)$ is nontrivial, then $v$ lies in the complement of the chambers so that either:
\begin{itemize}
  \item $v$ is a multiple of a vector $[T_i],[E_i]$ for some $i\in \Z$,
  \item or $v$ is a multiple of the vector $2[P_0] - [P_1]$, which spans the accumulation hyperplane.
\end{itemize}
In the former case $\cS_v(\Nccr)$ contains a unique stable module $M$ and $\cS_v(\Nccr) = \<M\>$, which can be obtained via tilting a simple module for some NCCR $\Nccr_n$.
The objects in $\D^b(\Coh Y)$ corresponding to these tilts of simples have been identified in \cite{DW19}, allowing us to deduce the following.
\begin{lemma}\label{lem:stabtilt}
  Let $v_i$ denote the g-vector $v_i = [T_i]$, then for all $n\geq 0$,
  \[
    \cS_{v_{2n}}(\Nccr) = \<\Uppsi(\O_\curve(n-1))\>,\quad
    \cS_{v_{2n+1}}(\Nccr) = \< \Uppsi(\O_{2\curve}(n))\>.
  \]
  Likewise, let $w_i$ denote the g-vector $w_i = [E_i]$, then for all $n< 0$.
  \[
    \cS_{w_{2n}}(\Nccr) = \<\Uppsi(\O_\curve(n-1)[1])\>,\quad
    \cS_{w_{2n+1}}(\Nccr) = \<\Uppsi(\O_{2\curve}(n)[1])\>.
  \]
\end{lemma}
\begin{proof}
  Let $i\geq 0$. By \cite{DW19} the tilting complexes $T_{i-1} \oplus T_i$ and $T_i \oplus T_{i+1}$ are obtained via finite sequence of mutations:
  \[
    P_0 \oplus P_1 \dashrightarrow \ldots \dashrightarrow 
    T_{i-1} \oplus T_i \dashrightarrow T_i \oplus T_{i+1},
  \]
  from the largest element\footnote{Note that $\Ext^1_\Nccr(P_0 \oplus P_1,-) = 0$ because $P_i$ are the projectives}$P_0\oplus P_1$ in the silting order to $T_i\oplus T_{i+1}$.
 The silting order is known to be monotonic with respect to mutation, see \cite[Theorem 9.34]{IWMemoir}, which shows that $T_{i-1} \oplus T_i > T_i \oplus T_{i+1}$.
 Therefore $T = T_{i-1} \oplus T_i$ is the Bongartz-completion of $T_i$,
 and Proposition \ref{prop:uniquewalls} implies that 
 \[
   \cS_{v_i}(\Nccr) = \<S\Lotimes_{\Nccr_i} T\>
 \]
 for $S\in \mod \Nccr_i$ the simple that is annihilated by the idempotent $T\to T_i \to T$.
 The images of these simples were calculated in \cite[Theorem 4.13]{DW19}; explicitly:
 \[
    S\Lotimes_{\Nccr_i} T \simeq \begin{cases}
      \Uppsi(\O_\curve(n-1)) & \text{if } i=2n\\
      \Uppsi(\O_{2\curve}(n)) & \text{if } i=2n+1
      \end{cases}
  \]
  We proceed similarly for the case $i<0$ using the complexes $E_i = F_i^*[1]$. The tilting complexes $F_i$ are again related by a sequence of mutations in $\cK^b(\proj \Nccr^\op)$
  \[
    F_{i-1} \oplus F_i \dashrightarrow F_i \oplus F_{i+1}.
    \dashrightarrow \ldots \dashrightarrow P_0^\op \oplus P_1^\op,
  \]
  so that $F_i\oplus F_{i+1} > F_{i-1} \oplus F_i$ with respect to the silting order.
  Because $(-)^*$ is an exact duality between $\cK^b(\proj\Nccr^\op)$ and $\cK^b(\proj\Nccr)$, it follows that
  \[
    \Ext^1_\Nccr(E_{i-1} \oplus E_i, E_i\oplus E_{i+1}) \simeq
    (\Ext^1(F_i\oplus F_{i+1},F_{i-1} \oplus F_{i+1}))^* = 0,
  \]
  which shows that $E = E_{i-1} \oplus E_i$ is the Bongartz-completion of $E_i$
  in $\tilt \Nccr$.
  Hence, it follows from Proposition \ref{prop:uniquewalls} that
  \[
    \cS_{w_i}(\Nccr) = \<S \Lotimes_{\Nccr_i} E\>
  \]
  for $S\in \mod \Nccr_i$ the simple module that is annihilated by the idempotent $E \to E_i \to E$.
  Because $\Nccr$ is 3-CY, a theorem of Iyama-Reiten \cite[Theorem 3.8]{IR08} yields natural isomorphisms
  \[
    \begin{aligned}
      (-) \Lotimes_{\Nccr_i} E 
      &\simeq \RHom_{\Nccr_i}(\RHom_{\Nccr_i^\op}(E,\Nccr_i^\op), -)
      \\&\simeq \RHom_{\Nccr_i}(E^*,-) 
      \\&= \RHom_{\Nccr_i}(F_{i-1}\oplus F_i, -)[1].
    \end{aligned}
  \]
  For $i<0$ the image of $S$ under the functor $\RHom_{\Nccr_i}(F_{i-1}\oplus F_i,-)$
  was also calculated in \cite[Proposition 4.13]{DW19}. Shifting their results by $[1]$ then yields
  \[
    \raisebox{\depth-.6em}{$
      S \Lotimes_{\Nccr_i} E = \begin{cases}\Uppsi(\O_\curve(n-1)[1]) & \text{if } i = 2n\\\Uppsi(\O_{2\curve}(n)[1]) & \text{if } i = 2n+1.\end{cases}$}
    \qedhere
  \]
\end{proof}
For $v$ on the ray spanned by $2[P_0] - [P_1]$, the vector $v$ is perpendicular to the class of the module $\Psi(\O_p)$ where $\O_p$ is structure sheaf of a point $p\in \curve$.
These modules are stable, and by adapting the proof of Nakamura's conjecture in \cite[\S 8]{BKR01} one shows that there are no other stable modules.

\begin{lemma}\label{lem:stabpts}
  For $p\in \curve$ let $\O_p$ denote the skyscraper sheaf on $p$. If $v\in \K_0(\proj \Nccr)_\R$
  is a positive real multiple of $2[P_0] - [P_1]$, then
  \[
    \cS_v(\Nccr) = \<\{\Psi(\O_p) \mid p \in \curve\}\>.
  \]
\end{lemma}
\begin{proof}
  Because the projectives $P_0$, $P_1$ are the images of the bundles $\O_Y$ and $\cN$ respectively,
  for each $p\in \curve$ the skyscraper sheaf $\O_p$ is mapped to an object which satisfies
  \[
    \begin{aligned}
      \RHom_\Nccr(P_0,\Uppsi(\O_p)) &\simeq \RHom_Y(\O_Y,\O_p) \simeq \C,\\
      \RHom_\Nccr(P_1,\Uppsi(\O_p)) &\simeq \RHom_Y(\cN,\O_p) \simeq \C^{\rk \cN} = \C^2.
    \end{aligned}
  \]
  Hence $\Uppsi(\O_p)$ is a module of dimension vector $\dvector12$.
  A module of this dimension vector is $Z_v$-stable for $v$ a multiple of $2[P_0] - [P_1]$ if $\<v,[N]\> < 0$ for any proper submodule $N$, or equivalently if the dimension vectors of any proper submodule is a multiple of $\dvector01$.
  The module $\Uppsi(\O_p)$ cannot contain a submodule of dimension vector $\dvector11$, 
  because any such submodule would induce a nontrivial quotient map $\Uppsi(\O_p) \onto S_1$, but
  \[
    \Hom_\Nccr(\Uppsi(\O_p),S_1) \simeq \Hom_Y(\O_p,\O_\curve(-1)) = 0
  \]
  shows that this is not possible.
  Likewise, $\Uppsi(\O_p)$ cannot contain $S_0$ as a submodule:
  \[
    \Hom_\Nccr(S_0,\Uppsi(\O_p)) \simeq \Hom_Y(\O_{2\curve}(-1)[1],\O_p) \simeq \Ext^{-1}(\O_{2\curve}(-1),\O_p) = 0.
  \]
  It follows that $\Uppsi(\O_p)$ is indeed a stable module in $\cS_v(\Nccr)$ for every $p\in\curve$.
  
  Now suppose there exists a module $M\in\cS_v(\Nccr)$ which is not isomorphic to $\Uppsi(\O_p)$ for any $p\in\curve$. We claim that $\Hom_\Nccr(M,\Uppsi(\O_p)) = 0$ for all $p\in\curve$. If $f\colon M\to \Uppsi(\O_p)$ is a homomorphism, then $\im f$ is simultaneously a submodule of the stable module $\Uppsi(\O_p)$ and a quotient module of the stable module $M$, which implies 
\[
  \<v,\im f\> \leq 0 \quad\text{and}\quad \<v,\im f\> \geq 0.
\]
Hence, $\<v,\im f\> = 0$ and it follows that either $\im f = \Uppsi(\O_p)$ or $\im f = 0$. Because, $M$ is not isomorphic to $\Uppsi(\O_p)$ it follows that $\Hom_\Nccr(M,\Uppsi(\O_p)) = 0$.

  By \cite{Bridgeland02} the complex $\Uppsi^{-1}(M)$ is a perverse sheaf of perversity $0$ and is thus quasi-isomorphic to a 2-term complex of sheaves supported on $\curve$. However, the vanishing of $\Hom_\Nccr(M, \Uppsi(\O_p))$ for all $p$ implies that the sheaf $\H^0(\Uppsi^{-1}(M)) \in \Coh Y$ satisfies
  \[
    \H^0(\Uppsi^{-1}(M))_p \simeq \Hom_Y(\Uppsi^{-1}(M),\O_p) \simeq \Hom_\Nccr(M,\Uppsi(\O_p)) = 0,
  \]
  over every point $p\in\curve$. Hence, $\H^0(\Uppsi^{-1}(M))$ has empty support and it follows that $\Uppsi^{-1}(M)$ is quasi-isomorphic to $\cF[1]$ for some sheaf $\cF\in \Coh Y$. 
  
  Because $Y$ is quasiprojective, there is an embedding $j\colon Y\into \bar Y$ into a
  projective variety, and the sheaf $j_*\cF\otimes_{\overline Y} \cL$ has Euler characteristic $\upchi(\cF\otimes_{\overline Y} \cL) \geq 0$ for some sufficiently ample line bundle $\cL$ on $\bar Y$.
  The King stability condition $\<v,[M]\>=0$ implies that $[\cF] = -n[\O_p]$ for some $n\geq 0$, so that by the positivity
  \[
    0\leq \upchi(\cF \otimes_{\overline Y} \cL) = -n \upchi(\O_p \otimes_{\overline Y} \cL) = -n\upchi(\O_p) = -n,
  \]
  which implies $n=0$. It follows $[M] = -[\cF] = 0$, so that $M$ is a module with dimension vector $\dvector00$, and is therefore not stable by definition.
  It follows that all stable modules in $\cS_v(\Nccr)$ are isomorphic to $\Uppsi(\O_p)$ for some $p\in \curve$, which yields the equality.
\end{proof}
Flor the sake of clarity we will henceforth write the dimension vectors/K-theory classes of the stable modules as $\updelta_\pt \colonequals [\Uppsi(\O_p)]$,
\[
  \updelta_{\curve,n} = \begin{cases}
    [\Uppsi(\O_\curve(n-1))] & n\geq 0\\
    [\Uppsi(\O_\curve(n-1)[1])] & n< 0\\
  \end{cases},\quad
  \updelta_{2\curve,n} = \begin{cases}
    [\Uppsi(\O_{2\curve}(n-1))] & n> 0\\
    [\Uppsi(\O_{2\curve}(n-1)[1])] & n\leq 0\\
  \end{cases}.
\]
A short argument shows that $\updelta_\pt = [S_0] + 2[S_1]$ while
\[
  \updelta_{\curve,n} = \pm([S_1] + n\updelta_\pt),\quad \updelta_{2\curve,n} = \pm([S_0] + n\updelta_\pt),
\]
where the sign depends on $n$. With this notation fixed, the results of the main result of this section can now be phrased as follows.
\begin{theorem}\label{thm:stableobjs}
  There exists a generic stability condition $Z\colon \K_0(\fdmod\Nccr) \to \C$ on $\fdmod \Nccr$ for which the $Z$-stable objects are as follows: for each $n\in\Z$ there is a unique $Z$-stable module of class $\updelta_{\curve,n}$ corresponding to a twist/shift
  \[
    \O_{\curve}(n-1) \quad (n\geq 0),\quad \O_\curve(n-1)[1]\quad (n<0),
  \]
  of the structure sheaf of $\curve$ across the derived equivalence.
  For each class $\updelta_{2\curve,n}$ with $n\in\Z$ there is a unique $Z$-stable module corresponding to a twist/shift
  \[
    \O_{2\curve}(n-1) \quad (n>0),\quad\quad \O_{2\curve(n-1)}[1]\quad (n\leq 0),
  \]
  of the structure sheaf of $2\curve \supset \curve$.
  The remaining stable objects are of class $\updelta_\pt$ and correspond to the point sheaves 
  \[
    \O_p\quad p\in\curve,
  \]
  and there are no other stable objects. In particular, there exists a semistable module of class $\updelta$ if and only if $\updelta$ is a multiple of one of $\updelta_\pt$, $\updelta_{\curve,n}$, $\updelta_{2\curve,n}$.
\end{theorem}
\begin{proof}
  Let $Z_v$ be the stability condition associated to any parameter $v = v_0[P_0] + v_1[P_1] \in \K_0(\proj \Nccr)$ with $v_0-v_1>0$, and suppose $M\in\fdmod \Nccr$ is $Z_v$-stable.
  Choosing a second parameter $w$ such that $\<w,[M]\> = 0$ with sign $w_0-w_1>0$, it follows from Lemma \ref{lem:vstab} that $M$ is also $Z_w$-stable, and hence 
  \[
    M\in \cS_w(\Nccr).
  \]
  Proposition \ref{prop:uniquewalls} implies that $\cS_w(\Nccr)=0$ unless $w$ is a positive real multiple of a g-vector $v_i,w_i$ as in Lemma \ref{lem:stabtilt} or of the parameter $2[P_0]-[P_1]$ as in Lemma \ref{lem:stabpts}.
  In the latter case one finds
  \[
    M \simeq \Uppsi(\O_p),
  \]
  which is a module of class $\updelta_\pt$.
  In the case $w= v_{2n}$ and $w = v_{2n+1}$ with $n\geq 0$, Lemma \ref{lem:stabtilt} implies 
  \[
    M \simeq \Uppsi(\O_\curve(n-1)),\quad \text{ and } M \simeq \Uppsi(\O_{2\curve}(n)),
  \]
  respectively, which are objects with class $\updelta_{\curve,n}$ and $\updelta_{2\curve,n+1}$ respectively. 
  For the cases $w = w_{2n}$ and $w= w_{2n+1}$ with $n < 0$, Lemma \ref{lem:stabtilt} implies
  \[
    M \simeq \Uppsi(\O_\curve(-1-n)[1]),\quad \text{ and } M \simeq \Uppsi(\O_{2\curve}(-n)[1]),
  \]
  respectively, which have the class $\updelta_{\curve,-n}$ and $\updelta_{2\curve,1-n}$ respectively.
  As there are up to scaling no other $w$ for which $\cS_w(\Nccr)\neq 0$, there are no other $Z_v$-stable modules.
\end{proof}

\section{DT Invariants for a Family of \texorpdfstring{$\ell=2$}{length 2} Flops}\label{sec:DTcalc}
In this section we first explain a technique for determining invariants of length $2$ flops in a general setup, and afterwards move to a concrete computation for the family in \S\ref{sec:familyexamples}. 

\begin{setup}\label{setup:general}
We consider a simple flopping contraction $\uppi\colon Y\to \Spec R$ of a length 2 curve $\curve\subset Y$ onto $\o\in \Spec R$, satisfying the following assumptions:
\begin{itemize}
\item there exists a symmetric quiver with potential $(Q,W)$ such that $\Jac(Q,W)$ is an $R$-algebra, with an $R$-linear derived equivalence
  \[
    \D^b(\Coh Y) \simeq \D^b(\mod \Jac(Q,W)),
  \]
  induced by a tilting complex.
\item the completion $\widehat \Jac(Q,W)$ is isomorphic to the NCCR $\Nccr = \End_{\widehat Y}(\cP)$ defined by Van den Bergh's tilting bundle $\cP$ on the completion $\widehat Y$ of $Y$ over $\o$.
\end{itemize}
\end{setup}

With these assumptions, $Y$ has a well-defined Donaldson--Thomas partition function
\[
  \Upphi(t) = \Upphi_{Q,W}(t) = \Sym\left(\sum_{\updelta\in\Delta} \frac{\BPS_\updelta}{\L^\half-\L^\mhalf} t^\updelta \right)
\]
counting the nilpotent $\Jac(Q,W)$-modules as in \S\ref{sec:DTtools}.
Because nilpotent $\Jac(Q,W)$-modules are the same as $\Nccr\simeq\widehat\Jac(Q,W)$-modules, the stability analysis of the previous section yields a decomposition of the partition function.

Let $Z_v\colon \K_0(\fdmod\Nccr) \to \C$ be the stability condition in Theorem \ref{thm:stableobjs}, with phase function $\Uptheta \colon \K_0(\fdmod\Nccr)^\times \to (0,\uppi\,]$, and let
\[
  \uptheta_\pt = \Uptheta(\updelta_\pt),\quad
  \uptheta_{\curve,n} = \Uptheta(\updelta_{\curve,n}),\quad
  \uptheta_{2\curve,n} = \Uptheta(\updelta_{2\curve,n}),
\]
denote the phases associated to the classes of the $Z_v$-stable objects.
Then the partition function has the following decomposition. 

\begin{prop}\label{prop:stabdecompflop}
  Suppose $(Q,W)$ is a quiver with potential associated to a length $2$ flopping contraction as in Setup \ref{setup:general}. Then the partition function decomposes into the infinite product
  \[
    \Upphi(t) = \Upphi^{\uptheta_\pt}(t) \cdot \prod_{n=-\infty}^\infty\left(\Upphi^{\uptheta_{2\curve,n}}(t)\cdot \Upphi^{\uptheta_{\curve,n}}(t)\right).
  \]
\end{prop}
\begin{proof}
  The factorisation identity \ref{lem:decomp} yields a decomposition $\Upphi(t) = \prod_{\uptheta\in (0,\uppi]} \Upphi^\uptheta(t)$, where the product is ordered by phase and the factors are the contributions
  \[
    \Upphi^\uptheta(t) = \int_{[\fC^\uptheta \to \fC]} \upphi_{\tr(W)}|_\fC,
  \]
  of the substacks $\fC^\uptheta$ with $\C$-points parametrising isomorphism classes in $\cS_\uptheta(\Nccr)$.
  By Theorem \ref{thm:stableobjs} one has $\cS_\uptheta(\Nccr) = 0$, hence $\Upphi^\uptheta(t)=1$, unless $\uptheta$ is one of $\uptheta_\pt$, $\uptheta_{\curve,n}$, or $\uptheta_{2\curve,n}$ for some $n\in\Z$. 
  It follows that the product only contains contributions for the given phases, and may be ordered arbitrarily, as $\Mothat[[\Delta]]$ is commutative.
\end{proof}

Comparing with the ansatz defining the BPS invariants, one sees that $\BPS_\updelta$ vanishes whenever $\updelta$ is not an integer multiple of one of the classes in \ref{thm:stableobjs}.
Under an additional rigidity assumption, the main theorem of section \S\ref{sec:autoequiv} yields a further simplification.
\begin{prop}\label{prop:BPSequiv}
  Let $(Q,W)$ be a quiver with potential for a flopping contraction as in Setup \ref{setup:general} and suppose in addition that $R^\times = \C^\times$. Then for all $n\in\Z$ and $k\in \N_{>0}$
  \[
    \BPS_{k \updelta_{\curve,n}} = \BPS_{k [S_1]}\quad\text{ and }\quad
    \BPS_{k \updelta_{2\curve,n}} = \BPS_{k [S_0]}.
  \]
\end{prop}
The proof of this proposition requires some additional background on cyclic $A_\infty$-categories, and is therefore deferred to the very end of \S\ref{sec:autoequiv}. 
The proposition shows that in the general setup, the DT theory of the flopping contraction is captured completely by the BPS invariants of the classes $k\updelta_{\curve,0} = k[S_1]$, $k\updelta_{2\curve,0} = k[S_0]$, and $k\updelta_\pt = k([S_0]+2[S_1])$.

Now consider the family of flopping contractions $Y_{a,b}\to\Spec R_{a,b}$ described in \S\ref{sec:familyexamples}, obtained from the 2-vertex quiver $Q$ on arrows $\{x,y,c,d,s\}$ with potential 
\[
  W_{a,b} = x^2y - f_{a,b}(y) + cdy^2 - csd + 2s^a,
\]
where for parameters for $a\in \N_{\geq 2}$ and $b\in \N_{\geq 1} \cup \{\infty\}$ the polynomial $f_{a,b}(y)$ is again
\[
  f_{a,b}(y) = \begin{cases}
    y^{2a} & b = \infty \\
    y^{2a} + y^{2b+1} & b \neq \infty
  \end{cases}.
\]
The partition function of $(Q,W_{a,b})$ determines a sequence of BPS invariants $\BPS_\updelta$ as in the general case.
Each dimension vector $\updelta \in \K_0(\fdmod \Nccr)$ corresponds to a unique pair $(\rk(\updelta),\chi(\updelta))$ of rank and Euler characteristic via the composition
\[
  \K_0(\fdmod\Nccr) \xrightarrow{\Uppsi^{-1}} \K_0(\Coh_\curve Y) \xrightarrow{\ (\rk,\chi)\ } \Z^2.
\]
where we note that $(\rk,\chi)([S_0]) = (-2,1)$ and $(\rk,\chi)([S_1]) = (1,0)$.
In terms of these $(\rk,\chi)$-pairs, the BPS invariants are as follows.
\begin{theorem}\label{thm:theinvariants}
  The BPS invariants $\BPS_\updelta$ associated to a length 2 flopping contraction $Y_{a,b} \to \Spec R_{a,b}$ have the following dependence on $\rk$ and $\chi$:
  \begin{itemize}
  \item if $\rk(\updelta) = 0$ then
    \begin{equation}
      \label{eq:DTpoint}
      \BPS_\updelta = \L^{-\frac32}[\P^1],
    \end{equation}
  \item if $\rk(\updelta) = \pm 1$ then
    \begin{equation}
      \label{eq:DTcontract}
      \BPS_\updelta = \begin{cases}
        \L^{-1}(1 - [D_{4a}]) + 2   & a \leq b, \\
        \L^{-1}(1 - [D_{2b+1}]) + 3 & a > b,
      \end{cases}
    \end{equation}
    where $D_{4a}$ and $D_{2b+1}$ are curves of genus $a$ and $b$, with a monodromy action of $\upmu_{4a}$ and $\upmu_{2b+1}$ respectively,

  \item if $\rk(\updelta) = \pm 2$ and $\upchi(\updelta)$ is odd then
    \begin{equation}
      \label{eq:DTcurve}
      \BPS_\updelta = \L^\mhalf(1-[\upmu_a]),
    \end{equation}
    and if $\upchi(\updelta)$ is even the BPS invariant has the realisation
    \begin{equation}
      \label{eq:MMHSreal}
      \chi_\mmhs(\BPS_\updelta) = \chi_\mmhs(\L^\mhalf(1-[\upmu_a])).
    \end{equation}
    
  \item if $|\!\rk(\updelta)| \geq 3$ and $\upchi(\updelta)$ is not divisible by $\rk(\updelta)$ then
    \[
      \BPS_\updelta = 0,
    \]
    while for $|\!\rk(\updelta)| \geq 3$ and $\upchi(\updelta)$ divisible by $\rk(\updelta)$ the realisation vanishes:
    \[
      \chi_\mmhs(\BPS_\updelta) = 0.
    \]
  \end{itemize}
\end{theorem}
\begin{proof}
  Comparing the BPS ansatz of $\Upphi(t)$ with that of the partition functions $\Upphi^{\uptheta_\pt}(t)$, $\Upphi^{\uptheta_{2\curve,n}}(t)$ and $\Upphi^{\uptheta_{\curve,n}}(t)$ via Proposition \ref{prop:stabdecompflop} yields
  \[
    \begin{aligned}
      \Sym\left(\sum_{\updelta\in\Updelta} \frac{\BPS_\updelta}{\L^\half-\L^\mhalf}\right) =
      &\Sym\left(
        \sum_{k>0} \frac{\BPS_{k\updelta_\pt} t^{\updelta_\pt}}{\L^\half-\L^\mhalf} 
        + \sum_{n\in\Z, k>0} \frac{\BPS_{k\updelta_{\curve,n}} t^{\updelta_{\curve,n}}}{\L^\half-\L^\mhalf}\right.
        \\&\quad\quad\quad \left.+ \sum_{n\in\Z, k>0} \frac{\BPS_{k\updelta_{2\curve,n}} t^{\updelta_{2\curve,n}}}{\L^\half-\L^\mhalf}
      \right),
    \end{aligned}
  \]
  so the BPS invariant vanishes unless $\updelta$ is a multiple of the class of a stable module.
  These classes correspond to the rank/Euler pairs
  \[
    (\rk,\chi)(\updelta_\pt) = (0,1),\quad 
    (\rk,\chi)(\updelta_{\curve,n}) = \pm (1,n)
    ,\quad 
    (\rk,\chi)(\updelta_{2\curve,n}) = \pm(2,2n+1)
  \]
  where the sign depends on the sign of $n$.
  The calculation of the invariants for multiples of these classes is given in the rest of this section, so it remains to verify the case distinction.

  A class $\updelta \in \Updelta$ with $\rk(\updelta) = 0$ is given by $\updelta = k\updelta_\pt$ for some $k$, and it is shown in Lemma \ref{lem:DTpoint} that the BPS invariant is given by \eqref{eq:DTpoint}.

  A class $\updelta$ with $\rk(\updelta) = \pm1$ is given by $\updelta = \updelta_{\curve,n}$ for some $n$, and it is shown in Lemma \ref{lem:DTcontract} that the BPS invariant is given by \eqref{eq:DTcontract}.
  
  A class $\updelta$ with $\rk(\updelta) = \pm 2$ is either given by $\updelta = \updelta_{2\curve,n}$, in which case $\chi(\updelta)$ is odd and Lemma \ref{lem:DTcurve} shows that the BPS invariant is given by \eqref{eq:DTcurve}, or $\updelta = 2\updelta_{\curve,n}$ for which the MMHS realisation is \eqref{eq:MMHSreal}, as shown in Lemma \ref{lem:MMHSreal}.

  Finally, Lemma \ref{lem:DTcurve} also shows that $\BPS_{k\updelta_{2\curve,n}} = 0$ for $k>1$, so if $\updelta$ is a class with $|\rk(\updelta)| \geq 3$ such that $\BPS_\updelta \neq 0$, then it must be of the form $\updelta = k\updelta_{\curve,n}$.
  By inspection
  \[
    (\rk,\chi)(\updelta) = k (\rk,\chi)(\updelta_{\curve,n}) = \pm (k,k n),
  \]
  so it follows that the Euler characteristic of such a class is divisible by its rank.
  It is however shown in Lemma \ref{lem:MMHSvanish} that the MMHS realisations of $\BPS_{k\updelta_{\curve,n}}$ vanish for $k>2$, which yields the last claim.
\end{proof}

To gain insight in the MMHS--realisation of the invariants, it is worth determining the Hodge structure and monodromy on the curves $D_{4a}$ and $D_{2b+1}$. The monodromy is concentrated on the middle cohomology, and as we show in \S\ref{sec:rankonecalc}, is of the following form.
\begin{prop}\label{prop:Hodgemon}
  The monodromic mixed Hodge structure $\H^1(D_{4a},\Q)$ decomposes over $\C$ into a direct sum of the following irreducible $\upmu_{4k}$-representations:
  \[
  \H^1(D_{4a}, \O_{D_{4a}}) \simeq \textstyle\bigoplus_{j=1}^a \upxi^{2j-1+2a},\quad
  \H^0(D_{4a}, \Omega_{D_{4a}}) \simeq \textstyle\bigoplus_{j=1}^a \upxi^{2j-1}.
  \]
  Likewise, $H^1(D_{2b+1},\Q)$ is the direct sum of the following $\upmu_{2b+1}$-representations:
  \[
  \H^1(D_{2b+1},\O_{D_{2b+1}}) \simeq \textstyle\bigoplus_{j=1}^b \upxi^{b+j},\quad
  \H^0(D_{2b+1},\Omega_{D_{2b+1}}) \simeq \textstyle\bigoplus_{j=1}^b \upxi^{j},
  \]
  where in each case $\upxi$ denotes a generator for the representation ring.
\end{prop}
From the above characterisation, one can easily deduce other realisations of the other refined invariants, such as the Hodge spectrum, weight polynomial and the numerical invariants.
\begin{corollary}
  The Hodge spectrum realisations $\hsp_\updelta \colonequals \upchi_\hsp([\BPS_\updelta])$ are
 \begin{itemize}
 \item for $\rk(\updelta) = 0$
   \[
     \hsp_\updelta(u,v) = - u^{-\frac32}v^{-\frac32} (1 + uv),
   \]
 \item for $\rk(\updelta) = \pm 1$ 
   \[
     \hsp_\updelta(u,v) = 
     \begin{cases}
       1 + \sum_{j=1}^{a} \left(u^{\frac{2j - 1}{4a}}v^{-\frac{2j-1}{4a}} + u^{-\frac{2j-1}{4a}}v^{\frac{2j-1}{4a}}\right) & a\leq b\\
       2 + \sum_{j=1}^{b} \left(u^{\frac j{2b+1}}v^{-\frac j{2b+1}} + u^{-\frac j{2b+1}}v^{\frac j{2b+1}} \right) & a> b
     \end{cases}
   \]
 \item for $\rk(\updelta) = \pm 2$
   \[
     \hsp_\updelta(u,v) = u^\half v^\mhalf \cdot \sum_{j=1}^{a-1} u^{\frac{j}{a}}v^{-\frac{j}{a}}
   \]
 \item and $\hsp_\updelta(u,v) = 0$ for $|\rk(\updelta)| \geq 3$.
 \end{itemize}
\end{corollary}
The weight-polynomial $\wt_\updelta(q) = \hsp_\updelta(q^\half,q^\half)$ is given by $\wt_\updelta(q) = -q^{3/2}(1+q)$ in the rank zero case, and is given by a constant otherwise:
\begin{align*}
  \wt_\updelta(q) &= \begin{cases}
    \min\{2a +1, 2b+2\} & \rk(\updelta) = \pm 1\\
    a-1 & \rk(\updelta) = \pm 2\\
    0 & \text{otherwise}
  \end{cases}
\end{align*}
As Katz \cite{Katz08} shows, these constants are precisely the Gopakumar--Vafa invariants of the flopping contraction.
By \cite{Toda15} these numbers determine the dimension of the \emph{contraction algebra} $\Nccr_\con$ of NCCR $\Nccr$ as defined by Donovan--Wemyss \cite{DW16} and the dimension of its abelianisation:
\begin{align*}
\dim_\C \Nccr_\con &= \GV_1 + 4 \GV_2 = \begin{cases} 6a - 3 & a\leq b \\ 4a+2b-2 & a > b\end{cases},\\
\dim_\C  \Nccr_\con^\ab &= \GV_1 = \min\{2a, 2b+1\} + 1.
\end{align*}
These same dimensions were also found independently by Kawamata \cite{Kawamata20}.

\subsection{Motivic invariants for $2\curve$}

Proposition \ref{prop:BPSequiv} shows that the BPS invariants for the dimension vectors $k\updelta_{2\curve,n}$ can be calculated from the case $\updelta = k\updelta_{2\curve,0} = k[S_0]$.
For this class the space $\Rep_\updelta(Q)$ parametrises representations $\uprho$ of the form
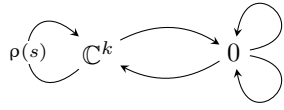
\[
  \begin{tikzpicture}[>=stealth,inner sep=1pt]
    \clip (-3,-.8) rectangle (3,.8);
    \node[circle,outer sep=.8pt] (A) at (-.9,0) {$\C^k$};
    \node[circle,outer sep=.8pt] (B) at (.9,0) {$0$};
    \draw[->] (A) to[bend left] (B);
    \draw[->] (B) to[bend left] (A);
    \draw[->] (B) to[loop,looseness=10,in=280,out=-10] (B);
    \draw[->] (B) to[looseness=10,in=80,out=10] (B);  
    \draw[->] (A) to[loop,looseness=7,in=140,out=220] node[midway,fill=white, inner sep=.8pt]{$\scriptstyle\uprho(s)$} (A);
  \end{tikzpicture}
\]
with gauge group $\GL_\updelta \simeq \GL_k(\C)$ acting on the first vertex. Because these are precisely the semistable representations of phase $\uptheta_{2\curve,0}$, it follows that
\[
  \fM_\updelta^{\uptheta_{2\curve,0}} = \fM_\updelta \simeq \fM_{\cQ,k},
\]
where $\cQ$ denotes the quiver with a single vertex and a single loop $s$.
Under this isomorphism, the function $\tr(W_{a,b})$ pulls back to $\tr(2s^a)$, and one obtains an equality
\[
  \Upphi^{\uptheta_{2\curve,0}}(t) = \sum_{k\geq 0} \int_{\fC_{\dvector k0}}\kern-12pt\upphi_{\tr(W)} t_1^k = \Upphi_{\cQ,2s^a}(t_1),
\]
which implies that the BPS invariants $\BPS_{k[S_0]}$ are the BPS invariants of the one loop quiver with potential $s^a$, which were found by Davison--Meinhardt \cite{DM15}.
\begin{lemma}\label{lem:DTcurve}
  Let $\updelta_{2\curve,n}$ denotes the class of the unique stable module of phase $\uptheta_{2\curve,n}$, then
  \[
    \Upphi^{\uptheta_{2\curve,n}}(t) = \Upphi_{\cQ,2s^a}(t^{\updelta_{2\curve,n}}) = \Sym\left(\frac{\L^\mhalf(1-[\upmu_a])}{\L^\half-\L^\mhalf} t^{\updelta_{2\curve,n}}\right).
  \]
  In particular, the associated BPS invariants are
  \[
    \BPS_{\updelta_{2\curve,n}} = \L^\mhalf(1-[\upmu_a]),\quad \BPS_{k\updelta_{2\curve,n}} = 0 \quad\text{for } k\geq 2.
  \]
\end{lemma}
\begin{proof}
  Setting $n=0$, it follows from \cite[Theorem 6.4]{DM15} and the above discussion that
  \[
    \Sym\left(\sum_{k> 0} \frac{\BPS_{k[S_0]}}{\L^\half-\L^\mhalf} t^{k[S_0]} \right) = \Upphi^{\uptheta_{2\curve,0}}(t) = \Upphi_{\cQ,s^a}(t^{[S_0]}) = \Sym\left(\frac{\L^\mhalf(1-[\upmu_a])}{\L^\half-\L^\mhalf} t^{[S_0]}\right).
  \]
  For other $n$, it follows from Proposition \ref{prop:BPSequiv} that the BPS invariants satisfy 
  \[
    \BPS_{k\updelta_{2\curve,n}} = \BPS_{k\updelta_{2\curve,0}} = \BPS_{k[S_0]}.
  \]
  Comparing the BPS ansatz for each partition function now yields the result.
\end{proof}

\subsection{Motivic rank zero invariants} \label{sec:skyscrapers}

For the phase $\uptheta = \uptheta_\pt$ the points of the moduli space $\fC^\uptheta \subset \fM^\uptheta$ correspond to the finite length sheaves with support in $\curve$, which are precisely the extensions of the point sheaves $\O_p$.
We will show that the partition function $\Upphi^\uptheta(t)$ decomposes along the supports of these finite length sheaves.

Fix a point $p\in\curve$, and define $\fC^p$ as the closed substack of $\fC^\uptheta$ parametrising semistable modules $M \in \fdmod \Nccr$ such that $\Uppsi^{-1}(M) \in \Coh_\curve Y$ is supported on $p$. 
Likewise, let $\fC^\circ$ be the open substack of $\fC^\uptheta$ of modules $M$ for which $\Uppsi^{-1}(M)$ is support in the complement $\curve\setminus\{p\}$. 
The associated elements in the Hall algebra define partition functions
\[
  \Upphi^p(t) = \int_{[\fC^p \to \fC]}\kern-2pt \upphi_{\tr(W_{a,b})}|_\fC,\quad
  \Upphi^\circ(t) = \int_{[\fC^\circ \to \fC]}\kern-2pt \upphi_{\tr(W_{a,b})}|_\fC,
\]
via an application of the integration map recalled in \S\ref{ssec:mothall}. 
These partition functions are a virtual count of the components of a finite length sheaf supported on $p$ and its complement, and together these partitition functions recover $\Upphi^\uptheta(t)$.
\begin{lemma}\label{lem:partdecomp}
  There is a decomposition $\Upphi^\uptheta(t) = \Upphi^\circ(t) \cdot \Upphi^p(t)$.
\end{lemma}
\begin{proof}
  Because the integration map is a homomorphism, it suffices to show the identity
  \begin{equation}\label{eq:toprove}
    [\fC^{\uptheta} \to \fC] = [\fC^\circ \to \fC] \star [\fC^p \to \fC]
  \end{equation}
  in the motivic Hall algebra. Recall that the product is defined via the stack $\fExt$ parametrising short exact sequences in $\fdmod\Nccr$.
  Chasing through the definition, one sees that the right hand side of \eqref{eq:toprove} is the class $[\fY \to \fC]$ associated to the substack $\fY \subset \fExt$ parametrising the short exact sequences
  \[
    0 \to M^\circ \to M \to M^p \to 0,
  \]
  with $M^\circ$ in $\fC^\circ$ and $M^p$ in $\fC^p_{\updelta_2}$, where the map $\fY\into \fExt \to \fC$ sends a short exact sequence to its middle term.
  Because an extension of semistable modules of phase $\uptheta$ is again semistable of phase $\uptheta$,
  this map factors as $\fY \to \fC^{\uptheta} \into \fC$, and we claim that this factorisation identifies the classes $[\fY\to \fC]$ and $[\fC^{\uptheta} \into \fC]$ in $\K(\St/\fC)$.
  By \cite[Lemma 3.2]{Bridgeland12} it is sufficient to check that functor $\fY(\C) \to \fC^{\uptheta}(\C)$ on $\C$-points is an equivalences of categories.
  A object in $\fC^\uptheta(\C)$ is simply a semistable module $M \in \fdmod \Nccr$ of phase $\uptheta$, which is the image $M=\Uppsi(\cF)$ of a finite length sheaf $\cF$ on $\curve$.
  Such a moduli is isomorphic to the direct sum $M \simeq M^\circ \oplus M^p$ of the modules
  \[
    M^\circ \colonequals \Uppsi(\cF|_{\A^1}),\quad M^p = \Uppsi(\cF|_{p}),
  \]
  which are objects $M^\circ \in \fC^\circ(\C)$ and $M^p\in \fC^p(\C)$.
  It follows that $M$ corresponds to a unique short exact sequence $M^\circ \to M \to M^p$, which is an object of $\fY(\C)$.
  It follows that $\fY(\C) \simeq \fC^\uptheta(\C)$, which yields the result.
\end{proof}

Recall that the space $Y_{a,b}$ was constructed in \S~\ref{sec:familyexamples} as a closed subscheme of the moduli scheme $\cM^\uptheta_{\updelta_\pt}(Q) = \Rep^\uptheta_{\updelta_\pt}(Q)\gitquot\!\GL_\updelta$.
The latter consists of two open charts
\[
  \begin{aligned}
    U_y &= \{\uprho \in \Rep^\uptheta_{\updelta_\pt}(Q) \mid \det (\uprho(c)\mid \uprho(yc)) \neq 0\}\gitquot\!\GL_\updelta,\\
    U_x &= \{\uprho  \in \Rep^\uptheta_{\updelta_\pt}(Q) \mid \det (\uprho(c) \mid \uprho(xc)) \neq 0\}\gitquot\!\GL_\updelta.
  \end{aligned}
\]
Intersecting these charts with the curve $\curve \subset Y_{a,b} \subset \cM^\uptheta_{\updelta_\pt}(Q)$ yields a decomposition of $\curve\simeq \P^1$ into an open chart $U_y\cap \curve \simeq \A^1$ and a point $p = \curve \setminus (U_y\cap \curve)$, corresponding the semistable nilpotent representation $\uprho_p \in \cM^\uptheta_{\updelta_\pt}(Q)$ such that $\uprho_p(yc) = 0$.
This choice of point then yields a decomposition of $\Upphi^\uptheta(t)$ as in the lemma.

To compute the partition functions of the critical loci $\fC^\circ$ and $\fC^p$ we employ a decomposition of $\fM^\uptheta$.
We note that $\fC^\circ$ is precisely the intersection of the nilpotent locus $\fN$ with the critial locus of $\tr(W)$ restricted to the open substack
\[
  \fM^\circ \colonequals \coprod_{k\geq 0} \big\{\uprho \in \Rep_{k\updelta_\pt}(Q) \mid 
    \left(\uprho(c) \mid \uprho(yc)\right) \text{ is invertible} \big\} /\GL_{k\updelta_\pt}. \\
\]
To compute the partition function $\Upphi^\circ(t)$, it is convenient to rewrite $\fM^\circ$ as the moduli stack of a quiver. Consider the quiver $\cQ$ with a unique vertex and $9$ loops
\[
\cQ_1 = \{\upalpha_1,\upalpha_2,\upalpha_3,\upbeta_1,\upbeta_2,\upbeta_3,\upgamma_1,\upgamma_2,\upgamma_3\}.%
\] 
Let $\loc_y\colon \C Q_\cyc \to \C \cQ_\cyc$ be the composition of the trace map $\tr_\cQ\colon \Mat_{3\times 3}(\C\cQ) \to \C \cQ$ with the homomorphism $\C Q \to \Mat_{3\times 3}(\C \cQ)$ defined on generators as
\begin{equation}\label{eq:basedchange}
  \begin{gathered}
    s \mapsto 
    \begin{pmatrix}
      \upgamma_3 & 0 & 0\\
      0 & 0 & 0\\
      0 & 0 & 0
    \end{pmatrix}
    ,\quad 
    c \mapsto 
    \begin{pmatrix}
      0 & 0 & 0\\
      1 & 0 & 0\\
      0 & 0 & 0
    \end{pmatrix},\quad
    d \mapsto 
    \begin{pmatrix}
      0 & \upbeta_1 & \upbeta_2\\
      0 & 0 & 0\\
      0 & 0 & 0
    \end{pmatrix},\\
    x \mapsto 
    \begin{pmatrix}
      0 & 0 & 0\\
      0 & \upgamma_2 & \upbeta_3 - \upgamma_1\upgamma_3\\
      0 & \upgamma_1 & \upalpha_3 - \upgamma_2
    \end{pmatrix},\quad
    y \mapsto 
    \begin{pmatrix}
      0 & 0 & 0\\
      0 & 0 & \upalpha_1 + \upgamma_3\\
      0 & 1 & \upalpha_2
    \end{pmatrix}.
  \end{gathered}
\end{equation}
Then $\cW = \loc_y(W_{a,b}) \in \C\cQ_\cyc$ is a potential on $\cQ$, and we have the following.
\begin{lemma}\label{lem:ninelooploc}
  There is an isomorphism $\fM^\circ \xrightarrow{\sim} \fM_{\cQ}$ that pulls back $\tr(W_{a,b})$ to $\tr(\cW)$.
\end{lemma}
\begin{proof}
  Fix $\updelta = k\updelta_\pt$ and consider the tautological representation $\uptau$ on $\Rep_\updelta(Q)$:
  the $\C[\Rep_\updelta(Q)]$-valued representation corresponding to the identity across the isomorphism
  \[
    \Rep_\updelta(Q)(\C[\Rep_\updelta(Q)]) \simeq \Hom_{\mathrm{Sch}}(\Rep_\updelta(Q),\Rep_\updelta(Q)).
  \]
  Let $A = \big(\!\!\begin{array}{c|c} \uptau(c)\! &\! \uptau(yc)   \end{array}\!\!\big)$ denote the $\C[\Rep_\updelta(Q)]$-valued $2k\times 2k$-matrix obtained by adjoining the block matrices $\uptau(c)$ and $\uptau(yc)$.
  Then $\fM^\circ$ is the quotient $U/\GL_\updelta$ of the subspace 
  \[
    U = \Spec \C[\Rep_\updelta(Q)][(\det A)^{-1}] \subset \Rep_\updelta(Q).
  \]
  There is a closed subspace $V\subset U$ defined by the vanishing of the $(2k)^2$ entries in the matrix $A - \id_{2k \times 2k}$. We claim that $U$ is a $\GL_{2k}$-torsor over this subspace $V$ with respect to the action of $\GL_{2k} \simeq \{\id\} \times \GL_{2k} \subset \GL_\updelta$. To show this, consider the invertible $\C[U]$-valued matrix
  \[
    g= \left(\begin{array}{c|c} 
               \id_{k\times k} & 0 \\ \hline 
               0 & A^{-1}
             \end{array}\right) \in \GL_\updelta(\C[U]).
  \]
  then the family $g \cdot \uptau$ of representations satisfies $\big( g\cdot \uptau(c)\mid g\cdot \uptau(yc)) = \id_{2k\times 2k}$ and hence defines a map $U \to V$. The $\GL_{2k}$-action restricts to a free \& transitive action on the fibres of this map, which shows that $U$ is indeed a $\GL_{2k}$-torsor over $V$. It follows that 
  \[
    \fM^\circ \simeq U/\GL_\updelta \simeq V/\GL_k.
  \]  
  Because $V$ is affine, any $k$-dimensional representation of $\cQ$ with values in $\C[V]$ determines a map $V\to \Rep_k(\cQ)$ via $\Rep_k(\cQ)(\C[V])\simeq \Hom_{\mathrm{Sch}}(V,\Rep_k(\cQ))$. One such representation is defined as follows.
  The restriction $\uptau|_V$ of the tautological representation to $V$ is of the form
  \begin{gather*}
    \uptau(s),\quad 
    \uptau(c) = \left(\begin{array}{c} \id_{k\times k} \\ \hline 0 \end{array}\right),\quad
    \uptau(d) = \left(\begin{array}{c|c} d_0 & d_1 \end{array}\right),\\
    \uptau(x) = \left(\begin{array}{c|c} x_{00} & x_{01} \\\hline x_{10} & x_{11} \end{array}\right),\quad
    \uptau(y) = \left(\begin{array}{c|c} 0 & y_{01} \\\hline \id_{k\times k} & y_{11} \end{array}\right),
  \end{gather*}
  where $\uptau(s),d_0,d_1,x_{00},x_{01},x_{10},x_{11},y_{01},y_{11}$ are $\C[V]$-valued $k\times k$ matrices.
  Hence there is a representation $\upsigma\in \Rep_k(\cQ)(\C[V])$ which maps the loops in $\cQ$ to 
  \begin{gather*}
    \upsigma(\upalpha_1) = y_{01} - \uptau(s),\quad 
    \upsigma(\upalpha_2) = y_{11},\quad
    \upsigma(\upalpha_3) = x_{11}+x_{00},\\
    \upsigma(\upbeta_1) = d_0,\quad 
    \upsigma(\upbeta_2) = d_1,\quad
    \upsigma(\upbeta_3) = x_{01} + x_{10}\uptau(s),\\
    \upsigma(\upgamma_1) = x_{10},\quad 
    \upsigma(\upgamma_2) = x_{00},\quad
    \upsigma(\upgamma_3) = \uptau(s).
  \end{gather*}
  One checks that this map is a $\GL_k$-equivariant isomorphism, and therefore yields an isomorphism $\fM^\circ \simeq V/\GL_k \simeq \fM_{\cQ,k}$ of moduli stacks. 
  Moreover, by comparing the above with \eqref{eq:basedchange} one
  sees that the isomorphism identifies the functions $\tr(W_{a,b})$ and $\tr(\cW)$ on the two spaces.
  Repeating this process for all $k$ and taking the disjoint union gives the required isomorphism.
\end{proof}
The lemma implies that the nilpotent critical locus $\fC^\circ$ of $\tr(W_{a,b})$ in $\fM^\circ$ is isomorphic to a substack of the critical locus of $\tr(\cW)$ and that the motivic vanishing cycles on the two spaces agree.
Hence, the partition function $\Upphi^\circ(t)$ can be computed from $(\cQ,\cW)$.
\begin{lemma}\label{lem:DTcalcA}
  The contribution of the stratum $\fC^\circ$ is
  \[
  \Upphi^\circ(t) = \Sym\left(\sum_{k\geq 1} \frac{\L^{-\frac32}[\A^1]}{\L^\half - \L^\mhalf} \cdot t_0^kt_1^{2k}\right)
  \]
\end{lemma}
\begin{proof}
  The potential $\cW = \loc_y(W_{a,b}) \in \C\cQ_\cyc$ has the following form:
  \[
    \begin{aligned}
      \cW &= \loc_y(x^2y - f_{a,b}(y) + cdy^2 - sdc + 2 s^a) \\
      &= 
      \upalpha_1\upbeta_1
      + \upalpha_2\upbeta_2
      + \upalpha_3\upbeta_3
      + [\upgamma_1,\upgamma_2]\upgamma_3
      + \upalpha_1([\upgamma_1,\upgamma_2] + \upalpha_3\upgamma_1)
      - \upalpha_3\upgamma_1\upgamma_3
      \\&\quad
      + \upalpha_2((\upalpha_3- \upgamma_2)^2  + \upgamma_1\upbeta_3 - \upgamma_1^2\upgamma_3)
      - \tr_\cQ\left(f_{a,b}\left(\begin{smallmatrix}0&\upalpha_1 + \upgamma_3\\1&\upalpha_2\end{smallmatrix}\right) \right)
      + 2 \upgamma_3^a.
    \end{aligned}
  \]
  We will construct an automorphism $\uppsi\colon \C\cQ\to \C\cQ$ which maps $\cW$ to the simplified form $\textstyle\sum_{i=1,2,3} \upalpha_i\upbeta_i + \cW_\min$, for some minimal potential $\cW_\min = \cW_\min(\upgamma_1,\upgamma_2,\upgamma_3)$ of degree $\geq 3$. By inspection, the potential can be written up to cyclic permutation as
  \[
    \begin{aligned}
      \cW &= \upalpha_1\upbeta_1 
      + \upalpha_2\upbeta_2 + \upalpha_3\upbeta_3 + 
      \upalpha_1 \cdot u_1 + \upalpha_2 \cdot u_2 + \upalpha_3 \cdot u_3\\
      &\quad
       + [\upgamma_1,\upgamma_2]\upgamma_3 - \tr_\cQ\left(f_{a,b}\left(\begin{smallmatrix}0&\upgamma_3\\1&0\end{smallmatrix}\right)\right) + 2 \upgamma_3^a.
    \end{aligned}
  \]
  where each $u_i$ is a noncommutative polynomial of order $\geq 2$ which only contains the $\upgamma$-variables and the variables $\upalpha_j$, $\upbeta_j$ for $j>i$.
  Consider the automorphisms $\uppsi_1,\uppsi_2,\uppsi_3$ of the path algebra $\C\cQ$ which map
  \[
    \uppsi_i(\upbeta_i) = \upbeta_i - u_i,
  \]
  and act as the identity on the other generators. Then $\uppsi_i(u_j) = u_j$ for $j>i$, and one sees that the composition $\uppsi \colonequals \uppsi_3\circ \uppsi_2\circ\uppsi_1$ maps $\cW$ to 
  $\uppsi(\cW) = \textstyle\sum_{i=1,2,3} \upalpha_i\upbeta_i + \cW_\min$ where
  \[
    \cW_\min = [\upgamma_1,\upgamma_2]\upgamma_3 
    - \tr_\cQ\left(f_{a,b}\left(\begin{smallmatrix}0&\upgamma_3\\1&0\end{smallmatrix}\right)\right) + 2\upgamma_3^a.
  \]
  is the minimal potential on the loops $\upgamma$. Decomposing $f_{a,b}$ as $f_{a,b}(y) = y^{2a} + y g(y^2)$ one sees the last two terms cancel:
  \[
    \tr_\cQ\left(f_{a,b}\left(\begin{smallmatrix}0&\upgamma_3\\1&0\end{smallmatrix}\right)\right)
    =
    \tr_\cQ\left(\begin{smallmatrix}\upgamma_3^a &0\\ 0 & \upgamma_3^a\end{smallmatrix}\right)
    + \tr_\cQ\left(\begin{smallmatrix}0& \upgamma_3 g(\upgamma_3) \\ g(\upgamma_3) &0\end{smallmatrix}\right)
    =  2 \upgamma_3^a
  \]
  so that $\cW_\min$ is simply given by the cubic term $[\upgamma_1,\upgamma_2]\upgamma_3$.  
  For each $k\in\N_{>0}$, let $\fJ_k$ denote the image of the substack $\fC_{k\updelta_\pt}^\circ \subset \fM^\circ$ under the isomorphism $\fM^\circ \simeq \fM_\cQ$ in Lemma \ref{lem:ninelooploc}.
  Then it follows by the motivic Thom--Sebastiani identity that
  \[
    \int_{\fC_{k\updelta_\pt}^\uptheta} \upphi_{\tr(W_{a,b})} 
    = \int_{\fJ_k} \upphi_{\tr(\cW)}
    = \int_{\fJ_k} \upphi_{\tr(\uppsi(\cW))} 
    = \int_{\fJ_k\cap\{\upalpha_i=\upbeta_i =0\}}\kern-4pt \upphi_{\tr([\upgamma_1,\upgamma_2]\upgamma_3)}.
  \]
  The motivic vanishing cycle of $[\upgamma_1,\upgamma_2]\upgamma_3$ was computed by Behrend--Bryan--Szendr\Hu oi \cite{BBS13} as a motivic point count of $\A^3$.
  Here, the vanishing cycle is restricted to the support over a line $\A^1\subset \A^3$, and a slight modification of their argument shows that
  \[
    \Upphi^\circ(t) 
    = \sum_{k\geq 0} \int_{\fJ_k\cap\{\upalpha_i=\upbeta_i =0\}}\kern-4pt \upphi_{\tr([\upgamma_1,\upgamma_2]\upgamma_3)} \cdot t^{k\updelta_\pt}
    = \Sym\left(\sum_{n\geq 1} \frac{\L^{-\frac32}[\A^1]}{\L^\half-\L^\mhalf} \cdot t^{k\updelta_\pt}\right).\qedhere
  \]
\end{proof}

For the partition function $\Upphi^p(t)$ we employ an analogous construction.
As in Lemma \ref{lem:ninelooploc}, for each $\updelta = k\updelta_\pt$ the tautological representation defines a matrix $A = \left(\!\!\begin{array}{c|c}\uptau(c)\!\! &\! \uptau(xc)\end{array}\!\!\right)$, and there is an open neighbourhood $\fU$ of $\fC^p_\updelta$ in $\fM_\updelta^{\uptheta}$ of the form $\fU \simeq U/\GL_\updelta$, where
\[
  U = \Spec \C[\Rep_\updelta(Q)][\det A^{-1}].
\]
As before, $U$ is a $\GL_{2k}$-torsor over the closed subspace $V \subset U$ cut out by the entries of the matrix $A-\id_{2k\times 2k}$, so that $\fU \simeq V/\GL_k$. 
The restriction of $\uptau$ to $V$ is of the form
\[
  \begin{gathered}
    \uptau(s),\quad 
    \uptau(c) = \left(\begin{array}{c} \id_{k\times k} \\ \hline 0 \end{array}\right),\quad
    \uptau(d) = \left(\begin{array}{c|c} d_0 & d_1 \end{array}\right),\\
    \uptau(x) = \left(\begin{array}{c|c} 0 & x_{01} \\\hline \id_{k\times k} & x_{11} \end{array}\right),\quad
    \uptau(y) = \left(\begin{array}{c|c} y_{00} & y_{01} \\\hline y_{10} & y_{11} \end{array}\right),
  \end{gathered}
\]
for $\C[V]$-valued $k\times k$-matrices $\uptau(s),d_0,d_1,x_{01},x_{11},y_{00},y_{01},y_{10},y_{11}$.
and there is a $\GL_k$-equivariant isomorphism $V\to \Rep_k(\cQ)$ determined by the family of representations $\upsigma\in \Rep_k(\cQ)(\C[V])$ which takes the following values on generators:
\[
  \begin{aligned}
    \upsigma(\upalpha_1) &= -d_0,\quad &
    \upsigma(\upalpha_2) &= x_{01},\quad &
    \upsigma(\upalpha_3) &= x_{11},\\    
    \upsigma(\upbeta_1) &= \uptau(s) - y_{00}^2 - y_{01}y_{10},\quad &
    \upsigma(\upbeta_2) &= y_{00} + y_{11},\quad &
    \upsigma(\upbeta_3) &= y_{01},\\    
    \upsigma(\upgamma_1) &= y_{10},\quad &
    \upsigma(\upgamma_2) &= y_{00},\quad &
    \upsigma(\upgamma_3) &= d_1,
  \end{aligned}
\]
Again, the isomorphism of stacks $\fU \simeq V/\GL_k \simeq \fM_{\cQ,k}$ identifies $\tr(W_{a,b})$ with the trace of a potential $\cW$ on $\cQ$.
One checks that $\cW = \loc_x(W_{a,b})$, for $\loc_x \colon \C Q_\cyc \to \C\cQ_\cyc$ defines as the composition of the trace $\tr_\cQ\colon \Mat_{3\times 3}(\C\cQ) \to \C\cQ$ with the homomorphism $\C Q \to \Mat_{3\times 3}(\C\cQ)$ which maps the generators of $\cQ$ to
\[
  \begin{gathered}
    s \mapsto 
    \begin{pmatrix}
      \upbeta_1 + \upgamma_2^2 + \upbeta_3\upgamma_1 & 0 & 0\\
      0 & 0 & 0\\
      0 & 0 & 0
    \end{pmatrix}
    ,\quad 
    c \mapsto 
    \begin{pmatrix}
      0 & 0 & 0\\
      1 & 0 & 0\\
      0 & 0 & 0
    \end{pmatrix},\quad
    x \mapsto 
    \begin{pmatrix}
      0 & 0 & 0\\
      0 & 0 & \upalpha_2\\
      0 & 1 & \upalpha_3
    \end{pmatrix},\\
    d \mapsto 
    \begin{pmatrix}
      0 & -\upalpha_1 & \upgamma_3\\
      0 & 0 & 0\\
      0 & 0 & 0
    \end{pmatrix},\quad
    y \mapsto 
    \begin{pmatrix}
      0 & 0 & 0\\
      0 & \upgamma_2 & \upbeta_3\\
      0 & \upgamma_1 & \upbeta_2 - \upgamma_2
    \end{pmatrix}.
  \end{gathered}
\]
Via this isomorphism of stacks, the stack $\fC^p$ is identified with the nilpotent critical locus $\fC_{\cQ,\cW}$, and so the $\Upphi^p(t) = \Upphi_{\cQ,\cW}(t)$.
\begin{lemma}\label{lem:DTcalcp}
  There is an equality
  \[
    \Upphi^p(t) = \Sym\left(\sum_{n\geq 1} \frac{\L^{-\frac32}[\pt]}{\L^\half - \L^\mhalf}\cdot t^{k\updelta_\pt}\right).
  \]
\end{lemma}
\begin{proof}
  As in Lemma \ref{lem:DTcalcA}, we calculate the motivic contribution via the potential
  \begin{align*}
    \cW &= \loc_x(x^2y - f_{a,b}(y) + cdy^2 - sdc + 2 s^a)
    \\&= 
    \upalpha_1\upbeta_1 + \upalpha_2\upbeta_2 + \upalpha_3\upbeta_3 + \upgamma_3[\upgamma_1,\upgamma_2]
    +\upalpha_2\upalpha_3\upgamma_1 + \upalpha_3^2(\upbeta_2 - \upgamma_2)
    +\upgamma_3\upbeta_2\upgamma_1
    \\&\quad+ 2(\upbeta_1 + \upgamma_2^2 + \upbeta_3 \upgamma_1)^a - \loc_x(f_{a,b}(y)).
  \end{align*}
  as a formal potential in $\widehat{\C\cQ}_\cyc$. We will construct a formal automorphism of $\widehat{\C\cQ}$ that maps $\cW$ to a potential of the form $\sum_i \upalpha_i\upbeta_i + \upgamma_3[\upgamma_1,\upgamma_2]$, as this is sufficient to yield the required identity: by Lemma \ref{lem:motintformal} and Lemma \ref{lem:quadterms} it follows that
  \[
    \Upphi^p(t) = \Upphi_{\cQ,\cW}(t) = \Upphi_{\cQ,\uppsi(\cW)}(t) = \Upphi_{\cQ,\upgamma_3[\upgamma_1,\upgamma_2]}(t),
  \]
  and the latter is again given by the computation of Behrend--Bryan--Szendr\Hu oi \cite{BBS13}:
  \[
    \begin{aligned}
      \Upphi^p(t) = \Upphi_{\cQ,\upgamma_3[\upgamma_1,\upgamma_2]}(t) = \Sym\left(\sum_{k\geq 1} \frac{\L^{-3/2}[\pt]}{\L^\half-\L^\mhalf} \cdot t^{k\updelta_\pt}\right).
    \end{aligned}
  \]
  To construct the automorphism, we first write $\cW$ in the form
  \begin{equation}\label{eq:randomeq}
    \begin{aligned}
      \cW &= 
      \upalpha_1\upbeta_1 + \upalpha_2\upbeta_2 + \upalpha_3\upbeta_3 
      + \upgamma_3[\upgamma_1,\upgamma_2] 
      + u \cdot \upbeta_1 
      + \upalpha_2(\upalpha_3\upgamma_1) 
      + \upalpha_3(-\upalpha_3\upgamma_2) 
      \\&\quad
      + w \cdot \upbeta_3
      + (\upalpha_3^2 + \upgamma_1\upgamma_2  + v) \cdot \upbeta_2
      + 2 \upgamma_2^a
      - \tr_\cQ\left(\begin{smallmatrix}f_{a,b}(\upgamma_2) & 0 \\ \ldots & f_{a,b}(-\upgamma_2) \end{smallmatrix}\right),
    \end{aligned}
  \end{equation}
  for $u= u(\upbeta_1,\upbeta_3,\upgamma_1,\upgamma_2)$ of order $\geq 1$, and $v = v(\upbeta_2,\upbeta_3,\upgamma_1,\upgamma_2)$, $w = w(\upbeta_3,\upgamma_1,\upgamma_2)$  of order $\geq 3$.
  Once again, the last two terms in \eqref{eq:randomeq} cancel by a parity argument:
  writing $f_{a,b}(y) = y^a + yg(y)$ one sees that
  \[
    \tr_\cQ\left(\begin{smallmatrix}f_{a,b}(\upgamma_2) & 0 \\ \ldots & f_{a,b}(-\upgamma_2) \end{smallmatrix}\right)
    = \upgamma_2^a + (-\upgamma_2)^a + \upgamma_2g(\upgamma_2) - \upgamma_2 g(\upgamma_2)
    = 2\upgamma_2^2.
  \]
  Let $\uppsi_1,\uppsi_2,\uppsi_3\colon \C\cQ\to \C\cQ$ be the automorphisms which map
  \[
    \begin{gathered}
    \uppsi_1(\upalpha_1) = \upalpha_1 - u,
    \quad
    \uppsi_2(\upbeta_2) = \upbeta_2 - \upalpha_3\upgamma_1,
    \quad
    \uppsi_3(\upalpha_2) = \upalpha_2 - \upalpha_3^2 - \upgamma_1\upgamma_2 - \uppsi_2(v),
    \end{gathered}
  \]
  and act as the identity on the other generators. 
  By construction, these map $\cW$ to
  \[
    \begin{aligned}
      \uppsi_3(\uppsi_2(\uppsi_1(\cW))) &=
          \upalpha_1\upbeta_1 + \upalpha_2\upbeta_2 + \upalpha_3\upbeta_3 
          + \upgamma_3[\upgamma_1,\upgamma_2] 
          \\&\quad
          + w \cdot \upbeta_3 
          + \upalpha_3 \cdot (-\upalpha_3\upgamma_2
          - \upgamma_1\upalpha_3^2 - \upgamma_1^2\upgamma_2 - \upgamma_1 
          \uppsi_2(v)).
        \end{aligned}
  \]
  Because the terms on the bottom line are of order $\geq 3$, one can use a recursive algorithm analogous to \cite[\S 3]{DWZ08} to further reduce these terms order by order:
  for each $n$ there is an automorphism $\uppsi_n \colon \C\cQ \to \C\cQ$, trivial up to order $n-1$, and noncommutative polynomials 
  $w^{(n)}(\upbeta_1,\upbeta_3,\upgamma_1,\upgamma_2)$,
  $v^{(n)}= v^{(n)}(\upbeta_1,\upgamma_1,\upgamma_2)$, and
  $\cW_\min^{(n)}(\upgamma_1,\upgamma_2)$
  of orders $\geq n-1$, $\geq n-1$, and $\geq n$ respectively,
  such that $\uppsi_n(\uppsi_{n-1}(\cdots\uppsi_1(\cW)))$ is of the form
  \[
    \upalpha_1\upbeta_1 + \upalpha_2\upbeta_2 + \upalpha_3\upbeta_3 
    + \cW_\min^{(3)} + \ldots + \cW_\min^{(n)}
    + w^{(n)} \upbeta_3 + \upalpha_3 v^{(n)}.
  \]
  The existence of the above data can be shown by induction on the base case 
  \[
    \cW_\min^{(3)} = \upgamma_3[\upgamma_1,\upgamma_2],\quad w^{(3)} = w,\quad v^{(3)} = -\upalpha_3\upgamma_2 -\upgamma_2\upalpha_3^2 - \upgamma_1^2\upgamma_2 - \upgamma_1\uppsi(v).
  \]
  Suppose the data $\uppsi_n$, $\cW_\min^{(n)}$, $v^{(n)}$, $w^{(n)}$ as above are given for $n\leq N$,
  then we construct the automorphism $\uppsi_{N+1}$ which maps
  \[
    \uppsi_{N+1}(\upalpha_2) = \upalpha_2 + w^{(N)},\quad \uppsi_{N+1}(\upbeta_3) = \upbeta_3 + v^{(N)},
  \]
  and sends all other generators to themselves.
  By assumption, this satisfies
  \[
    \begin{aligned}
      \uppsi_{N+1}(\cdots(\uppsi_1(\cW))) &=
      \upalpha_1\upbeta_1 + \upalpha_2\upbeta_2 + \upalpha_3\upbeta_3 
      + \cW_\min^{(3)} + \ldots + \cW_\min^{(N)}
      \\&\quad+ w^{(N)}v^{(N)} 
      - \uppsi_{N+1}(w^{(N)}) v^{(N)} - w^{(N)}\uppsi_{N+1}(v^{(N)}).
      \\&\quad 
      + (\uppsi_{N+1}(w^{(N)}) - w^{(N)})\upbeta_3 + \upalpha_3(\uppsi_{N+1}(v^{(N)})-v^{(N)}).
    \end{aligned}
  \]
  The bottom two lines contain only terms of order $\geq (N-1)^2 \geq N+1$.
  Hence, these lines can be written (up to cyclic permutation) as
  $\cW_\min^{(n)} + w^{(n)} \upbeta_3 + \upalpha_2 v^{(n)}$
  for nc-polynomials of the claimed form. By induction the required data then exists for all $n\geq 4$.

  Having constructed this sequence of automorphisms, the limit $\uppsi = \lim_{n\to\infty} \uppsi_n\circ\cdots\circ \uppsi_1$ is a well-defined formal automorphism $\widehat{\C\cQ} \to \widehat{\C\cQ}$, and maps $\cW$ to 
  \[
    \uppsi(\cW) = \sum_{i=1}^3 \upalpha_i\upbeta_i + \cW_\min
    = \sum_{i=1}^3 \upalpha_i\upbeta_i + \upgamma_3[\upgamma_1,\upgamma_2] + \cW_\min',
  \]
  where $\cW_\min'$ a powerseries in the variables $\upgamma_1, \upgamma_2$. 
  The Jacobi algebra $\widehat \Jac(\cQ,\uppsi(\cW)) \simeq \widehat \Jac(\cQ,\cW)$ is isomorphic to the ring $\C[[\O_{Y,p}]]$ of formal functions at $p$, and therefore commutative.
  This implies that the cyclic derivatives of $\cW_\min$ are contained in the (completed) commutator ideal, and because $\cW_\min'$ is only a function of $\upgamma_1,\upgamma_2$ we find
  \[
    \begin{gathered}
      \del_{\upgamma_1} \cW_\min' \equiv 0 \quad\mathrm{mod}\  ([\upgamma_1,\upgamma_2])_{\mathrm{top}},\quad
      \del_{\upgamma_2} \cW_\min' \equiv 0 \quad\mathrm{mod}\  ([\upgamma_1,\upgamma_2])_{\mathrm{top}},\quad
      \del_{\upgamma_3} \cW_\min' = 0,
    \end{gathered}
  \]
  where $(\ldots)_{\mathrm{top}}$ denotes the closure of the ideal in the adic topology.
  A moment of reflection shows\footnote{One can for example, apply the Euler identity $\sum_i \upgamma_i \del_{\upgamma_i} H =n \cdot H$ to the homogeneous parts of $\cW_\min$.} that $\cW_\min= \upgamma_3[\upgamma_1,\upgamma_2] + q\cdot [\upgamma_1,\upgamma_2]$ for some noncommutative polynomial $q$ of order $\geq 2$. One final automorphism $\upgamma_3 \mapsto \upgamma_3 - q$ then maps $\cW_\min$ to $\upgamma_3[\upgamma_1,\upgamma_2]$, as required.
\end{proof}
Adding up the contributions of $\fC^\circ$ and $\fC^p$ now yields the desired DT invariants.
\begin{lemma}\label{lem:DTpoint}
  The BPS invariants are $\BPS_{k[\pt]} = \L^{-\kern-.5pt\frac32}[\P^1]$ for all $k\geq 1$.
\end{lemma}
\begin{proof}
  By Lemma \ref{lem:partdecomp} the partition function decomposes as $\Upphi^\uptheta(t) = \Upphi^\circ(t) \cdot \Upphi^p(t)$, so it follows from Lemma \ref{lem:DTcalcA}, Lemma \ref{lem:DTcalcp} and the properties of the plethystic exponential that
  \[
  \Upphi^\uptheta(t) = 
  \Sym\left(\sum_{k\geq0}\frac{\L^{-\frac32}\left([\A^1] + [\pt]\right)}{\L^\half-\L^\mhalf}t^{k\updelta_\pt}\right).\qedhere
\]
\end{proof}

\begin{remark}
  In the framework of \cite{BBS13} the BPS invariants are interpreted as a virtual count of points, and are given by the restriction of the \emph{virtual motive} of $Y$ to $\curve$:
  \[
    \virt{Y}|_\curve = \L^{-3/2}[\P^1].
  \]
  The proposition shows that the invariants $\BPS_{k[\pt]}$, which we compute in the framework of \cite{KS08}, are in fact given by this virtual motive.
\end{remark}

\subsection{Invariants for $\curve$}
Finally, we compute BPS invariants $\BPS_{k\updelta_{\curve,n}}$ associated to the phases $\uptheta_{\curve,n}$. 
It once again follows by Proposition \ref{prop:BPSequiv} that these are independent of $n$, so we may focus on $\BPS_{k\updelta_{\curve,0}} = \BPS_{k[S_1]}$.
The moduli spaces for the dimension vectors $\updelta = k[S_1]$ parametrise representations of the form
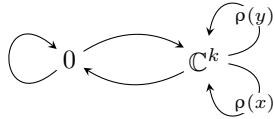
\[
  \begin{tikzpicture}[>=stealth]
    \clip (-3,-.9) rectangle (3,.9);
    \node[circle,outer sep=.8pt,inner sep=1pt] (A) at (-.9,0) {$0$};
    \node[circle,outer sep=.3pt,inner sep=1pt] (B) at (.9,0) {$\C^k$};
    \draw[->] (A) to[bend left] (B);
    \draw[->] (B) to[bend left] (A);
    \draw[->] (B) to[loop,looseness=7,in=280,out=-10]  node[pos=.46, fill=white, inner sep=.2pt, outer sep=0pt]{$\scriptstyle\uprho(x)$} (B);
    \draw[->] (B) to[looseness=7,in=80,out=10]  node[pos=.46,fill=white, inner sep=.2pt, outer sep=0pt]{$\scriptstyle\uprho(y)$} (B);  
    \draw[->] (A) to[loop,looseness=10,in=140,out=220] (A);
  \end{tikzpicture}
\]
and it once again follows that the moduli space $\fM^{\uptheta}$ is isomorphic to the moduli space $\fM_\cQ$ of a quiver $\cQ$ with a unique vertex and loops $\cQ_1 = \{x,y\}$. The potential restricts to $\cW = x^2y - f_{a,b}(y) \in \C\cQ_\cyc$ and the BPS invariants are determined by the partition function of $(\cQ,\cW)$ via
\[
\Upphi^{\uptheta_{\curve,0}}(t) = \sum_{k\geq 0} \int_{\fC_{\cQ,k}}\kern-5pt \upphi_{\tr(\cW)} \cdot t^{k[S_1]} = \Upphi_{\cQ,\cW}(t).
\]
The coefficients of the partition function can be computed via the integration formula of Denef--Loeser, which requires one to find an embedded resolution of the zero locus $\{\tr(\cW)=0\}$ in $\Rep_k(\cQ)$. 
We are able to find such an embedded resolution for $k=1$, but for $k>2$ the dimension of $\Rep_k(\cQ)$ is at least $8$ and finding a suitable embedded resolution seems out of reach.
Instead, we will determine the realisations in the Grothen\-dieck ring of monodromic mixed Hodge structures. 

As shown by Davison--Meinhardt \cite{DM20}, the $\MMHS$-realisation $\upchi_\mmhs(\BPS_{k[S_1]})$ coincides with the class $[\cBPS_k] \in \K_0(\MMHS)$ of a genuine MMHS
\[
  \cBPS_k \colonequals \H_c\left(\cM_{\cQ,k},\left(\upphi_{\tr(\cW)}^\mmhs \IC_{\cM_{\cQ,k}}\right)^{\mathrm{nilp}}\right),
\]
defined as the cohomology with compact support on the coarse moduli scheme $\cM_{\cQ,k}$ of $\fM_{\cQ,k}$ with values in the image of the intersection complex $\IC_{\cM_{\cQ,k}}$ under the vanishing cycle functor $\upphi_{\tr(\cW)}^\mmhs$, restricted to the nilpotent locus. 
It was shown by Davison \cite{Davison19} that this MMHS vanishes when $k$ is greater than the length.

\begin{lemma}\label{lem:MMHSvanish}
  The cohomological invariant $\cBPS_k$ vanishes for $k > 2$. In particular $\upchi_\mmhs(\BPS_{k[S_1]}) = 0$ for $k>2$.
\end{lemma}
\begin{proof}
  By \cite[Theorem B]{Davison19} the monodromic mixed Hodge structures $\cBPS_k$ are all concentrated in degree $0$, and by \cite[Proposition 5.2]{Davison19} the dimension $\dim_\C(\cBPS_k)$ of the degree $0$ part is given by the Gopakumar--Vafa invariant $n_k$ of the flop. Because $Y$ is a length $\ell=2$ flop, the GV invariant $n_k$ vanishes if $k > \ell = 2$ and $\cBPS_k$ is therefore trivial for $k>2$, and so are the K-theory classes.
\end{proof}

\subsection{The MMHS realisation for $k=2$}
For $k=2$ the coarse moduli space is smooth.
\begin{lemma}
$\cM_{\cQ,2} \simeq  \A^5$
\end{lemma}
\begin{proof}
As shown by Procesi \cite{Procesi84}, the ring of $\GL_2$-invariant functions on the space of representations $\Rep_2(\C\<x_1,\ldots,x_n\>)$ of the free algebra $\C\<x_1,\ldots,x_n\>$ is the ring of trace functions $\tr(p)\colon \uprho \mapsto \tr(\uprho(p))$ of noncommutative polynomials $p \in \C\<x_1,\ldots,x_n\>$, subject to the relations
\begin{equation}\label{eq:traceidentity}
  \begin{aligned}
    \tr(p_1p_2p_3) + \tr(p_1p_3p_2) &= \tr(p_1p_2)\tr(p_3) + \tr(p_1p_3)\tr(p_2)
    \\&\quad+ \tr(p_1)\tr(p_2p_3) - \tr(p_1)\tr(p_2)\tr(p_3).
  \end{aligned}
\end{equation}
for any triple of noncommutative polynomials $p_1,p_2,p_3$.
In the case $n=2$ the trace ring was shown to be a polynomial ring by LeBruyn--Procesi \cite[Proposition II.3.1]{LP87}, and one finds
\[
\cM_2(\cQ) = \Rep_2(\cQ)\gitquot \GL_2 \simeq \Spec \C[\tr(x), \tr(y), \tr(x^2), \tr(y^2), \tr(xy)].\qedhere
\]
\end{proof}
Because $\cM_{\cQ,2}$ is smooth, its intersection complex $\IC_{\cM_2(\cQ)}$ is trivial, and
we can calculate the BPS invariant of the function $\tr(\cW)$ on the coarse scheme.
\begin{lemma}\label{lem:BPSranktwo} 
Let $\{0\} \subset \cM_{\cQ,2}$ denote the origin, then for all $n\in\Z$
\[
  \upchi_\mmhs(\BPS_{2\updelta_{\curve,n}}) = \upchi_\mmhs\left(\int_{\{0\}}\kern-2pt \upphi_{\tr(\cW)}\right),
\]
where on the right-hand side $\tr(\cW)$ is regarded as a function on $\cM_{\cQ,2}$.
\end{lemma}
\begin{proof}
  Because $\cM_{\cQ,2} \simeq \A^5$ is smooth of dimension $5$, its intersection complex is simply
  \[
    \IC_{\cM_{\cQ,2}} = \constQ[\dim\cM_{\cQ,2}] = \constQ[5],
  \]
  where $\constQ$ denotes the constant sheaf with value $\Q$ on $\cM_{\cQ,2}$.
  It then follows from the monodromic version of \cite[Theorem 4.2.1]{DL98} (see \cite[\S 2.7]{Davison19}), that
  \[
  \upchi_\mmhs\left(\int_{\{0\}}\kern-2pt \upphi_{\tr(\cW)}\right) = 
    \left[\H_c\left(\cM_{\cQ,2},\left(\upphi^\mmhs_{\tr(\cW)}\constQ[5]\right)^\nilp\right)\right],
  \]
  and the right hand side is precisely $[\cBPS_2] = \upchi_\mmhs(\BPS_{2[S_1]}) = \upchi_\mmhs(\BPS_{2\updelta_{\curve,n}})$.
\end{proof}
\begin{lemma}\label{lem:MMHSreal}
  For all $n\in\Z$ the realisations of the BPS invariants satisfy
  \[
    \upchi_\mmhs(\BPS_{2\updelta_{\curve,n}}) = \upchi_\mmhs\left(\L^\mhalf(1-[\upmu_a])\right).
  \]
\end{lemma}
\begin{proof}
  Substituting $p_1 = p_2 = y$ and $p_3 = y^n$ into \eqref{eq:traceidentity}, there is a relation
  \[
  2\cdot \tr(y^{n+2}) = \tr(y^2)\tr(y^n) + 2\cdot\tr(y^{n+1})\tr(y) - \tr(y)^2\tr(y^n),
  \]
  in the coordinate ring of $\cM_{\cQ,2}$ for every $n>0$, which implies that one can expand
  \[
    \tr(f_{a,b}(y)) = \tr(y) \cdot v(\tr(y),\tr(y^2)) + (\tr(y^2))^a,
  \]
  for some polynomial $v(\tr(y),\tr(y^2))$. Likewise, substituting $p_1 = p_2 = x$, $p_3 = y$ into equation \eqref{eq:traceidentity} yields
  \[
    2\tr(x^2y) = \tr(x^2)\tr(y) + 2\tr(xy)\tr(x) - \tr(x)^2\tr(y).
  \]
  Choosing the coordinates
  \begin{gather*}
    a_1 = \tfrac12\tr(x^2) - \tfrac12\tr(x)^2 - v(\tr(y),\tr(y^2)),\quad b_1 = \tr(y),\\
    a_2 = \tr(xy),\quad b_2 = \tr(x),\quad z = \tr(y),
  \end{gather*}
  one finds that $\tr(\cW)$ can be written as the polynomial
  \[
    \tr(x^2y - f_{a,b}(y)) = a_1b_1 + a_2b_2 - z^a.
  \]
  Hence it follows from the Thom--Sebastiani identity that the DT invariant is determined by the minimal part:
  \[
  \int_{\{0\}}\kern-2pt \upphi_{\tr(\cW)} =
  \int_{\{0\}}\kern-2pt \upphi_{a_1b_1 + a_1 b_2 - z^a} =
  \int_{\A^1_z}\kern-2pt \upphi_{z^a} =
  \L^\mhalf (1-[\upmu_a]).\qedhere
  \]
\end{proof}
\subsection{The motivic invariant for $k=1$}\label{sec:rankonecalc}

The invariant for $k=1$ is the linear term in the partition function $\Upphi_{\cQ,\cW}(T)$: 
the plethystic exponential has the first order expansion $\Sym(\sum_{k>1} a_k\cdot T^k) = 1 + a_1\cdot T + \ldots$ so the partition function is of the form
\[
\Upphi_{\cQ,\cW}(T) = 1 + \frac{\BPS_{[S_1]}}{\L^\half-\L^\mhalf} \cdot T + (\ldots\text{higher order terms}\ldots).
\]
Hence $\BPS_{[S_1]}$ can be calculated as the motivic integral of the vanishing cycle $\upphi_{\tr(\cW)}$ on the origin in $\Rep_1(\cQ) \simeq \A^2$. 
The function $\tr(\cW) = x^2y-f_{a,b}(y) \in \C[x,y]$ has isolated singularities, so we can fix an open neighbourhood $U \into \Rep_1(\cQ)$ of the origin $\{0\} \subset \A^2 \simeq \Rep_1(\cQ)$ which does not contain any other singularities. Then
\[
  \BPS_{[S_1]} = (\L^\half-\L^\mhalf) \cdot \int_{\fC_{\cQ,1}}\kern-5pt \upphi_{\tr(\cW)}
  = \int_{\{0\}}\kern-2pt  \upphi_{\tr(\cW)} = \int_U\kern-2pt \upphi_{\tr(\cW)}|_U.
\]
To calculate the right-hand side we construct an embedded resolution $h\colon X \to U$ of the divisor $Z \colonequals \{{\tr(\cW)} = 0\}$ such that $h^*Z$ has normal crossings: 
i.e. every prime component of $h^*Z$ is a smooth codimension 1 subvariety of $X$ and the intersection of any set of components is defined by a regular sequence.
The embedded resolution depends on $a$ and $b$ as follows.
\begin{lemma}\label{lem:blowup1}
  If $a \leq b$ there exists an embedded resolution $h\colon X \to U$ such that 
  \[
    h^*Z = L_1 + L_2 + \sum_{k=2}^a (2k-1) \cdot E_{2k-1} + 2a \cdot E_{2a} + 4a \cdot E_{4a},
  \]
  where $L_1$ and $L_2$ are the components of the strict transform of $Z$ and the $E_i$ are exceptional $\P^1$'s, intersecting as in the following diagram:
  
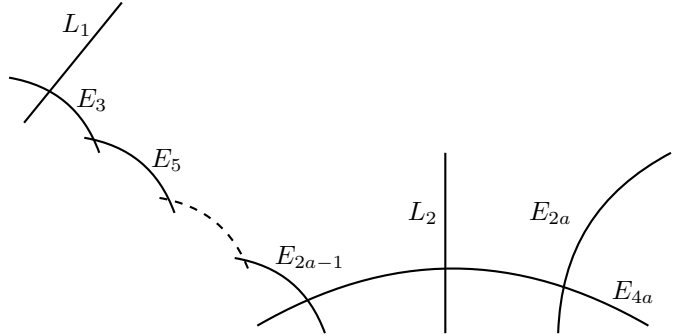
\[
    \begin{tikzpicture}
      \node at (-5.1,3.7) {$L_1$};
      \draw[thick] (-5.8,2.4) to (-4.5,4);
      \node at (-4.9,2.7) {$E_3$};
      \draw[thick] (-6,3) to[bend left] (-4.8,2);
      \node at (-3.9,1.9) {$E_5$};
      \draw[thick] (-5,2.2) to[bend left] (-3.8,1.2);
      \draw[thick, dashed] (-4,1.4) to[bend left] (-2.8,.4);
      \node at (-2,.6) {$E_{2a-1}$};
      \draw[thick] (-3,.6) to[bend left] (-1.8,-.4);
      \node at (-.5,1.2) {$L_2$};
      \draw[thick] (-.2,2) to (-.2,-.4);
      \node at (2.3,.15) {$E_{4a}$};
      \draw[thick] (-2.7,-.3) to[bend left] (2.5,-.3);
      \node at (1.2,1.2) {$E_{2a}$};
      \draw[thick] (1.3,-.4) to[bend left] (2.8,2);
    \end{tikzpicture}
  \]
\end{lemma}
\begin{lemma}\label{lem:blowup2}
  If $a > b$ there exists an embedded resolution $h\colon X \to U$ such that 
  \[
    h^*Z = L_1 + L_2 + \sum_{k=2}^b (2k-1) \cdot E_{2k-1} + (2b+1) \cdot E_{2b+1},
  \]
  where $L_1$ and $L_2$ are the components of the strict transform of $Z$ and the $E_i$ are exceptional $\P^1$'s, intersecting as in the following diagram:
  
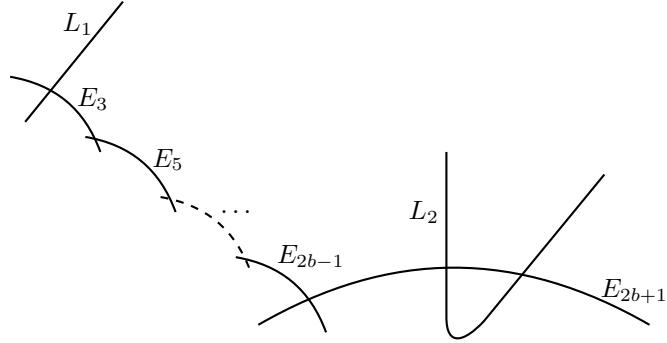
\[
    \begin{tikzpicture}
      \node at (-5.1,3.7) {$L_1$};
      \draw[thick] (-5.8,2.4) to (-4.5,4);
      
      \node at (-4.9,2.7) {$E_3$};
      \draw[thick] (-6,3) to[bend left] (-4.8,2);
      
      \node at (-3.9,1.9) {$E_5$};
      \draw[thick] (-5,2.2) to[bend left] (-3.8,1.2);
      
      \node at (-3,1.2) {$\ldots$};
      \draw[thick, dashed] (-4,1.4) to[bend left] (-2.8,.4);
      
      \node at (-2,.6) {$E_{2b-1}$};
      \draw[thick] (-3,.6) to[bend left] (-1.8,-.4);
      
      \node at (-.5,1.2) {$L_2$};
      \draw[thick] (-.2,2) to (-.2,-.2) to[out=270,in=-133,looseness=2] (.3,-.25) to (1.9,1.7);
      
      \node at (2.3,.15) {$E_{2b+1}$};
      \draw[thick] (-2.7,-.3) to[bend left] (2.5,-.3);
    \end{tikzpicture}
\]
\end{lemma}
The resolutions can be found via a sequence of blowups of points, a straightforward but somewhat long computation which we include in appendix \ref{sec:blowups}.

To compute the motives we use the formula of Denef--Loeser recalled in \S\ref{sec:vancyc}.
Write $h^*Z$ as a sum $\sum_{i\in I} m_i E_i$ of prime divisors $E_i$ with multiplicity $m_i>0$ ranging over an index set $I$, and let $E_J$ and $E_J^\circ$ be the strata for subsets $J\subset I$.
Looijenga \cite{Looijenga02} defines the following degree $m_I = \gcd\{m_j \mid j \in J\}$ cover $D_J \to E_J$
of $E_J$: let $g\colon \widetilde X \to \A^1$ be the normalisation of the base-change
\[
  \begin{tikzpicture}
    \node (A) at (-2,\gr) {${\A^1 \times_{\A^1} X}$};
    \node (B) at (0,\gr) {$X$};
    \node (C) at (-2,0) {$\A^1$}; 
    \node (D) at (0,0) {$\A^1$}; 
    \draw[->] (C) to[edge label=${\scriptstyle z\mapsto z^{m_I}}$] (D);
    \draw[->] (B) to[edge label=${\scriptstyle {\tr(\cW)} \circ h}$] (D);
    \draw[->] (A) to (B);
    \draw[->] (A) to (C);
  \end{tikzpicture}
\]
then $\widetilde X \to X$ is a $\upmu_I$-fold cover of $X$, and $D_J \to E_J$ is the restriction to $E_J \subset X$. This cover has a canonical $\upmu_{m_I}$-action via its action on $\A^1$.
We will also denote by $D^\circ_J \to E_J^\circ$ the restriction to the open subspace $E_J^\circ$, which is a regular cover with Galois group $\upmu_{m_J}$. To ease notation, we write $D_j$, etc. instead of $D_{\{j\}}$, etc. if $J = \{j\}$ is a one-element set.

\begin{lemma}\label{lem:DTcontract}
  For any $n\in\Z$ the BPS invariant associated to $\updelta_{\curve,n}$ is
  \[
    \BPS_{\updelta_{\curve,n}} = \BPS_{[S_1]} = \begin{cases}
      \L^{-1}(1-[D_{4a}]) + 2 & a \leq b \\
      \L^{-1}(1-[D_{2b+1}]) + 3 & a > b
    \end{cases}.
  \]
  where $D_{4a}$ is a genus $a$ curve equipped with a $\upmu_{4a}$-action and $D_{2b+1}$ is a genus $b$ curve equipped with a $\upmu_{2b+1}$ action.
\end{lemma}
\begin{proof}
  As before, the first equality follows from Proposition \ref{prop:BPSequiv}, so it suffices to calculate the case $\updelta_{\curve,0} = [S_1]$, which we do using the Denef--Loeser formula.
  
  Given an embedded resolution as described above, the Denef--Loeser formula for the motivic integral is
  \[
    \L^{\dim U/2} \cdot \int_U\kern-2pt \upphi_{\tr(\cW)} = [Z] - \sum_{\varnothing\neq J\subset I} (1-\L)^{|J|-1} [D_J^\circ],
  \]
  where $D_J^\circ$ carries the $\muhat$ action induced from the $\upmu_{m_J}$-action. For the case $a\leq b$, the explicit expression can then be read off from the diagram in Lemma \ref{lem:blowup1}: write $E_1 = L_1$ and $E_2 = L_2$ and let $I = \{1,2,3,5,\ldots,2a-1,2a,4a\}$ then the formula expands to
  \begin{align*}
    \L\cdot \int_U\kern-2pt \upphi_{\tr(\cW)} &= [Z] - [D_1^\circ] - [D_2^\circ]
    \\&\quad                            
    - (1-\L)[D_{\{1,3\}}^\circ] 
    - (1-\L)[D_{\{2,4a\}}^\circ]
    \\&\quad
    -\sum_{i=2}^a[D_{2i-1}^\circ] 
    - (1-\L) \sum_{i=2}^{a-1} [D_{\{2i-1,2i+1\}}^\circ]
    \\&\quad
    - [D_{2a}^\circ]
    - [D_{4a}^\circ] 
    - (1-\L)[D_{\{4a,2a\}}^\circ]
    - (1-\L)[D_{\{2a-1,4a\}}].
  \end{align*}
  We will reduce this expression line by line. The divisor $L_1$ appears with multiplicity $m_1 = 1$, so that $D_1 = L_1$ is a trivial cover and $D_1^\circ \subset L_1$ is the complement of the intersection point, which lies above the singularity of ${\tr(\cW)}$; similarly for $L_2$. Because $L_1 \sqcup L_2$ is the strict transform of $Z$, it is isomorphic to $Z$ outside the singular locus, so that
  \[
    [Z] - [D_1^\circ] - [D_2^\circ] = ([Z]- 1) - ([L_1] -1) - ([L_2] - 1) + 1 = 1.
  \]
  Likewise, the intersection points of $L_1 \cap E_3$ and $L_2 \cap E_{4a}$ have a trivial cover, so that
  \[
    -(1-\L)[D^\circ_{\{1,3\}}] - (1-\L)[D^\circ_{\{2,4a\}}] = 2\L - 2.
  \]
  For $i=2,\ldots,a-1$, the exceptional $E_{2i-1} \simeq \P^1$ has multiplicity $m_{2i-1} = 2i-1$ and intersects $E_{2i+1}$ in a point with multiplicity $\gcd(2i-1,2i+1) = 1$. It follows that each cover $D_{2i-1} \to E_{2i-1}$ is connected, and therefore restricts to a regular covering
  \[
    D_{2i-1}^\circ \to E_{2i-1}^\circ \simeq \G_m,
  \]
  for each $i=2,\ldots,a$. The only connected cover is $D^\circ_{2i-1} \simeq \G_m$, which means that the map $D_{2i-1}^\circ \to E_{2i-1}^\circ$ is an equivariant isomorphism. Hence in $\Mot^\muhat(\C)$ there is an equality
  \[
    [D_{2i-1}^\circ] = [E_{2i-1}^\circ] = \L -1.
  \]
  It follows that these curves and their intersections contribute
  \[
    -\sum_{i=2}^a [D_{2i-1}^\circ] - (1-\L)\sum_{i=2}^{a-1} [D_{\{2i-1,2i+1\}}^\circ] = (a-1)(1-\L) - (a-2)(1-\L) = 1-\L
  \]
  Likewise, $D_{2a-1}$ intersects $D_{4a}$ in a point with multiplicity $\gcd(2a-1,4a) = 1$ and contributes
  \[
    -(1-\L)[D_{\{2a,4a\}}] = \L-1.
  \]
  The curve $E_{2a}$ only intersects $E_{4a}$ in a single point, so that $E_{2a}^\circ \simeq \A^1$, which has only the trivial $\upmu_{2a}$-cover $D_{2a}^\circ = (\A^1)^{\sqcup 2a} \to \A^1$ for which $\upmu_{2a}$ permutes the sheets. Hence there is an equivariant isomorphism $D_{2a}^\circ \simeq \A^1 \times \upmu_{2a}$, and it follows that $[D_{2a}] = \L[\upmu_{2a}]$. Likewise, the intersection $E_{2a} \cap E_{4a}$ is a point which is covered by $E_{\{2a,4a\}} = \upmu_{2a}$ because the multiplicity is $\gcd(2a,4a)=2a$. Adding these two contributions gives:
  \[
    - [D_{2a}^\circ] 
    - (1-\L)[D_{\{2a,4a\}}^\circ] 
    = -\L[\upmu_{2a}] - (1-\L)[\upmu_{2a}] = -[\upmu_{2a}].
  \]
  The curve $E_{4a}$ intersects $L_2$ and $E_{2a-1}$ in a point of multiplicity $1$ and $E_{2a}$ in a point of multiplicity $2a$, so $D_{4a}\to E_{4a}$ is a connected cover with Euler characteristic
  \[
    \upchi(D_{4a}) = 4a \upchi(E_{4a}^\circ) + (2+2a) = 4a\upchi(\P^1-3\pt) - (2+2a) = 2 - 2a.
  \]
  Hence, $D_{4a}$ is a smooth projective curve of genus $a$ with equivariant motive
  \[
    [D_{4a}] = [D_{4a}^\circ] + 2 + [\upmu_{2a}].
  \]
  Collection the terms found above, it follows that the motivic integral is
  \[
  \begin{aligned}
    \int_U\kern-2pt \upphi_{\tr(\cW)} &= \L^{-1}\left(1 + 2\L - 2 + (1-\L) + (\L-1) - [\upmu_{2a}] - [D_{4a}] + 2 + [\upmu_{2a}]\right)
    \\&= \L^{-1}(1- [D_{4a}]) + 2.
  \end{aligned}
  \]
  The case $a > b$ proceeds in much the same way, and yields the motivic integral
  \[
  \begin{aligned}
    \L\int_U\kern-2pt \upphi_{\tr(\cW)} &= 1 + (b-1)(1-\L) 
    + (b+2)(\L-1) - [D_{2b+1}] + 3
    \\&= (1- [D_{2b+1}]) + 3\L,
  \end{aligned}
  \]
  where $D_{2b+1}$ is a genus $b$ curve with an $\upmu_{2b+1}$ action.
\end{proof}
To complete the calculation, we will make the Hodge structure and monodromy on the curves $D_{4a}$ and $D_{2b+1}$ explicit. We recall some generalities.

Suppose $C_g$ is a smooth projective curve of genus $g$ over $\C$ with $\uprho\colon \upmu_i \into \Aut(C_g)$ an action of $\upmu_i$. The components of its integral cohomology 
\[
\H^\bullet(C_g, \Z) \simeq \Z \oplus \Z^{2g}[1] \oplus \Z[2] \simeq \H_\bullet(C_g,\Z)^\vee[2],
\]
carry an induced action $\H^i(\uprho,\Z)$ of $\upmu_i$. 
Because the action preserves effective classes, it is trivial on $\H^0(C_g,\Z)$ and $\H^2(C_g,\Z)$, and hence concentrated on the middle cohomology. 
The middle cohomology of a smooth projective curve has a weight $1$ pure Hodge structure
\[
\H^1(C_g, \Z) \otimes_\Z \C = \H^n(C_g, \C) \simeq \H^{1,0}(C_g) \oplus \H^{0,1}(C_g),
\]
where the summands $\H^{p,q}(C_g)$ are isomorphic to $\H^q(C_g, \Omega_{C_g}^p)$ by the degeneration of the Hodge-to-de Rham spectral sequence.
The action of $\upmu_i$ restricts to an action on $H^{p,q}(C_g)$, turning it into a $g$-dimensional representation.
The representations decompose into a direct sum of 1-dimensional irreducible representations $\upxi^j$ on which $\upmu_i$ acts by weight $j$, and the two summands are conjugate $\H^1(C_g,\O_{C_g}) \simeq \overline{\H^0(C_g,\Omega^1_{C_g})}$.

\begin{proof}[Proof of Proposition \ref{prop:Hodgemon}]
  The curve $D_{4a}$ is a ramified cover $q\colon D_{4a} \to \P^1$ of degree $4a$. By Birkhoff--Grothendieck, the push-forward $q_*\O_{D_{4a}}$ splits as a direct sum $\bigoplus_{i=0}^{4k} L_i$ of line bundles $L_i$ on $\P^1$. It follows from \cite[Lemma 3.14]{Steenbrink77} that this decomposition can be chosen to be invariant with respect to the monodromy action, with $\upmu_{4a}$ acting with weight $i$ on $L_i$. Furthermore, the degrees of these linebundles are determined by the multiplicities of the components that intersect $E_{4a}$ in the diagram of proposition \ref{lem:blowup1}. The components $E_{4a}$ intersects the components $L_2, E_{2a}, E_{2a-1}$ of multiplicities $1,2a,2a-1$ each in a single point, so Steenbrink's formula yields
  \[
  L_i \colonequals \O_{\P^1}\left(-i + \floor{\frac{i}{4a}} + \floor{\frac{2a \cdot i}{4a}} + \floor{\frac{(2a-1) \cdot i}{4a}} \right),
  \]
  where $\floor{-}\colon \Q \to \Z$ is the floor function. Some pleasant modular arithmetic shows that 
  \[
  L_i \simeq \begin{cases} \O_{\P^1}(-1) & i = 2j, \\
    \O_{\P^1}(-1) & i = 2j-1, j \leq a, \\  
    \O_{\P^1}(-2) & i = 2j-1, j > a.
  \end{cases}
  \]
  Because the morphism $q\colon D_{4a} \to \P^1$ is affine, $\H^1(D_{4a}, \O_{D_{4a}}) \simeq \H^1(\P^1, q_*\O_{D_{4a}})$ and the Hodge decomposition contains exactly a summand $\upxi^i$ for each $i$ such that $L_i \simeq \O_{\P^1}(-2)$:
  \[
  \H^1(D_{4a}, \O_{D_{4a}}) = \upxi^{2a+1} \oplus \upxi^{2a+3} \oplus \cdots \oplus \upxi^{4a-3} \oplus \upxi^{4a-1}.
  \]
  The second summand $\H^0(D_{4a},\Omega^1_{D_{4a}})$ is obtained by duality. Likewise, the curve $D_{2b+1}$ is a ramified cover $q\colon D_{2b+1} \to \P^1$ of degree $2b+1$ and the decomposition $q_*\O_{D_{2b+1}} = \bigoplus_{i=0}^{2b+1} L_i$ can be chosen invariantly, with $\upmu_{2k+1}$ acting on $L_i$ by weight $i$. Because the curve $E_{2b+1}$ intersects $E_{2b-1}$ of multiplicity $2b-1$ and has a double intersection with the curve $L_2$, which has multiplicity $1$, these line bundles are 
  \[
    L_i \colonequals \O_{\P^1}\left(-i + 2\floor{\frac i{2b+1}} + \floor{\frac{(2b-1) \cdot i}{2b+1}} \right)
    \simeq \begin{cases} \O_{\P^1}(-1) & i \leq b, \\  
      \O_{\P^1}(-2) & i > b.
    \end{cases}
  \]
  Taking the first cohomology once more, one finds
  \[
  \H^1(D_{2b+1},\O_{D_{2b+1}}) = \upxi^{b+1} \oplus \upxi^{b+2} \oplus \cdots \oplus \upxi^{2b-1} \oplus \upxi^{2b},
  \]
  with $\H^0(D_{2b+1},\Omega^1_{D_{2b+1}})$ being the dual representation.
\end{proof}

\section{Preservation of Superpotentials under Auto-Equivalences}\label{sec:autoequiv}
The goal of this section is to prove Proposition \ref{prop:BPSequiv} using an enhancement of the derived category $\D^b_\nilp(A)$ of nilpotent modules over the Jacobi algebra $A=\Jac(Q,W)$.
This enhancement captures the Calabi--Yau structure of $A$, and endows modules $M \in \nilp A \subset \D^b_\nilp(A)$ with a quiver with minimal potential $(\cQ_{M},\cW_{M})$, which expresses its deformation theory.
The latter potential determines the contribution of $M$ and its self-extensions to to the DT theory of $A$.

Working with this enhancement, it becomes possible to compare potentials of different objects $M$ and $N$ related by $N \simeq F(M)$ via a \emph{standard} derived equivalence $F$, i.e. a derived equivalence that lifts to the enhancement of $\D^b(A)$.
This includes in particular all tilting functors defined in \S\ref{sec:stables}. 
Such an equivalence has an action on Hochschild homology
\[
  \HH_\bullet(F)\colon \HH_\bullet(A)\to\HH_\bullet(A),
\]
and we formulate a sufficient condition for the potentials to be preserved by $F$ in terms of this action.
We find the following theorem, which applies to a much more general setting than that of Proposition \ref{prop:BPSequiv}.
\begin{theorem}\label{thm:CYstructPots}
  Let $A=\Jac(Q,W)$ be a Jacobi algebra which is finite over a central Noetherian subring $R \subset A$.
  Suppose $F\colon \D^b(A) \to \D^b(A)$ is an $R$-linear standard equivalence such that
  \[
    \HH_3(F) = \uplambda \in\C^\times.
  \]
  Then for every pair of nilpotent modules $M,N$ such that $\End_A(M) \simeq \C$ and $F(M) \simeq N$, the potentials $\cW_M$ and $\uplambda\cdot \cW_N$ are equivalent via a formal change of variables.
\end{theorem}
The partition functions $\Upphi^{\uptheta_{\curve,n}}(t)$ and $\Upphi^{\uptheta_{2\curve,n}}(t)$ in \S\ref{sec:DTcalc} counts semistable modules of phase $\uptheta_{\curve,n}$ and $\uptheta_{2\curve,n}$ respectively. Because these semistables arise as extensions of a unique stable module $M$, these can be computed by applying the integration map to the stack of extensions of $M$.
The theorem has the following consequence for such contributions.
\begin{corollary}\label{cor:equivBPS}
  In the setting of Theorem \ref{thm:CYstructPots}, let $\cP_{M,k},\cP_{N,k} \subset \fC_{Q,W}$ be the substacks of $k$-fold self extensions of the modules $M$ and $N$ respectively.
  Then there is an equality
  \[
    \int_{\fP_{M,k}}\kern-5pt \upphi_{\tr(W)} = \int_{\fP_{N,k}}\kern-5pt \upphi_{\tr(W)},
  \]
  of the associated contributions to the DT theory of $(Q,W)$.
\end{corollary}
Because the stable modules of phases $\uptheta_{\curve,n},\uptheta_{2\curve,n}$ are related to the simples by a tilting equivalence, the proof of Proposition \ref{prop:BPSequiv} follows easily from Corollary \ref{cor:equivBPS}, provided the tilting functors act as a scalar on Hochschild homology.
We show in Proposition \ref{prop:everyautoequiv} that this is always holds if the units of $R$ coincide with the nonzero scalars.

To prove the theorem, we relate the Hochschild homology of a smooth DG-enhancement $\cA$ of $\D^b(A)$ with (a version of) the Hochschild cohomology of a proper DG-enhancement $\cN$ of $\D^b_\nilp(A)$, and to show that this relation is compatible with derived equivalences. This relation comes from a pairing on Hochschild homology reviewed in \S\ref{ssec:HH} and is related to Koszul duality as we show in \S\ref{ssec:KozDu}. The potentials are defined on the minimal model of $\cN$, which (as we explain in \S\ref{ssec:cycminmod}) is given by a \emph{cyclic} $A_\infty$-category of twisted complexes. The cyclic inner product on this category expresses the Calabi--Yau property, and is the crucial additional structure which allows one to define the potentials, as we recall in \S\ref{ssec:cyc}.

\subsection{Hochschild homology}\label{ssec:HH}
We recall the related notions of Hochschild (co-)homology and Calabi--Yau structures on DG-categories and $A_\infty$-categories. Detailed introductions to the theories of DG and $A_\infty$-categories can be found in \cite{Keller06} and \cite{LefevreHasegawa03} respectively. In what follows we work over the base-field $\C$, all DG-/$A_\infty$-categories are assumed to be small and all $A_\infty$-categories are assumed to have strict units. If $\cC$ is a DG-/$A_\infty$-category we write $\DGPerf \cC$ for its DG-category of perfect complexes.

Given a DG-/$A_\infty$-category $\cC$, the Hochschild complex is (see e.g. \cite[\S 5.3]{Keller06})
\[
  \bC(\cC) \colonequals \left(\bigoplus_{k\geq 0} \bigoplus_{c_i\in\Ob\cC} \cC(c_1,c_0) \otimes (\cC(c_2,c_1)\otimes\ldots\otimes\cC(c_0,c_k)), b\right)
\]
where the differential $b$ is given by application of the composition $\circ$ and differential $\d$ if $\cC$ is a DG-category, and involves also the higher multiplications in case $\cC$ is an $A_\infty$-category, see e.g. \cite{Ganatra12} or the appendix to \cite{Sheridan16}. Its homology is the \emph{Hochschild homology}
\[
  \HH_\bullet(\cC) \colonequals \H^{-\bullet}\kern-1pt\bC(\cC).
\]
For a DG-category $\cC$, the complex $\bC(\cC)$ is an explicit model for the derived tensor product $\cC \Lotimes_{\cC^e} \cC$ over the enveloping DG-category $\cC^e = \cC^\op \otimes_\C \cC$, where the $\cC^e$-module structure on $\cC$ is via the obvious bimodule action. 
If $\cC$ is a \emph{smooth} DG-category, in the sense that $\cC$ is perfect as a $\cC^e$-module, then there is a duality
\[
  \HH_k(\cC) \simeq \H^0(\cC\Lotimes_{\cC^e}\cC[-k]) \simeq \Hom_{\D(\cC^e)}(\cC^!,\cC[-k]),
\]
where $\cC^! \colonequals \RHom_{\cC^e}(\cC,\cC^e)$ denotes the derived bimodule dual, so that cycles in $\HH_k(\cC)$ can be interpreted as morphisms.

Write $(-)^* = \Hom_\C(-,\C)$ for the \emph{linear} dual, then the \emph{Hochschild cohomology} with coefficients in $\cC^*$ is the cohomology of the dual complex:
\[
  \HH^\bullet(\cC,\cC^*) \colonequals \H^{\bullet}\kern-1pt(\bC(\cC)^*)
\]
Recall that $\cC$ is \emph{proper} if the cohomology $\H^\bullet\kern-1pt\cC(c,c')$ of the underlying complex is finite dimensional for all $c,c'\in\Ob\cC$. For proper DG-/$A_\infty$-categories one can again identify the cohomology classes in $\HH^\bullet(\cC,\cC^*)$ with morphisms through the adjunction:
\[
  \HH^k(\cC,\cC^*) \simeq \H^k\kern-1pt\Hom_\C(\cC\Lotimes_{\cC^e} \cC,\C) \simeq \Hom_{\D(\cC^e)}(\cC,\cC^*[k]).
\]
The Hochschild (co-)homology can be used to define the two (dual) versions of the Calabi-Yau property.
\begin{define}
  A (weak) \emph{left $k$-Calabi--Yau structure} on a smooth DG-category $\cC$ is a cycle $\upnu\in\bC_k(\cC)$ such that $[\upnu] \in \Hom_{\D(\cC^e)}(\cC^!,\cC[-k])$ is an isomorphism.
\end{define}
\begin{define}
  A (weak) \emph{right $k$-Calabi--Yau structure} on a proper DG-/$A_\infty$-category $\cC$ is a cocycle $\upxi\in (\bC_k(\cC))^*$ such that $[\upxi]\in\Hom_{\D(\cC^e)}(\cC,\cC^*[-k])$ is an isomorphism.
\end{define}
Suppose $F\colon \cC\to \cD$ is a DG-/$A_\infty$-functor, then application of $F$ defines a chain map $\bC(F) \colon \bC(\cC) \to \bC(\cD)$. For a DG-functor the map $\bC(F)$ is defined pointwise:
\[
  \cC(c_1,c_0)\otimes \ldots\otimes\cC(c_0,c_k) \xrightarrow{F\otimes\ldots\otimes F} \cD(F(c_1),F(c_0))\otimes\ldots\otimes\cD(F(c_0),F(c_k)), 
\]
and for an $A_\infty$-functor $F = (F_k)_{k\geq 1}$ it also involves the higher maps (see \cite[\S 2.9]{Ganatra12}). We denote its dual as $\bC(F)^* \colon \bC(\cD) \to \bC(\cC)$, and write
\[
  \HH_\bullet(F) \colon \HH_\bullet(\cC)\to\HH_\bullet(\cD),\quad 
  \HH^\bullet(F)\colon\HH^\bullet(\cD)\to\HH^\bullet(\cC).
\]
for the induced maps on (co-)homology. If a DG category $\cC$ is smooth and proper, it admits a perfect pairing $\HH_\bullet(\cC) \simeq \HH^\bullet(\cC,\cC^*)$ which is compatible with DG-functors (see the work of Shklyarov \cite{Shklyarov13}),
and identifies left and right Calabi--Yau structures. 
This is the DG-categorical analogue of the Mukai pairing for smooth projective schemes as in the work of Caldararu \cite{Caldararu03}.

In the noncompact Calabi-Yau setting we work in, all DG-categories of interest (e.g. enhancements of $\D^b(\mod A)$) are smooth but \emph{not} proper. There nonetheless exists a pairing when restricting to a subcategory $\cN\subset \cC$ of \emph{compactly supported} objects, as shown by Brav--Dyckerhoff \cite{BD19}. Recall that an object $p\in \cC$ is compactly supported if $\cC(c,p) \in \DGPerf\C$ for all $c\in\cC$. If $\cN \subset \cC$ is the full DG-subcategory on a set of compactly supported objects, then the diagonal bimodule $\cC$ defines a morphism of DG categories
\[
  \cC(-,-) \colon \cC^\op \otimes \cN \to \DGPerf \C,
\]
and hence a chain map $\bC(\cC^\op \otimes \cN) \xrightarrow{\bC(\cC(-,-))} \bC(\DGPerf \C)$ on the associated Hochschild complexes. 
Recall (see e.g. \cite[\S 4.2.1]{Loday92}) that the Hochschild homology admits a shuffle product $\nabla\colon \bC(\cC^\op)\otimes \bC(\cN) \to \bC(\cC^\op\otimes \cN)$, which maps a pair of classes $\textbf{f} = f_0[f_1\mid\ldots\mid f_n]$, $\textbf{g} = g_0[g_1\mid\ldots\mid g_m]$ (written in bar notation) to the class
\[
  \nabla(\textbf{f} \otimes \textbf{g})
  = \sum_\upsigma  \pm (f_0\otimes g_0) [\upsigma_1 | \ldots | \upsigma_{m+n}]
\]
where the sum is over the $(n,m)$-shuffles of the terms $f_1\otimes 1,\ldots f_n\otimes 1, 1\otimes g_1,\ldots,1\otimes g_m$. 
This shuffle product induces the following pairing between the Hochschild complexes
\begin{equation}\label{eq:mukaipairing}
  \bC(\cC^\op) \otimes \bC(\cN) 
  \xrightarrow{\nabla} \bC(\cC^\op \otimes \cN)
  \xrightarrow{\bC(\cC(-,-))} \bC(\DGPerf \C),
\end{equation}
and this yields a pairing on cohomology:
\[
  \<-,-\>_\cN \colon \HH_\bullet(\cC^\op) \otimes \HH_\bullet(\cN) \to \HH_\bullet(\DGPerf\C) \simeq \HH_\bullet(\C) \simeq \C.
\]
If $\upnu\in \HH_d(\cC)\simeq \HH_d(\cC^\op)$ is the Hochschild class of a (weak) left Calabi-Yau structure on $\cC$, then its dual $\<\upnu,-\>\in (\HH_\bullet(\cN))^* \simeq \HH^\bullet(\cN,\cN^*)$ is the class of a (weak) right Calabi-Yau structure on $\cN$ (see \cite[Theorem 3.1]{BD19}); although not every right Calabi--Yau structure necessarily arises in this way. 
The following lemma shows that the pairing is preserved under suitable DG functors.

\begin{lemma}\label{lem:pairinginv}
  Suppose $F\colon \cC \to \cD$ is a quasi-fully-faithful DG-functor that maps a compactly supported subcategory $\cN \subset \cC$ to $\cN' \subset \cD$, then the pairings satisfy
  \[
    \<\HH_\bullet(F^\op)(-), \HH_\bullet(F)(-)\>_{\cN'} = \<-,-\>_{\cN}.
  \]
\end{lemma}
\begin{proof}
  Given elements $\textbf{f} = f_0[f_1|\ldots|f_n]\in \bC(\cC^\op)$ and $\textbf{g} = g_0[g_1|\ldots|g_n]\in\bC(\cN)$, the definition of the shuffle product directly yields
  \[
    \begin{aligned}
      \nabla \circ (\bC(F^\op)\otimes \bC(F))(\textbf{f}\otimes\textbf{g}) &= \sum_\upsigma \pm (F(f_0)\otimes F(g_0))[(F^\op\otimes F)(\upsigma_1)|\ldots | (F^\op\otimes F)(\upsigma_n)]
      \\&= \bC(F^\op\otimes F) \circ \nabla(\textbf{f}\otimes\textbf{g})
    \end{aligned}
  \]
  which shows that $\nabla \circ (\bC(F^\op)\otimes \bC(F)) = \bC(F^\op\otimes F) \circ \nabla$.
  Because $F$ is quasi-fully-faithful, for all $M\in\cC$, $N\in\cN$ there are quasi-isomorphisms
  \[
    F_{M,N} \colon \cC(M,N) \to \cD(F(M),F(N)),
  \]
  which are natural in $M,N$.
  This data defines a DG-natural transformation between the functors $\cC(-,-)$ and $\cD(F^\op(-),F(-))$, i.e. a homotopy equivalence.
  Composing with $\nabla$ yields another homotopy equivalence, and because homotopic functors induce homotopic chain maps by a result of Keller \cite[Lemma 3.4]{Keller99} it follows that
  \begin{align*}
    \<\HH_\bullet(F^\op)(-),\HH_\bullet(F)(-)\>_{\cN'} &=                                                        \H^0(\bC(\cD(-,-)) \circ \nabla \circ (\bC(F^\op)\otimes \bC(F)))
    \\&\simeq \H^0(\bC(\cD(F^\op(-),F(-))) \circ \nabla)
    \\&\simeq \H^0(\bC(\cC(-,-)) \circ \nabla)
    = \<-,-\>_{\cN}.    
    \qedhere
  \end{align*}
\end{proof}

Some of the DG-categories we consider are equipped with an additional $R$-linear structure over a commutative $\C$-algebra $R$.
This $R$-linear structure induces an $R$-module structure on the Hochschild homology (defined over $\C$), where $r\in R$ acts on an element of $\bC(\cC)$ by
\[
  f_0 [f_1\mid \ldots \mid f_n] \mapsto rf_0 [f_1\mid \ldots \mid f_n],
\]
The action is compatible with the Hochschild differential, so that $\HH_\bullet(\cC)$ is a graded $R$-module.
Likewise, an $R$-linear DG-functor $F\colon \cC\to\cD$, induces $R$-linear chain maps $\bC(F)\colon \bC(\cC)\to\bC(\cD)$, and similarly for the maps $\HH_\bullet(F)$ and $\HH^\bullet(F)$.
The $R$-linear structure is compatible with the pairing in Lemma \ref{lem:pairinginv} in the following sense.
\begin{lemma}\label{lem:Rpairing}
  If $\cC$ is an $R$-linear DG-category $\cN \subset \cA$ a subcategory of compactly supported objects,
  then the pairing is $R$-linear: $\<r\cdot -,-\>_\cN = \<-,r\cdot -\>_\cN$ for all $r\in R$.
\end{lemma}
\begin{proof}
  For clarity, we write $G\colon \cC^\op\otimes\cN \to \DGPerf \cC$ for the functor that maps a pair of morphisms $(f\colon c'\to c,g\colon p\to p')$ in $\cC^\op\otimes\cN$ to the map
  \[
    G(f,g) \colon \cC(c,p) \to \cC(c',p'),\quad  h\mapsto  g\circ h \circ f.
  \]
  By inspection this satisfies $G(r\cdot f, g) = G(f,r\cdot g)$ because the composition commutes with the $R$-action. Applying the shuffle product now yields
  \begin{align*}
    (\bC(G) \circ \nabla)(r\cdot \mathbf{f},\mathbf{g}) 
    &= \sum \pm G(rf_0,g_0)[G(\upsigma_1)\mid\ldots\mid G(\upsigma_{n+m})]\\
    &= \sum \pm G(f_0,rg_0)[G(\upsigma_1)\mid\ldots\mid G(\upsigma_{n+m})]\\
    &= (\bC(G) \circ \nabla)(\mathbf{f},r\cdot \mathbf{g}).
  \end{align*}
  The same identity then holds in cohomology, making $\<-,-\>_\cN$ an $R$-linear pairing.
\end{proof}

\subsection{Koszul duality}\label{ssec:KozDu}
\newcommand{\pS}{{\mathbf p}S}

Let $A$ be a (module-)finite algebra over a commutative Noetherian $\C$-algebra $R$, and assume it is homologically smooth over $\C$. Then the DG-category of perfect complexes $\cA \colonequals \DGPerf A$ is a smooth DG-category which is moreover $R$-linear. Given a maximal ideal $\m\subset R$ there is a full DG-subcategory $\cN \subset \cA$ of objects with cohomology supported on $\m\in\Spec R$,
i.e. $\H^0\kern-3pt\cN = \D^\perf_\m(A) \subset \D^\perf(A)$.
These are compactly supported objects and hence induce a pairing $\<-,-\>_\cN$ as in \eqref{eq:mukaipairing}.

The objects in $\D^\perf_\m(A)$ have finite length: they are obtained as a finite extension of shifts of the 
simple $A$-modules supported over $\m$. Hence $\D^\perf_\m(A)$ is generated by some finite sum $S = \bigoplus_{i} S_i$ of simple modules.
Let $\pS \in \cN$ be the associated perfect complex, so that the DG-algebra
\[
  E \colonequals \cA(\pS,\pS),
\]
computes $\REnd_A(S)$.
Because $S$ generates, the embedding $E \to \cN$ is a Morita equivalence, hence defines a quasi-isomorphism $\bC(E) \to \bC(\cN)$ between the Hochschild complexes.
Likewise, $\cA^\op$ is Morita equivalent to $\cA^\op(A,A) \simeq A$, giving a quasi-isomorphism $\bC(A) \to \bC(\cA^\op)$. 
The pairing therefore restricts to a pairing between Hochschild homologies of (DG-)algebras
\[
  \<-,-\>_\cN \colon \HH_\bullet(A) \otimes \HH_\bullet(E) \to \C,
\]
and by adjunction this gives a morphism of $R$-modules
\[
  \Upsilon\colon \HH_\bullet(A) \to \HH_\bullet(E)^* = \HH^\bullet(E,E^*).
\]
In general this map fails to be an isomorphism (certainly for flops) but this is to be expected: we may as well have replaced $A$ by a suitable localisation.
However, one can replace $A$ by its $\m$-adic completion, in which case Van den Bergh \cite[Corollary D.2]{VandenBergh15} shows the analogous map to be an isomorphism due to Koszul duality.

\begin{prop}\label{prop:completion}
  The map $\Upsilon$ factors through the completion of $\HH_\bullet(A)$ as
  \[
    \Upsilon \colon \HH_\bullet(A) \to \HH_\bullet(A) \otimes_R \widehat R \simeq \HH^\bullet(E,E^*).
  \]
\end{prop} 
\begin{proof}
  As remarked before, the Hochschild homology and its dual compute derived bimodule morphisms: there are
  $R$-linear isomorphisms
  \[
    \HH_\bullet(A) \simeq \RHom_{A^e}(A^!,A),\quad \HH^\bullet(E,E^*) \simeq \RHom_{E^e}(E,E^*).
  \]
  It follows from the proof of \cite[Theorem 3.1]{BD19}, the composition of these isomorphisms with the map $\Upsilon \colon \HH_\bullet(A) \to \HH^\bullet(E,E^*)$
  is induced by the following derived functor
  \[
    \RHom_A(S,\RHom_A(-,S)) \colon \D^\perf(A^e) \to \D^\perf(E^e)^\op,
  \]
  which maps $A$ to $E$ and $A^!$ to $E^*$. Let $\widehat R$ be the completion of $R$ at $\m$, then because $R$ is Noetherian we may identify the completion $\widehat M$ of any finitely generated $R$-module $M$ with $M\otimes_R\widehat R$. In particular, the completion of $A$ is the base-change $\Nccr \simeq A\otimes_R \widehat R$.
  This completion is a pseudocompact algebra, which Van den Bergh shows \cite{VandenBergh15} is Koszul dual to $E$. Let $\D^\perf_{\mathrm{pc}}(\Nccr^e)$ denote the category of perfect complexes of pseudocompact $\Nccr$-bimodules (see e.g. the appendix of \cite{KY11}). By Koszul duality, the functor
  \begin{equation}\label{eq:botrow}
    \RHom_{\Nccr}(S,\RHom_{\Nccr}(-,S)) \colon \D^\perf_{\mathrm{pc}}(\Nccr^e) \to \D^\perf(E^e)^\op,
  \end{equation}
  is an equivalence of triangulated categories. In particular, it defines an isomorphism $\RHom_{\Nccr^e}(\Nccr^!,\Nccr) \to \RHom_{E^e}(E,E^*)$, making the following diagram of $R$-modules commute:
  \[
    \begin{tikzpicture}
      \node (A) at (0,0) {$\RHom_{A^e}(A^!,A)$};
      \node (B) at (0,-\gr) {$\RHom_{{\Nccr}^e}(\Nccr^!, \Nccr)$};
      \node (C) at (8,-\gr) {$\RHom_{E^e}(E,E^*)$};
      \draw[->] (A) to[bend left=10,edge label=${\scriptstyle \RHom_A(S,\RHom_A(-,S))}$,near start] (C);
      \draw[->] (B) to[edge label=${\scriptstyle \RHom_{\Nccr}(S,\RHom_{\Nccr}(-,S))}$] (C);
      \draw[->] (A) to[edge label'=${\scriptstyle -\otimes_R \widehat R}$] (B);
    \end{tikzpicture}
  \]
  where $-\otimes_R\widehat R$ is the map induced by the completion functor (which is exact).
  The $R$-module $\RHom_{A^e}(\Nccr^!,\Nccr)$ is obtained by base-change from the Hochschild homology:  
  \[
    \RHom_{\Nccr^e}(\Nccr^!,\Nccr) 
    \simeq \RHom_{A^e}(A^!,A) \otimes_R \widehat R \simeq \HH_\bullet(A) \otimes_R \widehat R.
  \]
  Let $K$ denote the composition of this isomorphism with \eqref{eq:botrow}, then $\Upsilon$ is the composition
  \[
    \HH_\bullet(A) \xrightarrow{-\otimes_R\widehat R} \HH_\bullet(A)\otimes_R\widehat R \xrightarrow{\ K\ } \HH^\bullet(E,E^*).\qedhere
  \]
\end{proof}
Suppose $F\colon \cA\to\cA$ is an $R$-linear quasi-equivalence preserving $\cN$, then it induces $R$-linear endomorphisms $\HH_\bullet(F)$ on $\HH_\bullet(A)\simeq \HH_\bullet(\cA^\op)$ and $\HH^\bullet(F)$ on $\HH^\bullet(E,E^*) \simeq \HH^\bullet(\cN,\cN^*)$.
By the previous proposition, the actions are related as follows:
\begin{prop}\label{prop:completeaction}
  Let $F\colon \cA \to \cA$ be an $R$-linear quasi-equivalence
  preserving $\cN$, then 
  \[
    \HH^\bullet(F) = K\circ (\HH_\bullet(F)^{-1} \otimes_R \widehat R) \circ K^{-1}
  \]
  for $K\colon \HH_\bullet(A)\otimes_R \widehat R \to \HH_\bullet(E,E^*)$ the isomorphism
  from the previous proposition.
\end{prop}
\begin{proof}
  By Lemma \ref{lem:pairinginv} the pairing $\<-,-\>_\cN$ is invariant under the simultaneous action
  of $\HH_\bullet(F)$ on both arguments. Hence, by adjunction the map $\Upsilon$ satisfies
  \[
    \HH^\bullet(F) \circ \Upsilon \circ \HH_\bullet(F) = \Upsilon,
  \]
  for any quasi-fully faithful functor $F$. If $F$ is a quasi-equivalence, 
  then $\HH_\bullet(F)$ is moreover invertible, so that
  \begin{equation}\label{eq:commHact}
    \HH^\bullet(F) \circ \Upsilon = \Upsilon \circ \HH_\bullet(F)^{-1}.
  \end{equation}
  Let $c\colon \HH_\bullet(A) \to \HH_\bullet(A)\otimes_R \widehat R$ denote the completion map.
  Then by Proposition \ref{prop:completion} above, there is a factorisation $\Upsilon = K\circ c$,
  and we can consider the following diagram of $R$-modules
  \[
    \begin{tikzpicture}[baseline=(current bounding box.center)]
      \node (A) at (0,0) {$\HH_\bullet(A)$};
      \node (B) at (3,0) {$\HH_\bullet(A) \otimes_R\widehat R$};
      \node (C) at (6,0) {$\HH_\bullet(E,E^*)$};
      \draw[->] (A) to[edge label=$\scriptstyle c$] (B);
      \draw[->] (B) to[edge label=$\scriptstyle K$] (C);
      \node (D) at (0,-\gr) {$\HH_\bullet(A)$};
      \node (E) at (3,-\gr) {$\HH_\bullet(A) \otimes_R\widehat R$};
      \node (F) at (6,-\gr) {$\HH_\bullet(E,E^*)$};
      \draw[->] (D) to[edge label=$\scriptstyle c$] (E);
      \draw[->] (E) to[edge label=$\scriptstyle K$] (F);
      \draw[->] (A) to[edge label=$\scriptstyle \HH_\bullet(F)^{-1}$] (D);
      \draw[->] (B) to[edge label=$\scriptstyle \HH_\bullet(F)^{-1} \otimes_R \widehat R$] (E);
      \draw[->] (C) to[edge label=$\scriptstyle \HH^\bullet(F)$] (F);
    \end{tikzpicture}
  \]
  The outer compositions agree by \eqref{eq:commHact}, and by the universal property
  of the completion $\HH_\bullet(F)^{-1}\otimes_R\widehat R$ is the unique map which makes 
  the left inner square commute.
  Hence the right-inner square also commutes and the result follows.
\end{proof}

\begin{corollary}\label{cor:raction}
  Suppose $F\colon\cA\to\cA$ is an $R$-linear quasi-equivalence and which acts
  on $\HH_d(\cA)$ as multiplication $\HH_d(F) = r\cdot$ by a unit $r\in R^\times$.
  Then $\HH^{-d}(F) = r^{-1}\cdot$.
\end{corollary}

\begin{remark}
  In the context of CY structures,
  Proposition \ref{prop:completion} shows that any right CY structure for the objects $\cN$ supported on $\m$
  is determined by a left CY structure defined in a formal neighbourhood of $\m$, and that
  a `global' left CY structure restricts to this formal neighbourhood.
  Although not every right CY structure for $\cN$ is the image of a global left CY structure,
  Proposition \ref{prop:completeaction} shows that the action of a global equivalence
  on the right CY structures on $\cN$ is nonetheless determined
  by its action on the global left CY structures.
\end{remark}

\subsection{Cyclic \texorpdfstring{$A_\infty$}{A-infinity}-categories}\label{ssec:cyc}
\newcommand{\inftymod}{\,\overset\infty{\vphantom{:}\smash\mod}\,}

In order to endow the properly supported objects in our 3-CY categories with a \emph{potential}, we use $A_\infty$-categories equipped with a \emph{cyclic structure}, which are a strict version of a right Calabi--Yau structure.
In what follows all $A_\infty$-categories/functors/modules are strictly unital.

Given an $A_\infty$-category $\cC$, we write $\cC\inftymod\cC$ for its DG-category of $A_\infty$-bimodules. 
The Hom-complex between bimodules $M,N \in \cC\inftymod\cC$ is of the form
\[
  \cC\inftymod\cC(M,N) \colonequals \left({\textstyle\bigoplus_{i,j\geq 0} }\Hom_\C(\cC^{\otimes i}\otimes M \otimes \cC^{\otimes j}, N),\d\right),
\]
and so any degree $k$ bimodule map $\upalpha\colon M\to N[k]$ is given by its components $\upalpha_{i,j}$.
Any $A_\infty$-category $\cC$ is a bimodule over itself, and so is its linear dual $\cC^*$ by pre-composition.
Given an $A_\infty$-functor $F\colon \cC\to \cD$ there is a pullback $F^*\colon \cD\inftymod\cD \to \cC\inftymod\cC$,
which identifies $F^*\kern-1pt\cM(c,c') = \cM(F(c),F(c'))$.
The functor also gives a morphism $F\colon \cC\to F^*\cD$ in a natural way, so that we may complete
any bimodule morphism $\upalpha\colon \cD\to\cD^*$ to a bimodule morphism $\cC\to \cC^*$ via the diagram
\[
  \begin{tikzpicture}
    \node (A) at (0,0) {$\cC$};
    \node (B) at (3,0) {$F^*\cD$};
    \node (C) at (0,-\gr) {$\cC^*$};
    \node (D) at (3,-\gr) {$F^*\cD^*$};    
    \draw[->] (A) to[edge label=${F}$] (B);
    \draw[->,dashed] (A) to (C);
    \draw[->] (B) to[edge label=${\upalpha}$] (D);
    \draw[->] (D) to[edge label'=${F^*}$] (C);
  \end{tikzpicture}
\]
in $\cC\inftymod\cC$. By slight abuse of notation we denote the dashed vertical arrow as $F^*\upalpha$.
Following Cho--Lee \cite{CL11}, a cyclic structure can be defined in this bimodule formulism as follows.

\begin{define}\label{def:cycstruct}
  Let $\cC$ be a finite dimensional $A_\infty$-category, by which we mean that $\cC(c,c')$ is a finite dimensional vectorspace for all $c,c'\in\Ob\cC$. A cyclic structure on $\cC$ is an $A_\infty$-bimodule homomorphism
  $\upsigma = (\upsigma_{i,j}) \colon \cC \to \cC^*[-3]$ such that:
  \begin{enumerate}
  \item the higher maps $\upsigma_{i,j}$ for $(i,j) \neq (0,0)$ vanish,
  \item for all $a,b\in\Ob\cC$ the map $\upsigma_{0,0}(a,b) \colon \cC(a,b) \to \cC(b,a)^*$ is an isomorphism,
  \item the dual $\upsigma^*\colon \cC^{**}[3] \to \cC^*$ is identified with $\upsigma$ via the isomorphism $\cC \simeq \cC^{**}$ and obvious shift
  \end{enumerate}
  Under these conditions the pair $(\cC,\upsigma)$ is a \emph{cyclic $A_\infty$-category}.
  A \emph{cyclic $A_\infty$-functor} $F\colon (\cC,\upsigma) \to (\cD,\upsigma')$ is an $A_\infty$-functors $F\colon \cC\to \cD$ such that $F^*\upsigma' = \upsigma$.
\end{define}
Objects in a cyclic $A_\infty$-category are endowed with a potential.
Let $(\cC,\upsigma)$ be a cyclic $A_\infty$-category and $T\in \Ob \cC$ an object with endomorphism $A_\infty$-algebra $\cC_T\colonequals \cC(T,T)$,
which has a cyclic structure $\upsigma|_T\colon \cC_T \to \cC_T^*$ given by the restriction of $\upsigma$.
Then the potential of $T$ is the noncommutative formal function
\[
\cW = \cW_T \in \left(\textstyle\bigoplus_{k\geq 1} (\cC_T^1)^{\otimes k}\right)^*
\]
which maps the $k+1$ tuple $f_0\otimes\ldots \otimes f_k$ of degree 1 elements to
\begin{equation}\label{eq:formalfunc}
  \cW(f_0,\ldots,f_k) \colonequals \upsigma(f_0)(m_k(f_1,\ldots,f_k)).
\end{equation}
After picking a basis cardinality $N= \dim_\C \cC_T^1$, the function $\cW$ can be identified with a formal potential $\cW\in\widehat{\C \cQ}_\cyc$ on the $N$-loop quiver $\cQ_T$. If $F\colon (\cC,\upsigma) \to (\cD,\upsigma')$ is a cyclic $A_\infty$-functor then Kajiura \cite[Proposition 4.16]{Kajiura07} shows that there is an induced formal homomorphism $\widehat{\C\cQ}_{F(T)} \to \widehat{\C\cQ}_T$ of the quiver algebras which maps $\cW_{F(T)}$ to $\cW_T$.

If an (ordinary) $A_\infty$-functor $F$ between cyclic $A_\infty$-categories fails to be cyclic, one can instead ask for a weaker ``homotopic'' version of the condition $F^*\upsigma' = \upsigma$.
Kontsevich--Soibelman \cite{KS09} have shown that an $A_\infty$-morphism satisfying such weaker condition can be made cyclic via a perturbation.
Given a cyclic $A_\infty$-category $(\cC,\upsigma)$, the map $\upsigma = \upsigma_{0,0}$ defines an cochain 
in the dual Hochschild complex via the isomorphism\footnote{N.B. one checks that this isomorphism is compatible with the Hochschild and bimodule differential. It extends to a quasi-isomorphism $\bC(\cC)^* \to \cC\inftymod\cC(\cC,\cC^*)$, see e.g. \cite{Ganatra12}.}
\[
  \bigoplus_{c,c'\in\Ob\cC} \Hom_\C(\cC(c,c'),\cC^*(c,c'))
  \simeq   \bigoplus_{c,c'\in\Ob\cC} \Hom_\C(\cC(c,c')\otimes \cC(c',c),\C) \subset \bC(\cC)^*,
\]
and its homotopy class coincides with a class $[\upsigma] \in \HH^{-3}(\cC,\cC^*)$.
If $F\colon \cC\to \cD$ is an $A_\infty$-functor onto a second cyclic $A_\infty$-category $(\cD,\upsigma')$, 
then $\HH^\bullet(F)[\upsigma']$ corresponds to the homotopy class of the bimodule morphism $F^*\upsigma'$.
One can therefore ask that the condition $F^*\upsigma' = \upsigma$ holds up to homotopy:
\[
  \HH^\bullet(F)[\upsigma'] = [\upsigma].
\]
If this condition holds, there exists an automorphism of $\cC$ that perturbs $F^*\upsigma$ to $\upsigma$. 
These automorphisms are described explicitly by Cho--Lee \cite{CL11} in the setting of $A_\infty$-algebras.

\begin{lemma}\label{lem:perturb}
  Let $(C,\upsigma)$ and $(D,\upsigma')$ be minimal cyclic $A_\infty$-algebras with
  an $A_\infty$-homo\-mor\-phism $f\colon C \to D$. Suppose $\HH^\bullet(f)([\upsigma']) = [\upsigma]$,
  then there exists an $A_\infty$-automorphism $g\colon C\to C$ such that the composition
  $f\circ g$ is a \emph{cyclic} $A_\infty$-homomorphism.
\end{lemma}
\begin{proof}
  See the proof of \cite[Proposition 7.4]{CL11}.
\end{proof}

This result applies to the endomorphism $A_\infty$-algebras of objects in a cyclic $A_\infty$-category.

\begin{lemma}\label{lem:hocycfun}
  Let $(\cC,\upsigma)$ and $(\cD,\upsigma')$ be minimal cyclic $A_\infty$-categories and $F\colon \cC \to \cD$ a quasi-fully-faithful $A_\infty$-functor which satisfies $\HH^\bullet(F)[\upsigma'] = [\upsigma]$. Then for every $M\in \cC$ there exists a cyclic $A_\infty$-algebra isomorphism $(\cC_M, \upsigma|_M) \to (\cD_{F(M)}, \upsigma'|_{F(M)})$.
\end{lemma}
\begin{proof}
  If an $A_\infty$-functor between minimal $A_\infty$-categories is quasi-fully-faithful,
  then the restrictions $F|_M \colon \cC_M \to \cD_{F(M)}$ are $A_\infty$-isomorphisms.
  By the perturbation Lemma \ref{lem:perturb} it suffices to shows that this preserves the Hochschild cohomology
  classes of the cyclic structures.
  Let $i_{F(M)}$ and $i_M$ denote the inclusion functors of $\cD_{F(M)}$ and $\cC_M$, then 
  \[
    \HH^\bullet(F|_M)[\upsigma'|_{F(M)}] = \HH^\bullet(i_{F(M)}\circ F|_M)[\upsigma'] = \HH^\bullet(i_M)(\HH^\bullet(F)[\upsigma']) = [\upsigma|_M].\qedhere
  \]
\end{proof}

We will also make use of the following auxilary lemma, which shows that the image $F(M)$ in the above lemma can be replaces by a quasi-isomorphic object.

\begin{lemma}\label{lem:cyccanon}
  Let $(\cC,\upsigma)$ be a minimal cyclic $A_\infty$-category and $M,N \in \Ob \cC$. If $M$ and $N$ are isomorphic in $\H^0\kern-2pt\cC$, then $(\cC_M, \upsigma|_M) \simeq (\cC_N, \upsigma|_N)$ as cyclic $A_\infty$-algebras.
\end{lemma}
\begin{proof}
  Recall that the $A_\infty$-category $\cC$ admits a DG-envelope $\cD$, which is a DG-category with the same set of objects as $\cC$, for which $\cC$ is a minimal model, and which comes equipped with a quasi-equivalence $\DGMod \cD^e \to \cC\inftymod\cC$ (see \cite[Lemme 2.5.2.2]{LefevreHasegawa03}).
  
  Let $u\in \cD(M,N)$ and $u^{-1} \in \cD(N,M)$ be the lifts of the isomorphism $M\xrightarrow{\sim} N$ in $\H^0\kern-2pt\cD = \H^0\kern-2pt\cC$ and its homotopy inverse, and consider the induced map
  \[
    u\circ - \circ u^{-1} \colon \cD_M \to  \cD_N.
  \]
  This DG-algebra homomorphism induces DG-bimodule morphisms $\overline u \colon \cD_M \to \cD_N$ and $\overline u^*\colon \cD_N^* \to \cD_M^*$. Let $\upalpha\colon \cD\to\cD^*[k]$ be a lift of the cyclic structure $\upsigma$, then
  \begin{align*}
    (\overline u^*\circ \upalpha|_N \circ \overline u)(f)(g) 
    &= \upalpha(u\circ f \circ u^{-1})(u\circ g\circ u^{-1})\\
    &= \upalpha(f\circ u^{-1}\circ u)(g \circ u^{-1}\circ u).
  \end{align*}
  Because $u^{-1}\circ u$ is homotopic to the identity, it follows that for any such $\upalpha\colon\cD\to\cD^*[k]$
  \[
    [\upalpha|_M] = [\overline u^* \circ \upalpha|_N \circ \overline u] = \HH^\bullet(\overline u)[\upalpha|_N].
  \]
  Because $\upalpha$ is a lift of $\upsigma$, the $A_\infty$-homomorphism $f\colon \cC_M \to \cC_N$ induced by $u\circ-\circ u^{-1}$ then satisfies $[\upsigma|_M] = \HH^\bullet(f)[\upsigma|_N]$. 
  The result then follows from Lemma \ref{lem:perturb}.
\end{proof}
\begin{remark}
  Note that the existence of a quasi-isomorphism $M\simeq N$ in $\cC$ is much stronger than
  the existence of a $A_\infty$-isomorphism $\cC_M \simeq \cC_N$, and the latter does not guarantee that the homotopy-cyclic condition holds.
\end{remark}

\subsection{Cyclic minimal models}\label{ssec:cycminmod}
Given a quiver with potential, it has a standard cyclic $A_\infty$-category associated to it.

\begin{define}
  Let $(Q,W)$ be a quiver with potential and for vertices $v,w\in Q_0$ denote by $Q(v,w)$ the set of arrows from $v$ to $w$.
  The $A_\infty$-category $\cD = \cD_{Q,W}$ has objects $\Ob\cD = Q_0$ and morphism spaces
  \[
    \cD(v,w) = \begin{cases}
      \C 1_v \oplus \C Q(w,v)^*[1] \oplus \C Q(v,w) [2] \oplus \C 1_v^*[3]& v = w\\
      \C  Q(w,v)^*[1] \oplus \C Q(v,w) [2]  & \text{otherwise}
    \end{cases}
  \]
  The higher products are required to have $1_v$ as strict units, for each $a \in Q(v,w)$
  \[
    m_2(a^*,a) = 1_v^*,\quad m_2(a,a^*) = 1_w^*,
  \]
  and for any chain of arrows $a_1,\ldots,a_k$ in $Q$ where $a_1 \in Q(v,w')$ and $a_k \in Q(v',w)$,
  \[
    m_k(a_k^*,\ldots,a_1^*) = \sum_{a \in Q(w,v)} c^a_{a_1\cdots a_k} \cdot a,
  \]
  where $c^a_{a_k\cdots,a_1}$ is the coefficient of $a_1\cdots a_k$ in the cyclic derivative $\del W/\del a \in \C Q$ of the potential. 
  All other compositions are zero, and in particular $\cD$ is minimal.
  We endow $\cD$ with the cyclic structure
  \[
    \upsigma(f)(g) = \tr_Q(m_2(f,g)),
  \]
  where $\tr_Q\colon \bigoplus_{v\in Q_0} \cD(v,v) \to \C$ is the linear map which send the generators $1_v^*$ to $1\in\C$ and maps all other generators to $0$.
\end{define}

One would like to extend the cyclic structure on $\cD_{Q,W}$ to the DG category $\DGPerf \cD_{Q,W}$ of perfect complexes, but this is not possible, as the latter does not have finite dimensional Hom-spaces.
Instead, one can take the $A_\infty$-category $\cC \colonequals \tw\cD_{Q,W}$ of twisted complexes, defined e.g. in the work of Lef\`evre-Hasegawa \cite[\S7]{LefevreHasegawa03}, which is a finite dimensional replacement for $\DGPerf\cD_{Q,W}$.
The cyclic structure extends the cyclic structure on $\cD_{Q,W}$, which we again denote by $\upsigma$, and every $T\in \cC$ is endowed with a potential $\cW_T \in {\cQ_T}_\cyc$.
\begin{theorem}[see {\cite[Theorem 7.1.3]{DavisonThesis}}]\label{thm:ben}
  Let $(Q,W)$ be a quiver with potential, $M\in\nilp \Jac(Q,W) \simeq \H^0\!\cC$ a module with $\End_{\Jac(Q,W)}(M) \simeq \C$, and $\cP\subset \fM_{Q,W}$ the locus of repeated self-extensions of $M$.
  Then there for any lift $T \in \cC$ of $M$
  \begin{equation}\label{eq:DTtwist}
    \int_{[\cP \to \fC]}\kern-2pt \upphi_{\tr(W)}|_\fC = \Upphi_{\cQ_T,\cW_T}(t^{[M]}),
  \end{equation}
  where $\cQ_T$ is the $N$-loop quiver of $T$ with potential $\cW_T$ as defined in \eqref{eq:formalfunc}.
\end{theorem}
The potential $\cW_T$ of a twisted complex is too coarse of an invariant to track under derived quasi-equivalences, and we will instead consider the associated \emph{minimal} potential.
Given $T \in \cC$ corresponding to a nilpotent module $M$ with $\End_A(M) \simeq \C$, the cyclic decomposition theorem \cite[Theorem 5.15]{Kajiura07} gives a splitting of the cyclic endomorphism $A_\infty$-algebra $\cC_T$ of $T$: there is a cyclic $A_\infty$-isomomorphism 
\begin{equation}\label{eq:minimalsplitting}
   (\cC_T,\upsigma|_T) \xrightarrow{\sim} (\H^\bullet\kern-2pt\cC_T, \upsigma_\min) \oplus (L_T, \upsigma'),
\end{equation}
where $(\H^\bullet\kern-2pt\cC_T, \upsigma_\min)$ is \emph{the cyclic minimal model}, a cyclic minimal $A_\infty$-algebra structure on the cohomology of $\cC_T$, and $(L_T,\upsigma')$ is a linearly contractible $A_\infty$, i.e. a cyclic $A_\infty$-algebra with $m_k = 0$ for $k\geq2$ and trivial cohomology. There is an associated decomposition
\begin{equation}\label{eq:basiscoho}
  (\cQ_T)_1 = \{x_1,\ldots,x_n\} \sqcup \{y_1,\ldots,y_{N-n}\}
\end{equation}
of the set of loops for the $N$-loop quiver $\cQ_T$, so that $x_i$ form a basis for the cohomology $\H^1\kern-2pt\cC_T$. Let $\cQ_{\min,T}$ be the subquiver of $\cQ_T$ generated by the $n$-loops $\{x_1,\ldots,x_n\}$, then the \emph{minimal potential} on $\cQ_{\min,T}$ is the noncommutative formal function
\[
  \cW_{\min,T} = \cW_{\min,T}(x_1,\ldots,x_n),
\]
defined as in \eqref{eq:formalfunc} from $(\H^\bullet\kern-2pt\cC_T,\upsigma_\min)$. Likewise, the linearly contractible summand $(L,\upsigma')$ has a potential $q= q(y_1,\ldots,y_{N-n})$, which is a nondegenerate quadratic form. The splitting \eqref{eq:minimalsplitting} induces formal isomorphism $\uppsi_T \colon \widehat{\C\cQ}_T \to \widehat{\C\cQ}_T$ such that $\uppsi_T(\cW_T) = \cW_{\min,T} + q$.

If $\cW_{\min,T}$ is again a \emph{finite} potential, the partition function $\Upphi_{\cQ_{\min,T},\cW_{\min,T}}(t)$ is well-defined, and Lemma \ref{lem:quadterms} implies that it is equal to the partition function $\Upphi_{\cQ_T,\cW_T}(t)$. Even if the minimal potential is a formal powerseries, it can still be used to compare the partition functions associated to two twisted complexes.
\begin{lemma}\label{lem:twequivmotives}
  Let $T_1,T_2 \in \Ob \cC$ be twisted complexes corresponding to nilpotent modules $M,N\in\nilp A$ with simple endomorphism algebras as above. If there exists a formal isomorphism $\uppsi\colon \widehat{\C\cQ}_{\min,T_1} \to \widehat{\C\cQ}_{\min,T_2}$ between their complete path algebras such that
  \[
    \uppsi(\cW_{\min,T_1}) = \uplambda\cdot \cW_{\min,T_2},
  \]
  for some scalar $\uplambda\in\C^\times$, then the partition functions of $T_1$, $T_2$ are equal:
  \begin{equation}\label{eq:formequiv}
    \Upphi_{\cQ_{T_1},\cW_{T_1}}(t) = \Upphi_{\cQ_{T_2},\cW_{T_2}}(t).
  \end{equation}
\end{lemma}
\begin{proof}
  Without loss of generality, we can identify the first $n$ loops in the $N_1$-loop quiver $\cQ_{T_1}$ with the first $n$ loops in the $N_2$-loop quiver $\cQ_{T_2}$, and write the splitting in \eqref{eq:basiscoho} as
  \[
    (\cQ_{T_1})_1 = \{x_1,\ldots,x_n\} \sqcup \{y_1,\ldots,y_{N_1-n}\},\quad
    (\cQ_{T_2})_2 = \{x_1,\ldots,x_n\} \sqcup \{z_1,\ldots,z_{N_2-n}\},
  \]
  so that $\uppsi$ is a formal automorphism of the quiver generated by the variables $x_i$. The potentials $\cW_{\min,T_1}$, $\cW_{\min,T_2}$ are functions in the variables $x_i$, and the quadratic terms $q_1$, $q_2$ of the linearly contractible summands for $T_1$ and $T_2$ are functions in the variables $y_i$ and $z_i$ respectively. Let $\cQ$ be the $N_1 + N_2 - n$-loop quiver with loops
  \[
    \cQ_1 = \{x_1,\ldots,x_n\} \sqcup \{y_1,\ldots,y_{N_1-n}\} \sqcup \{z_1,\ldots,z_{N_2-n}\}.
  \]
  Then the formal isomophisms $\uppsi_{T_1}$, $\uppsi_{T_2}$, $\uppsi$ lift to formal automorphisms of the completed path algebra $\widehat{\C\cQ}$ in the obvious way, and satisfy:
  \[
    \begin{aligned}
       \uppsi_{T_1}(\cW_{T_1} + \uplambda \cdot q_2) &= \cW_{\min,T_1} + q_1 + \uplambda\cdot q_2,\\
      \uppsi(\cW_{\min,T_1} + q_1 + \uplambda\cdot q_2) &= \uplambda\cdot \cW_{\min,T_2} + q_1 + \uplambda\cdot q_2,\\
      \uppsi_{T_2}(\uplambda\cdot \cW_{T_2} + q_1) &= \uplambda\cdot \cW_{\min,T_2} + q_1 + \uplambda\cdot q_2.
    \end{aligned}
  \]
  By inspection, the composition $\uppsi_{T_2}^{-1}\circ\uppsi\circ\uppsi_{T_1}$ maps $\cW_{T_1} + \uplambda\cdot q_2$ to $\uplambda\cdot \cW_{T_2} + q_1$. Hence, Lemma \ref{lem:motintformal} and Lemma \ref{lem:quadterms} imply that
  \[
    \Upphi_{\cQ_{T_1},\cW_{T_1}}(t) \overset{\text{\ref{lem:quadterms}}}= 
    \Upphi_{\cQ,\cW_{\min,T_1} + \uplambda\cdot q_2}(t) \overset{\text{\ref{lem:motintformal}}}= 
    \Upphi_{\cQ,\uplambda\cdot\cW_{\min,T_2} + q_1}(t) \overset{\text{\ref{lem:quadterms}}}= 
    \Upphi_{\cQ_{T_2},\uplambda\cdot \cW_{T_2}}(t).
  \]
  The partition function of $(\cQ_{T_2},\uplambda\cdot \cW_{T_2})$ is independent of $\uplambda$ as the vanishing cycle of $\tr(\uplambda\cdot \cW_{T_2}) = \uplambda\cdot \tr(\cW_{T_2})$ depends only on the zero locus of the function.
\end{proof}

In view of the above, it suffices to work with the cyclic minimal model $\H^\bullet\kern-2pt\tw\cD_{Q,W}$ of the cyclic $A_\infty$-category $\tw\cD_{Q,W}$. 

\subsection{Cyclic minimal models associated to finite $R$-algebras}

We return to the setting of \S\ref{ssec:KozDu} where $A$ is an algebra over a commutative Noetherian $\C$-algebra $R$, which is smooth over $\C$. We let $\Nccr = A\otimes_R \widehat R$ denote the completion of $A$ at a choice of maximal ideal $\m\subset R$, and let $E$ denote the Koszul dual of $A$ in $\m$.

If the completion is isomorphic to a completed Jacobi algebra of a quiver with potential $(Q,W)$, then the following theorem of Van den Bergh relates the Koszul dual to the $A_\infty$-category of $(Q,W)$. 

\begin{theorem}[{See \cite[Theorem 12.1]{VandenBergh15}}]\label{lem:minmodKoz}
  Suppose the completion $\Nccr$ is isomorphic to $\widehat \Jac(Q,W)$ for some quiver with potential $(Q,W)$.
  Then $\cD_{Q,W}$ is $A_\infty$-quasi-isomorphic to the DG algebra $E$.
\end{theorem}

If $A$ satisfies the conditions of the theorem we then obtain the following chain of quasi-equivalences
\[
  U \colon \cH
  \xrightarrow{\sim \text{q.e}} \tw\cD_{Q,W}
  \xrightarrow{\sim \text{q.e}} \DGPerf\cD_{Q,W}
  \xrightarrow{\sim \text{q.e}} \DGPerf E 
  \xrightarrow{\sim \text{q.e}} \cN.
\]
where $\cH\colonequals \H^\bullet\kern-2pt\tw\cD_{Q,W}$ is the cyclic minimal model of $\tw\cD_{Q,W}$ and $\cN \subset \cA = \DGPerf A$ is the DG-subcategory of objects supported on the maximal ideal $\m\subset R$ as in \S\ref{ssec:KozDu}. Via the equivalence $U$ we can relate the Hochschild actions of autoequivalences on $\cH$ and $\cN$, yielding the main theorem.

\begin{proof}[Proof of Theorem \ref{thm:CYstructPots}]
  Let $A$ be an algebra with a completion isomorphic to $\widehat\Jac(Q,W)$, and
  write $\cA = \DGPerf A$ as before.
  If $F\colon \cA \to \cA$ is an $R$-linear quasi-equivalence, such that $\HH_3(F) = \uplambda \in \C^\times$, then by Corollary \ref{cor:raction} it acts on $\HH^{-3}(\cN,\cN^*)$ as 
  \[
    \HH^{-3}(F) = \uplambda^{-1}.
  \]
  By \cite[Theorem 9.2.0.4]{LefevreHasegawa03}, the $A_\infty$-functor $U\colon \cH \to \cN$ has a quasi-inverse $U^{-1}\colon \cN \to \cH$, and one can hence lift $F$ to a quasi-auto-equivalence $F'\colonequals U^{-1}\circ F \circ U$ on $\cH$, which acts on the Hochschild cohomology $\HH^{-3}(\cH,\cH^*)$ as  
  \[
    \HH^{-3}(F') = \HH^{-3}(U^{-1}) \circ \uplambda^{-1} \circ \HH^{-3}(U) = \uplambda^{-1}.
  \]
  This shows that the functor $F'$ satisfies the homotopy-cyclic condition 
  \[
    \HH^{-3}(F')([\uplambda\cdot\upsigma]) = [\upsigma],
  \]
  with respect to the cyclic structures $\upsigma$ and $\uplambda\cdot\upsigma$ on $\cH$.
  Let $T\in\Ob \cH$ be a twisted complex, then Lemma \ref{lem:hocycfun} 
  shows that there exists a cyclic $A_\infty$-algebra isomorphism
  \begin{equation}\label{eq:cychomscaled}
    (\cH_T,\upsigma|_T) \to (\cH_{F(T)},\uplambda\cdot\upsigma|_{F'(T)}).
  \end{equation}
  Now suppose $M,N \in \nilp A$ are modules with $\End_A(M) \simeq \End_A(N) \simeq \C$ such that there exists a quasi-isomorphism $F(M) \simeq N$ in the derived category $\D_\m(A) \simeq \H^0\cN$. Then they can be represented by the twisted complexes $T_1,T_2 \in \Ob \cH$ such that $U(T_1) \simeq M$ and $U(T_2) \simeq N$. Because $F(M) \simeq N$, it then it also follows that $F'(T_1) \simeq T_2$ in $\H^0\cH$. Combining the map \eqref{eq:cychomscaled} with Lemma \ref{lem:cyccanon}, we obtain a cyclic $A_\infty$-isomorphism
  \[
    (\cH_{T_2},\uplambda\cdot\upsigma|_{T_2}) \xrightarrow{\ \sim\ } (\cH_{F(T_1)},\uplambda\cdot\upsigma|_{F'(T_1)})\xrightarrow{\ \sim\ } (\cH_{T_1},\upsigma|_{T_1}).
  \]
  In particular, there is an isomorphism $\uppsi\colon \widehat{\C\cQ}_{\min,T_1} \to \widehat{\C\cQ}_{\min,T_2}$ of the completed path algebras which maps $\cW_M = \cW_{\min,T_1}$ to the potential
  \[
    \begin{aligned}
      \uppsi(\cW_M)(f_0,\ldots,f_i) &= \sum_{i=2}^\infty (\uplambda\cdot \upsigma|_{T_2})(f_0)(m_i(f_1,\ldots,f_i)) \\
      &= \uplambda \cdot \sum_{i=2}^\infty \upsigma|_{T_2}(f_0)(m_i(f_1,\ldots,f_i)) \\
      &= \uplambda\cdot \cW_{\min,T_2}(f_0,\ldots,f_i).
    \end{aligned}
  \]
  Hence $\uppsi(\cW_M) = \uplambda\cdot \cW_{\min,T_2} = \uplambda\cdot \cW_N$ as claimed.
\end{proof}

With Theorem \ref{thm:CYstructPots} established, the proof of the corollary now follows almost directly from Theorem \ref{thm:ben} and Lemma \ref{lem:twequivmotives}.

\begin{proof}[Proof of Corollary \ref{cor:equivBPS}]
  By assumption $M \in \nilp A$ and $F(M) \in \nilp A$ are modules with $\End_A(M) \simeq \End_A(F(M)) \simeq \C$, so Theorem \ref{thm:ben} implies that 
  \[
    \begin{aligned}
      \sum_{k\geq 0}\int_{\cP_{M,k}}\kern-5pt \upphi_{\tr(W)}\cdot t^{[M]} &= \Upphi_{\cQ_{T_1},\cW_{T_1}}(t^{[M]}),\\
      \sum_{\k\geq 0}\int_{\cP_{F(M),k}}\kern-5pt \upphi_{\tr(W)}\cdot t^{[F(M)]} &= \Upphi_{\cQ_{T_2},\cW_{T_2}}(t^{[F(M)]}),
    \end{aligned}
  \]
  for some twisted complexes $T_1,T_2\in\tw \cD_{Q,W}$ corresponding to $M$ and $F(M)$ respectively. Theorem \ref{thm:CYstructPots} shows that there exists a formal isomorphism between the completed path algebras of $\cQ_{\min,T_1}$ and $\cQ_{\min,T_2}$ which maps $\cW_{\min,T_2}$ to $\uplambda\cdot \cW_{\min, T_1}$ for some scalar $\uplambda\in\C^\times$. Hence, Lemma \ref{lem:twequivmotives} show that
  \[
    \sum_{k\geq 0} \int_{\cP_{M,k}}\kern-5pt \upphi_{\tr(W)}\cdot t^{[M]} 
    = \Upphi_{\cQ_{T_1},\cW_{T_1}}(t^{[M]})
    = \Upphi_{\cQ_{T_2},\cW_{T_2}}(t^{[M]})
    = \sum_{k\geq 0} \int_{\cP_{F(M),k}}\kern-5pt \upphi_{\tr(W)}\cdot t^{[M]},
  \]
  and the result follows after comparing coefficients.
\end{proof}

\subsection{The geometric setting}

We return to the setting of threefolds. Let $Y$ be a smooth quasi-projective threefold, then the bounded complexes of locally free sheaves form a DG-category $\DGPerf Y$, whose Hochschild homology has a geometric interpretation.

\begin{lemma}\label{lem:smoothvarHH}
  The DG-category $\DGPerf Y$ is a smooth and has Hochschild homology 
  \[
    \HH_3(\DGPerf Y) \simeq \H^0(Y,\omega_Y).
  \]
\end{lemma}
\begin{proof}
  The smoothness of $\DGPerf Y$ for a smooth quasi-projective variety is well known, see e.g. the work of Orlov \cite{Orlov16} and Lunts \cite{Lunts10}. 
  It was moreover shown by Keller \cite{Keller98} that the Hochschild homology $\HH_\bullet(\DGPerf Y)$ of the DG category $\DGPerf Y$ coincides with the geometric Hochschild homology $\HH_\bullet(Y)$. 
  Because $Y$ is smooth, the Hochschild-Kostant-Rosenberg theorem (see e.g. \cite[Theorem 3.4.4]{Loday92}) yields a decomposition 
  \[
    \HH_d(Y) \simeq \bigoplus_{j-i=d} \H^i(Y,\Omega^j_Y) 
  \] 
  where $\Omega^j_Y$ denotes the sheaf of differential $j$-forms on $Y$. 
  Because $Y$ is a threefold, it then follows that $\HH_3(Y) \simeq \H^0(Y,\Omega^3_Y) = \H^0(Y,\omega_Y)$.
\end{proof}

Now let $R$ be a threedimensional Gorenstein ring, and suppose $Y$ is equipped with a proper morphism $\uppi\colon Y \to \Spec R$ such that $\bR \uppi_*\O_Y = \O_{\Spec R}$. Then the Hochschild homology reduces to the Hochschild homology of $R$, and we find the following.

\begin{lemma}\label{lem:HHaut}
  Let $\uppi\colon Y\to \Spec R$ be as above, then
  \[
    \Aut_R(\HH_3(\DGPerf Y)) \simeq R^\times.
  \]
\end{lemma}
\begin{proof}
  By the Gorenstein assumption $\Spec R$ is equipped with a dualising sheaf $\omega_R$. Because $\uppi$ is proper, the functor $\bR\uppi_*$ has the right adjoint $\uppi^!$ which maps the dualising sheaf $\omega_R$ on $\Spec R$ to $\uppi^!\omega_R \simeq \omega_Y$.
  Together with the assumption $\bR\uppi_*\O_Y \simeq \O_{\Spec R}$ this then implies that
  \[
    \begin{aligned}
      \H^0(Y,\omega_Y) 
      &\simeq \H^0\kern-1pt\RHom_Y(\O_Y,\uppi^!\omega_R) 
      \\&\simeq \H^0\kern-1pt\RHom_{\Spec R}(\bR\uppi_*\O_Y,\omega_R)
      \\&\simeq \H^0\kern-1pt\RHom_{\Spec R}(\O_R,\omega_R)
      \\&\simeq \H^0(\Spec R,\omega_R)
    \end{aligned}
  \]
  The canonical $\omega_R$ is moreover a line bundle, so its $R$-module of endomorphisms is free of rank 1.
  Hence, Lemma \ref{lem:smoothvarHH} implies that
  \[
    \End_R(\HH_3(\DGPerf Y)) \simeq \End_R(\H^0(Y,\omega_Y)) = \End_{\Spec R}(\omega_R) \simeq R,
  \]
  and the automorphisms are the invertible elements $R^\times\subset R$.
\end{proof}

If the group of units is equal to $\C^\times \subset R^\times$ then the above lemma implies that any $R$-linear DG autoequivalence of $\DGPerf Y$ acts by a scalar on the Hochschild homology, as in the condition of Theorem \ref{thm:CYstructPots}.
If $Y$ is derived equivalent to a Jacobi algebra, this then implies the following.

\begin{prop}\label{prop:everyautoequiv}
  Let $A = \Jac(Q,W)$ be a Jacobi algebra which is isomorphic to $\End_Y(\cP)$ for a tilting bundle $\cP$ on a smooth quasi-projective threefold $Y$ equipped with a proper map $\uppi\colon Y\to \Spec R$ satisfying $\bR\uppi_*\O_Y = \O_{\Spec R}$ to a Gorenstein scheme $\Spec R$ with units $R^\times = \C^\times$.
  Then the statement of Theorem \ref{thm:CYstructPots} holds for any $R$-linear standard equivalence $F\colon \D^b(A) \to \D^b(A)$.
\end{prop}
\begin{proof}
  Because $\cP$ is a tilting bundle, it induces quasi-inverse DG functors
  \[
    -\otimes_A \cP \colon \DGPerf A \to \DGPerf Y,\quad 
    (\DGPerf Y)(\cP,-)\colon \DGPerf Y \to \DGPerf A,
  \]
  which are moreover $R$-linear, via the canonical embedding of $R$ into $A = \End_Y(\cP)$ via its action on $\cP$ of multiplication by global sections in $\H^0(Y,\O_Y) \simeq \H^0(\Spec R,\bR\uppi_*\O_Y) \simeq R$.
  In particular, the Hochschild homologies are $R$-linearly isomorphic, so that
  \[
    \Aut_R(\HH_3(\DGPerf A)) \simeq \Aut_R(\HH_3(\DGPerf Y)) \simeq R^\times = \C^\times,
  \]
  by Lemma \ref{lem:HHaut} and the assumption on the units.
  Hence, the condition $\HH_3(F) = \uplambda$ for $\uplambda\in\C^\times$ in Theorem \ref{thm:CYstructPots} is automatically satisfied.
\end{proof}

Returning to Setup \ref{setup:general}, we recover Proposition \ref{prop:BPSequiv} as a special case.

\begin{proof}[{Proof of Proposition \ref{prop:BPSequiv}}]
  Recall that we are working in Setup \ref{setup:general}, where $(Q,W)$ has Jacobi algebra $A=\Jac(Q,W)$ derived equivalent to a length $2$ flopping contraction $Y\to \Spec R$ via a tilting bundle $\cP$, and the completion $\Nccr = \widehat A$ is the endomorphism algebra of Van den Bergh's tilting bundle on the completion over the origin in $\Spec R$.
  Moreover, we are working with the assumption $R^\times = \C^\times$, so that $A$ satisfies the assumptions of Proposition \ref{prop:everyautoequiv}.
  
  By Theorem \ref{thm:stableobjs} the semistable objects of phase $\uptheta_{\curve,n}$ in $\nilp A \simeq \fdmod \Nccr$ are all isomorphic extensions of the unique stable module $M$ with class $\updelta_{\curve,n}$.
  Therefore, the substack $\fC^{\uptheta_{\curve,n}} \into \fC$ is isomorphic to the substack $\fP_M \into \fC$ of self-extensions of $M$, which implies
    \[
      \Upphi^{\uptheta_{\curve,n}}(t) = \int_{[\fC^{\uptheta_{\curve,n}}\into \fC]} \upphi_{\tr(W)}|_{\fC} = \int_{[\fP_M \into \fC]} \upphi_{\tr(W)}|_{\fC} = \sum_{k>0} \int_{\fP_{M,k}} \upphi_{\tr(W)} t^{k\updelta_{\curve,n}}.
    \]
    Because $M$ is isomorphic to $\Uppsi(\O_\curve(n-1)[m])$ for $m=0,1$ depending on the sign of $n$, it is the image $F(S_1)$ of the simple $S_1 = \Uppsi(\O_\curve(-1))$ via the $R$-linear standard equivalence
    \[
      F\colon \D^b(\mod A) \xrightarrow{\Uppsi^{-1}} \D^b(\Coh Y) \xrightarrow{-\otimes \O_Y(n)[m]} \D^b(\Coh Y) \xrightarrow{\Uppsi} \D^b(\mod A).
    \]
    Hence, it follows from Proposition \ref{prop:everyautoequiv} that $\HH_3(F)$ is given by scalar multiplication, and Corollary \ref{cor:equivBPS} applied to $\fP_{S_1,k} = \fC_{k[S_1]}$ therefore implies that
    \[
      \begin{aligned}
        \Upphi^{\uptheta_{\curve,n}}(t) 
        &= \sum_{k>0}\int_{\fP_{M,k}} \upphi_{\tr(W)} t^{k\updelta_{\curve,n}}
        \\&= \sum_{k>0}\int_{\fC_{k[S_1]}} \upphi_{\tr(W)} t^{k\updelta_{\curve,n}}
        \\&= \Sym\left(\sum_{k>0}\frac{\BPS_{k[S_1]}}{\L^\half-\L^\mhalf} t^{k\updelta_{\curve,n}}\right).
      \end{aligned}
    \]
    Comparing with the BPS ansatz then yields $\BPS_{k\updelta_{\curve,n}} = \BPS_{k[S_1]}$ for all $k>0$.
    The statement for the invariants $\BPS_{k\updelta_{2\curve,n}}$ follows by an analogous argument.
\end{proof}

\appendix

\section{Blowup calculation}\label{sec:blowups}
Here we prove Lemma \ref{lem:blowup1} and Lemma \ref{lem:blowup2} by constructing an embedded resolution over $U\subset \A^2$ of the divisor $Z\subset U$ defined by
\[
  Z \colonequals \{0 = \cW = x^2y - f(y) \}.
\]
In what follows we decompose the parameter $f$ as $f(y) = y^{k+1} \cdot u(y)$ for $k\geq 2$ such that the factor $u(y)$ is invertible on the neighbourhood $U$. 

We construct an embedded resolution via a sequence of blowups. Consider the blowup $\uppi\colon \Bl \A^2 \to \A^2$ of the origin, which is a gluing $\Bl \A^2 = \A^2 \cup \A^2$ of two affine charts, and write 
\[
  \uppi_x\colon \A^2 \to \A^2,\quad \uppi_x(x,y) = (xy,y),\quad
  \uppi_y\colon \A^2 \to \A^2,\quad \uppi_y(x,y) = (x,xy),
\]
for the restriction of $\uppi$ to these charts. Let $N = \floor{\tfrac k2}$, then blowing up $N$ times gives
a resolution with $N+1$ charts, on which the resolution restricts to the maps
\[
  \uppi_y,\quad \uppi_x\circ \uppi_y,\quad \uppi_x \circ \uppi_x \circ \uppi_y,\quad\ldots,\quad \uppi_x^{N-1}\circ \uppi_y,\quad \uppi_x^N.
\]
The pullback of $Z$ through the resolution is locally given by 
\[
  (\uppi_x^j \circ \uppi_y)^* Z = \left\{\ y^{2j+1}x^{2j+3}(1 - x^{k-2-2j}y^{k-2j} \cdot u(xy)) = 0\ \right\}.
\]
for $j<N$ on the first $N$ charts and on the remaining chart by the equation
\[
  (\uppi_x^N)^*Z = \left\{\ y^{2N+1}(x^2 - y^{k-2N}\cdot u(y)) = 0\ \right\}.
\]
Then the pullback is normal-crossing on the former $N$ charts.
\begin{lemma}\label{lem:Hpoly}
  The divisor $(\uppi_x^j \circ \uppi_y)^* Z$ has normal-crossing singularities when restricted to the
  pre-image of $U \subset \A^2$.
\end{lemma}
\begin{proof}
  The pullback of $Z$ is the sum of the following prime divisors with multiplicity
  \[
    (\uppi_x^j \circ \uppi_y)^* Z = (2j+1) \cdot \{y=0\} + (2j+3) \cdot \{x=0\} + \{1-x^{k-2-2j}y^{k-2j}\cdot u(xy) = 0\}.
  \]
  Each of the prime divisors appearing in this sum is smooth on $(\uppi_x^j\circ\uppi_y)^{-1}(U)$,
  so it suffices to check that their intersections are generated by a regular system of parameters.
  The only intersection to consider is the intersection of the axes $\{y=0\}$ and $\{x=0\}$ 
  in the origin. This is clearly normal-crossing because $x,y$ is a regular system of parameters for the equation $xy = 0$.
\end{proof}
\begin{lemma}
  If $k=2N$ then $(\uppi_x^N)^*Z$ is normal-crossing on $(\uppi_x^N)^{-1}(U)$.
\end{lemma}
\begin{proof}
  For $k=2N$, the pullback of $Z$ is following sum of divisors with multiplicity
  \[
    (\uppi_x^N)^*Z = 2N\cdot \{y=0\} + \{x^2 - u(y) = 0\}.
  \]
  Note that $x^2-u(y)$ is not necessarily irreducible, but nonetheless defines a smooth
  reduced curve in $(\uppi_X^N)^{-1}(U)$. It therefore suffices to show
  that the intersections of this curve with the $x$-axis are generated by a regular system of parameters.
  Let $c$ be one of the square roots of $u(0) \neq 0$, then the curve intersects the $x$-axis
  at the points $(c,0)$ and $(-c,0)$.
  The defining equation of the curve can be put into the form
  \[
    x^2 - u(y) = x_-x_+ - (u(y)-c^2).
  \]
  where $x_\pm \colonequals x \pm c$. Then $x_+$ is invertible at the point $(0,c)$ and
  \[
    y,\quad x_- x_+ - (u(y)-c^2)
  \]
  is a regular system of parameters for the equation $y(x_-x_+ - u(y) -c^2)$ in $\O_{(0,c)}$.
  It follows that $(\uppi_x^N)^*Z$ is normal crossing at $(0,c)$, and similarly it is normal crossing at $(0,-c)$.
\end{proof}
The proof of Lemma \ref{lem:blowup2} now follows easily from the previous two lemmas.
\begin{proof}[Proof of Lemma \ref{lem:blowup2}]
  The condition $a>b$ implies that $y^{2b+1}$ is the lowest term in $f(y)$, so that the divisor $Z$ is defined by the equation
  \[
    y(x^2 - y^{2b}\cdot u(y)),
  \]
  where $u(y)$ is invertible with a leading term that is odd. Hence, we set $N=b$, and define $h\colon X = \bigcup_{j=0}^N X_j \to U$ as the gluing of the $N+1$ charts 
  \[
    X_0 = \uppi_y^{-1}(U),\quad \ldots,\quad
    X_{N-1} = (\uppi_x^{N-1} \circ \uppi_y)^{-1}(U),\quad X_N = (\uppi_X^N)(U),
  \]
  as schemes over $U$ via the maps $\uppi_x^j\circ\uppi_y$ and $\uppi_x^N$.
  Then the previous two lemmas show that $h^*Z$ is a normal-crossing divisor, and it remains to show that $h^*Z$ is the sum of the prime divisors
  $L_1,\quad E_3,\quad \ldots,\quad E_{2N+1},\quad L_2$ with the stated multiplicities and intersections. 
  On the chart $X_0$ the divisor $h^*Z$ restricts to $\uppi_y^*Z$, which is a sum of three prime divisors
  \[
    L_1 = \{y = 0\},\quad E_3|_{X_0} = \{x=0\},\quad L_2|_{X_0} = \{1 - x^{k-2}y^k u(xy) = 0\}
  \]
  with multiplicities $1$, $3$ and $1$ respectively. The lines $L_1$ and $E_3|_{X_0}$ meet in the origin and do no intersect $L_2|_{X_0}$.
  On the charts $X_j$ for $j = 1,\ldots,N-1$ the divisor $h^*Z$ restricts to $(\uppi_x^j\circ \uppi_y^*Z)$, which is a sum of prime divisors
  \[
    E_{2j+1}|_{X_j} = \{y = 0\},\quad E_{2j+3}|_{X_j} = \{x=0\},\quad L_2|_{X_j} = \{1 - x^{k-2-2j}y^{k-2j} u(xy) = 0\}
  \]
  with multiplicities $2j+1$, $2j+3$ and $1$ respectively, with the former two intersecting in the origin.
  On the chart $X_N$ the divisor $h^*Z$ restricts to $(\uppi_x^N)^*Z$, which is a sum of two prime divisors
  \[
    E_{2N+1}|_{X_N} = \{y=0\},\quad L_2|_{X_N} = \{x^2 = u(y)\},
  \]
  with multiplicities $2N+1$ and $1$ respectively. By inspection, $E_3,\ldots,E_{2N+1}$ form a chain of intersecting rational curves meeting eachother in a single point. Likewise $L_1$ meets $E_3$ in a single point, while  $L_2$ meets $E_{2N+1}$ in two points, which are the distinct solutions of $x^2 = u(0)$.
\end{proof}
For the defining equation in \ref{lem:blowup1} the parameter $k = 2N+1 = 2a-1$ is odd, and $(\uppi_x^N)^*Z$ is not normal crossing. One needs to blowup twice more.
\begin{lemma}
  The following divisors are normal-crossing on the pre-images of $U$:
  \begin{align*}
    (\uppi_x^N\circ \uppi_y)^*Z &= \{\ y^{2N+2}(x^2y - u(y)) = 0 \ \}\\
    (\uppi_x^{N+1}\circ \uppi_y)^*Z &= \{\ y^{2N+1}x^{4N+4}(1 - y\cdot u(xy))=0\ \}\\
    (\uppi_x^{N+2})^*Z &= \{\ y^{4N+4}x^{2N+2}(x- u(xy))=0\ \}
  \end{align*}
\end{lemma}
\begin{proof}
  In all three cases the axes $\{y=0\}$ and $\{x=0\}$ are smooth and intersect only in the origin.
  By assumption the polynomial $u$ has a constant term, which implies the curves $x^2y = u(y)$, $1 = y u(xy)$, and $x=u(xy)$
  are smooth.
  The $(\uppi_x^N\circ \uppi_y)^*Z$ is therefore normal-crossing, because the intersection 
  \[
    \{y=0\} \cap \{x^2y-u(y) = 0\} = \varnothing.
  \]
  The radical of the defining equation for the second divisor is $xy(1-y\cdot u(xy))$. 
  The curve $\{1 = y\cdot u(xy)\}$ does not intersect the axis $\{y=0\}$ and intersects $\{x=0\}$ in the point $p = (0,1/u(0))$.
  The variable $y$ is invertible in the local ring $\O_p$, so 
  \[
    x,\quad y(y - 1/u(0)),
  \]
  is a regular system of parameters defining $xy(1-y\cdot u(xy)$ in $\O_p$. It follows that the second divisor is normal-crossing.
  The radical of the third defining equation is $yx(x-u(xy))$.
  The curve $\{x = u(xy)\}$ does not intersect the axis $\{x = 0\}$ and intersects $\{y=0\}$ in the point $p = (u(0),0)$.
  The intersection is again normal crossing, as $yx(x-u(xy))$ has the regular system of parameters
  \[
    y,\quad x(x-u(xy)),
  \]  
  because $u(0)\neq 0$ implies $x$ is invertible in $\O_p$.
\end{proof}
The proof of Lemma \ref{lem:blowup1} now follows analogously to the proof of Lemma \ref{lem:blowup2}.
\begin{proof}[Proof of Lemma \ref{lem:blowup1}]
  The divisor $Z$ is defined by the equation
  \[
    y(x^2 - y^{2a-1}\cdot u(y)),
  \]
  for $u(y)$ invertible on $U$. Set $N = a-1$ and define $h\colon X = \bigcup_{j=0}^{N+2} X_j \to U$ as the gluing of 
  the $N+3$ charts 
  \[
    X_0 = \uppi_y^{-1}(U),\quad \ldots,\quad
    X_{N+1} = (\uppi_x^{N-1} \circ \uppi_y)^{-1}(U),\quad X_{N+2} = (\uppi_X^{N+2})(U),
  \]
  as schemes over $U$ via the maps $\uppi_x^j\circ\uppi_y$ and $\uppi_x^{N+2}$.
  As in the proof of \ref{lem:blowup2} we obtain a curve $L_1$ of multiplicity $1$ in $X_0$ and a chain of 
  exceptional $\P^1$'s $E_3,\ldots,E_{2N+1}$ of multiplicities $3,\ldots,2N+1$ glued from the lines in the charts $X_0,\ldots, X_N$.
  The remaining terms are $E_{4N+4}$, which is glued from
  \[
    E_{4N+4}|_{X_{N+1}} = \{\ x^{4N+4} = 0\ \},\quad E_{4N+4}|_{X_{N+2}} = \{\ y^{4N+4}=0\ \},
  \]
  and has multiplicity $4N+4 = 4a$, the divisor $E_{2N+2}$, which is glued from 
  \[
    E_{2N+2}|_{X_N} = \{\ y^{4N+2} = 0\ \},\quad E_{2N+2}|_{X_{N+2}} = \{\ x^{2N+2}=0\ \},
  \]
  and has multiplicity $2N+2 = 2a$, and the curve $L_2$ which is given by the equation $x = u(xy)$ on the
  chart $X_{N+2}$. By inspection, $E_{4N+4}$ meets $L_2$ and $E_{2N+2}$ in separate points on the chart $X_{N+2}$ and meets $E_{2N+1}$ on the chart $X_{N+1}$.
  The components $L_2$ and $E_{2N+2}$ do not intersect any other divisor.
\end{proof}

\let\H\Hu
\printbibliography

\end{document}